\pdfoutput=1
\documentclass[a4paper,10pt]{amsart}
\usepackage{
 amsmath,
 amssymb,
 mathrsfs, 
 mathtools, 
 fullpage,
 enumerate,
 diagrams
 }
\usepackage[breaklinks,pagebackref]{hyperref}

\newcommand{\stbt}[4]{\left(\begin{smallmatrix}#1 & #2 \\ #3 & #4\end{smallmatrix}\right)}

\title{Rankin--Eisenstein classes for modular forms}

\author{Guido Kings}
\address[Kings]{Fakult\"at f\"ur Mathematik \\
Universit\"at Regensburg\\
93040 Regensburg\\
Germany}
\email{guido.kings@mathematik.uni-regensburg.de}

\author{David Loeffler}
\address[Loeffler]{Mathematics Institute\\
Zeeman Building, University of Warwick\\
Coventry CV4 7AL, UK}
\email{d.a.loeffler@warwick.ac.uk}
\urladdr{\url{http://orcid.org/0000-0001-9069-1877}}

\author{Sarah Livia Zerbes}
\address[Zerbes]{Department of Mathematics \\
University College London\\
Gower Street, London WC1E 6BT, UK}
\email{s.zerbes@ucl.ac.uk}
\urladdr{\url{http://orcid.org/0000-0001-8650-9622}}

\thanks{The authors' research was supported by the following grants: SFB 1085 ``Higher invariants'' (Kings); Royal Society University Research Fellowship ``$L$-functions and Iwasawa theory'' and NSF Grant No. 0932078 000 (Loeffler); EPSRC First Grant EP/J018716/1 and NSF Grant No. 0932078 000 (Zerbes).}

\theoremstyle{plain}
    \newtheorem{theorem}{Theorem}[subsection]
\newtheorem{theorem*}{Theorem}
    \newtheorem{lemma}[theorem]{Lemma}
    \newtheorem{proposition}[theorem]{Proposition}
    \newtheorem{corollary}[theorem]{Corollary}
    \newtheorem{definition}[theorem]{Definition}
    \newtheorem{conjecture}[theorem]{Conjecture}
    
\theoremstyle{remark}
    \newtheorem{remark}[theorem]{Remark}
    \newtheorem{note}[theorem]{Note}
    \newtheorem*{note*}{Note}

\DeclareMathOperator{\TSym}{TSym}
\DeclareMathOperator{\Sym}{Sym}
\DeclareMathOperator{\Hom}{Hom}
\DeclareMathOperator{\Ext}{Ext}
\DeclareMathOperator{\Spec}{Spec}

\DeclareMathOperator{\Fil}{Fil}
\DeclareMathOperator{\comp}{comp}
\DeclareMathOperator{\GL}{GL}
\DeclareMathOperator{\SL}{SL}
\DeclareMathOperator{\Gr}{Gr}
\DeclareMathOperator{\pr}{pr}

\DeclareMathOperator{\Gal}{Gal}
\DeclareMathOperator{\coker}{coker}
\DeclareMathOperator{\AJ}{AJ}

\newcommand{\Eis}{\mathrm{Eis}}
\newcommand{\Cusp}{\mathrm{Cusp}}
\newcommand{\dR}{\mathrm{dR}}
\newcommand{\rig}{\mathrm{rig}}
\newcommand{\an}{\mathrm{an}}
\newcommand{\cris}{\mathrm{cris}}
\newcommand{\et}{\text{\textup{\'et}}}
\newcommand{\mot}{\mathrm{mot}}
\newcommand{\ord}{\mathrm{ord}}
\newcommand{\syn}{\mathrm{syn}}
\newcommand{\fp}{\mathrm{fp}}
\newcommand{\can}{\mathrm{can}}
\newcommand{\res}{\mathrm{res}}
\newcommand{\AH}{\cH}

\newcommand{\sC}{\mathscr{C}}
\newcommand{\sE}{\mathscr{E}}
\newcommand{\sF}{\mathscr{F}}
\newcommand{\sH}{\mathscr{H}}
\newcommand{\sS}{\mathscr{S}}

\newcommand{\sY}{\mathscr{Y}}

\newcommand{\cD}{\mathcal{D}}
\newcommand{\cE}{\mathcal{E}}

\newcommand{\cH}{\mathcal{H}}

\newcommand{\cN}{\mathcal{N}}
\newcommand{\cO}{\mathcal{O}}

\newcommand{\cR}{\mathcal{R}}
\newcommand{\cT}{\mathcal{T}}

\newcommand{\cX}{\mathcal{X}}
\newcommand{\cY}{\mathcal{Y}}

\newcommand{\frP}{\mathfrak{P}}
\newcommand{\frS}{\mathfrak{S}}
\newcommand{\frT}{\mathfrak{T}}

\newcommand{\CC}{\mathbf{C}}
\newcommand{\HH}{\mathbf{H}}
\newcommand{\QQ}{\mathbf{Q}}
\newcommand{\RR}{\mathbf{R}}
\newcommand{\ZZ}{\mathbf{Z}}
\newcommand{\DD}{\mathbf{D}}
\newcommand{\FF}{\mathbf{F}}

\newcommand{\Qp}{{\QQ_p}}
\newcommand{\Zp}{{\ZZ_p}}

\newcommand{\Lp}{L_{\frP}}

\newcommand{\isom}{\cong}
\newcommand{\sgn}{\mathrm{sgn}}
\newcommand{\id}{\mathrm{id}}
\newcommand{\into}{\hookrightarrow}
\newcommand{\Tate}{\mathrm{Tate}}

\numberwithin{equation}{subsection}

\begin{document}

\begin{abstract}
 In this paper we make a systematic study of certain motivic cohomology classes (``Rankin--Eisenstein classes'') attached to the Rankin--Selberg convolution of two modular forms of weight $\ge 2$. The main result is the computation of the $p$-adic syntomic regulators of these classes. As a consequence we prove many cases of the Perrin-Riou conjecture for Rankin--Selberg convolutions of cusp forms.
\end{abstract}

\maketitle


\setcounter{tocdepth}{1}

\tableofcontents

\section{Introduction}

 \subsection{Background} 

  In this paper we study certain motivic cohomology classes, which we call \emph{Rankin--Eisenstein classes}, associated to the values of Rankin--Selberg convolution $L$-functions of pairs of cuspidal modular eigenforms. 

  The first examples of Rankin--Eisenstein classes were introduced by Beilinson in \cite{Beilinson-L-values}. The classes considered by Beilinson live in the motivic cohomology group
  \[
   H^3_\mot\left(Y_1(N)\times Y_1(N),\QQ(2)\right),
  \]
  where $Y_1(N)$ is the modular curve classifying elliptic curves with a point of exact order $N$, and are obtained by push-forward of modular units along the diagonal. Beilinson mainly considered the regulator into Deligne cohomology
  \[
   r_\cD:H^3_\mot\left(Y_1(N)\times Y_1(N),\QQ(2)\right)
   \to H^3_\cD\left(Y_1(N)_\RR\times Y_1(N)_\RR,\RR(2)\right).
  \]
  and was able to compute explicitly the image of Rankin--Eisenstein classes under $r_{\cD}$ in terms of the values of Rankin--Selberg $L$-functions $L(f, g, s)$, for $f$ and $g$ weight 2 eigenforms of level $N$. 
  
  Beilinson's construction of the Rankin--Eisenstein classes, and his computation of their images under the Deligne regulator, were generalized to higher-weight modular forms by Scholl (unpublished, but see \cite{Kings-Higher-Regulators} for similar results in the case of Hilbert modular surfaces.) Using the geometry of Kuga--Sato varieties over $Y_1(N)$, Scholl defined motivic cohomology classes whose Deligne regulators were related to the Rankin--Selberg $L$-functions of pairs of modular forms of any weights $\ge 2$.

  Extending Beilinson's ideas in a different direction, Flach \cite{Flach-finiteness} studied the image under the \'etale regulator
  \[
   r_\et:H^3_\mot\left(Y_1(N)\times Y_1(N),\QQ(2)\right)
   \to H^3_\et\left(Y_1(N)\times Y_1(N),\Qp(2)\right)
  \]
  of certain motivic classes closely related to Rankin--Eisenstein classes, and used these to obtain finiteness results for the Tate--Shafarevich group of the symmetric square of an elliptic curve.

  In recent years the programme of Bertolini--Darmon--Rotger on the systematic study of Rankin--Selberg convolutions in $p$-adic families \cite{BDR-BeilinsonFlach, BDR-BeilinsonFlach2} has given a new impulse to the study of Rankin--Eisenstein classes (which appear in their work under the alternative name of \emph{Beilinson--Flach classes}). Roughly, the work of Bertolini--Darmon--Rotger is concerned with the \emph{syntomic regulator}
  \[
   r_\syn: H^3_\mot\left(Y_1(N)\times Y_1(N),\QQ(2)\right)
   \to H^3_\syn\left(Y_1(N)_\Zp\times Y_1(N)_\Zp,\Qp(2)\right)
  \]
  which is a $p$-adic analogue of $r_{\cD}$. The main result of \cite{BDR-BeilinsonFlach} relates the images of weight 2 Beilinson--Flach elements under $r_\syn$ to the values of $p$-adic $L$-functions (for good ordinary primes $p$), a $p$-adic version of Beilinson's computation of the Deligne regulators of these elements.

 \subsection{Outline of the paper}
 
  The main aims of this paper are as follows. Firstly, we give a careful account of the construction of the Rankin--Eisenstein classes of general weights, and the evaluation of their images under the Deligne regulator $r_{\cD}$. More precisely, for any integers $k, k', j$ such that $0 \le j \le \min(k, k')$, and any level $N \ge 4$, we define a class
  \[ 
   \Eis^{[k, k', j]}_{\mot, 1, N}
   \in H^3_{\mot}\left(Y_1(N)^2, \TSym^{[k, k']}(\sH_\QQ)(2-j)\right).
  \]
  (See \S \ref{sect:lieberman} for the definition of the motivic cohomology group appearing on the right-hand side, and Defintion \ref{def:Rankin-Eisenstein-class} for the definition of the element $\Eis^{[k, k', j]}_{\mot, 1, N}$.) 
  
  If $f, g$ are eigenforms of level $N$ and weights $k + 2, k' + 2$, then the condition $0 \le j \le \min(k, k')$ forces the $L$-function $L(f, g, s)$ to vanish to order 1 at $s = 1+ j$, due to the $\Gamma$-factors appearing in its functional equation. In Theorem \ref{thm:deligne-regulator-formulae}, we relate the first derivative $L'(f, g, 1 + j)$ to the image of $\Eis^{[k, k', j]}_{\mot, 1, N}$ under the Deligne regulator. These results are not new -- as mentioned above, they appear in unpublished notes of Scholl -- but they were not available in published form prior to our work.
  
  \begin{note*} 
   Since the first version of this paper was posted on the Arxiv, an independent proof of the same theorem has also been given by Brunault and Chida \cite{brunaultchida16}. Their paper also shows that the Rankin--Eisenstein classes extend to the cohomology of a product of compactified modular curves, an issue which we do not address in the present paper.
  \end{note*}
  
  Secondly, we give a higher-weight generalisation of the syntomic regulator computations of \cite{BDR-BeilinsonFlach}, relating the $p$-adic regulators of Rankin--Eisenstein classes to the values of $p$-adic $L$-functions, for all good ordinary primes $p$. This result, Theorem \ref{thm:syntomicreg}, is the main new result presented in this paper. In order to adapt the strategy of \cite{BDR-BeilinsonFlach} to this setting, we need to develop a new tool in arithmetic geometry, a generalisation of the ``finite-polynomial cohomology''' of Besser \cite{Besser-Coleman-integration} to allow coefficients in an overconvergent filtered F-isocrystal. We anticipate that this cohomology theory may be useful for computing $p$-adic regulators in many other settings (for instance, see \cite{LSZ} for an application to Hilbert modular surfaces).
  
  Thirdly, we use our computations of syntomic and Deligne regulators to verify an instance of a very general conjecture of Perrin-Riou, which roughly says that the leading terms of the complex $L$-functions at all integers $n$ (divided by a suitable ``period'') can be $p$-adically interpolated by a $p$-adic $L$-function. We verify this for integers in a certain range: we show in Theorem \ref{thm:Perrin-Riou} that the leading terms at $s = 1+j$, for all $0 \le j \le \min(k, k')$, can be $p$-adically interpolated by Hida's $p$-adic Rankin $L$-function. (This inequality on $j$ singles out precisely those values where the complex $L$-function vanishes to the first order. For $s \le 0$ the corresponding motivic cohomology is expected to be two-dimensional and in that case we even do not know Beilinson's conjecture.) We would like to mention that so far there are only very few cases where the conjecture of Perrin-Riou has been proven; besides the case of Dirichlet $L$-functions (essentially treated by Perrin-Riou), only the cases of modular forms \cite{Niklas-thesis} and certain elliptic curves with complex multiplication \cite{Bannai-Kings-CM} are known.
  
 \subsection{Relation to other works}

  This paper is the first in a series, of which the next instalment is \cite{KLZ1b}. As it serves as one of the major motivations for the present work, we shall briefly describe what will be carried out in the next paper. 
  
  The first main result of \cite{KLZ1b} is to construct three-parameter $p$-adic families of cohomology classes, interpolating the Rankin--Eisenstein classes constructed in the present paper for all values of the parameters $k, k', j$. This relies crucially on the first author's results on $p$-adic interpolation of Beilinson's Eisenstein symbol \cite{Kings-Eisenstein}. Having constructed these families, we then use Theorem \ref{thm:syntomicreg} of the present paper, together with a $p$-adic interpolation argument, to prove an ``explicit reciprocity law'' relating certain specialisations of this $p$-adic family to \emph{critical} values of complex $L$-functions. This is used in \emph{op.cit.} to prove many new cases of the Bloch--Kato conjecture for Rankin convolutions. 

  The beautiful idea of proving a relation to $L$-values in a non-critical range, and then using $p$-adic analytic continuation to pass from this range to the range of critical values, originates in the work of Bertolini, Darmon and Rotger. They used this strategy in \cite{BDR-BeilinsonFlach2}, to prove an explicit reciprocity law for a 1-parameter family of Rankin--Eisenstein classes; and our strategy is to a large extent inspired by their work. The chief difference between their approach and ours is in the input to the $p$-adic interpolation argument. The input we use is Theorem \ref{thm:syntomicreg} below, which holds for modular forms of level prime to $p$ and arbitrary weights; whereas they start from an analogous formula in which the modular forms are taken to have weight 2, but possibly high $p$-power levels. 

  Hence our approach has the advantage of giving extra information information about the \'etale and syntomic regulators of these higher-weight motivic classes, leading to a proof of Perrin-Riou's conjecture relating classical and $p$-adic $L$-values (Theorem \ref{thm:Perrin-Riou} of this paper). Moreover, the use of higher-weight modular forms means that our explicit reciprocity law can be generalised to non-ordinary Coleman families, as we shall show in the third paper in this series \cite{loefflerzerbes16}.

 \subsection{Acknowledgements}

  The authors are very grateful to Massimo Bertolini, Henri Darmon, and Victor Rotger for many inspiring discussions, in which they shared with us their beautiful ideas about Beilinson--Flach elements. The authors also would like to thank the organizers of the Banff workshop ``Applications of Iwasawa Algebras'' in March 2013, at which the collaboration was initiated which led to this paper. Part of this paper was written while the second and third author were visiting MSRI in Autumn 2014; they would like to thank MSRI for the hospitality. Finally, we thank the referee for several valuable comments and corrections.

\section{Geometrical preliminaries}

 In this section, we recall a number of cohomology theories attached to algebraic varieties (or more general schemes) over various base rings, and the relationships between these. The results of this section are mostly standard, but one aspect is new: we outline in \S \ref{sect:fpdefs} how to extend Besser's finite-polynomial cohomology to cover general coefficient sheaves.

 \subsection{Cohomology theories}

  We begin by introducing several cohomology theories associated to algebraic varieties (or, more generally, schemes). The most fundamental of these is motivic cohomology:

  \begin{definition}[{Beilinson, \cite{Beilinson-L-values}}]
   If $X$ is a regular scheme, we define \emph{motivic cohomology} groups
   \[ H^i_\mot(X,\QQ(n)) \coloneqq \Gr^\gamma_n K_{2n-i}(X)\otimes\QQ,\]
   the $n$-th graded piece of the $\gamma$-filtration of the $(2n-i)$-th algebraic $K$-theory of $X$.
  \end{definition}

  \begin{remark}
   This definition is compatible with the definition due to Voevodsky used in \cite{LLZ14}. Voevodsky's motivic cohomology can be defined with $\ZZ$-coefficients, but in this paper we shall only consider motivic cohomology with $\QQ$-coefficients (as we will need to decompose the motivic cohomology groups into eigenspaces for the action of a finite group), so the older definition via higher $K$-theory suffices.
  \end{remark}

  We will also need to work with several other cohomology theories. In each of these there is an appropriate notion of a \emph{coefficient sheaf}:

  \begin{itemize}
   \item \'etale cohomology, with coefficients in lisse \'etale $\Qp$-sheaves (for schemes on which the prime $p$ is invertible);
   \item algebraic de Rham cohomology (for smooth varieties over fields of characteristic 0), with coefficients in vector bundles equipped with a filtration and an integrable connection $\nabla$;
   \item Betti cohomology (for smooth varieties over $\CC$), with coefficients in $\QQ$ or more generally locally constant sheaves of $\QQ$-vector spaces;
   \item Absolute Hodge cohomology (for smooth varieties over $\RR$ or $\CC$), with coefficients in variations of mixed $\RR$-Hodge structures;
   \item rigid cohomology for smooth $\Zp$-schemes, with coefficients in overconvergent F-isocrystals;
   \item rigid syntomic cohomology for smooth $\Zp$-schemes, with coefficients in the category of ``admissible overconvergent filtered F-isocrystals'' defined in \cite{Bannai-syntomic, BannaiKings}. (The definition of this coefficient category and the associated cohomology theory in fact depends not only on the choice of a $\Zp$-scheme $X$, but also on a choice of a suitable smooth compactification $\bar X$; but we shall generally suppress this from the notation.)
  \end{itemize}

  We denote these theories by $H^\bullet_{\cT}(\dots)$, for $\cT \in \{ \et, \dR, B, \AH, \rig, \syn\}$. We sometimes write $\overline{\et}$ for \'etale cohomology over $\overline{\QQ}$. We write $\QQ_{\cT}$ for the trivial coefficient sheaf, and $\QQ_{\cT}(n)$ for the $n$-th power of the Tate object in the relevant category. (For Betti cohomology we take $\QQ_B = \QQ$, and $\QQ_B(n) = (2\pi i)^n \QQ \subseteq \CC$.)

  \begin{remark}
   We shall occasionally abuse notation slightly by writing $H^i_{\cT}(Y)$, for $Y$ a scheme over a ring such as $\ZZ[1/N]$; in this case, we understand this to signify the cohomology of the base-extension of $Y$ to a ring over which the cohomology theory $\cT$ makes sense. Thus, for instance, we write $H^i_{\syn}(Y)$ or $H^i_{\cD}(Y)$ for $Y$ a scheme over $\ZZ[1/N]$, and by this we intend $H^i_{\syn}(Y_{\Zp})$ and $H^i_{\cD}(Y_{\RR})$ respectively. This convention makes it significantly easier to state theorems which hold in all of the above theories simultaneously.
  \end{remark}

  For any of these theories, we can define higher direct images of coefficient sheaves for smooth proper morphisms, and we have a Leray spectral sequence. (For syntomic cohomology this was shown in the PhD thesis of N.~Solomon, \cite{Solomon}.)

  \begin{remark}
   All our coefficients are of geometric origin, and in fact the cohomology with coefficients $H^i_\cT(X, \sF_\cT)$ we use arises as a direct summand of an appropriate $H^i_\cT(Y, \QQ_\cT)$ by decomposing the Leray spectral sequence for some map $Y \to X$.
  \end{remark}

 \subsection{Comparison maps}

  The above cohomology theories are related by a number of natural maps.

  \subsubsection*{Regulator maps}

   Firstly, for each $\cT$ there is a ``regulator'' map
   \[ r_{\cT}: H^i_{\mot}(X, \QQ(n)) \to H^i_{\cT}(X, \QQ_{\cT}(n)).\]
   These maps are compatible with cup-products, pullbacks, and pushforward along proper maps.

   \begin{remark}
    For the compatibility of the syntomic regulator $r_{\syn}$ with pushforward maps, see \cite{deglisemazzari15}.
   \end{remark}

  \subsubsection*{Geometric comparison isomorphisms}

   We have the following well-known comparison isomorphisms. Firstly, if $X_{\CC}$ is a smooth variety over $\CC$ one has a comparison isomorphism
   \begin{equation}
    \label{eq:cplxcomparison}
    H^i_\dR(X_\CC, \CC) \cong H^i_B(X_{\CC}(\CC), \CC).
   \end{equation}
   compatible with $r_B, r_{\dR}$.

   If $X_{\CC} = X_{\RR} \times_{\RR} \CC$ for some variety $X_{\RR}$ over $\RR$, then both sides have
   an $\RR$-structure
   \[ H^i_\dR(X_\RR, \RR) \otimes\CC \cong H^i_B(X_{\RR}(\CC), \RR)\otimes\CC \]
   which are not respected by the comparison isomorphism,
   but $\id\otimes c$ on the
   left hand side (where $c:\CC\to\CC$ is complex conjugation) corresponds on $H^i_B(X_{\RR}(\CC), \RR)\otimes \CC$ to the map $\overline{F_\infty} = F_\infty \otimes c$, where $F_\infty$ denotes pullback via the complex conjugation automorphism $F_{\infty}$ of the topological space $X_{\RR}(\CC)$. As usual
   we define $H^i_B(X_{\RR}(\CC), \QQ(n))^\pm \coloneqq H^i_B(X_{\RR}(\CC), \QQ(n))^{\overline{F}_\infty=\pm 1}$ where
   $\QQ(n) \coloneqq (2\pi i)^n\QQ$.

   Similarly, let $X$ be a smooth $\Zp$-scheme, and suppose that we can embed $X$ in a smooth proper $\Zp$-scheme $\bar X$ as the complement of a simple normal crossing divisor relative to $\Spec \Zp$ (so that $\mathscr{X} = (X, \bar X)$ is a \emph{smooth pair} in the sense of \cite[Appendix A]{BannaiKings}). Let $\sF$ be any admissible overconvergent filtered $F$-isocrystal on $X$; then we can define de Rham and rigid realizations $\sF_{\dR}$ and $\sF_{\rig}$, and one has an isomorphism
   \begin{equation}
    \label{eq:rigidcomparison}
    H^i_{\dR}(X_{\Qp}, \sF_{\dR}) \cong H^i_{\rig}(X, \sF_{\rig}).
   \end{equation}

   If $X_{\Qp}$ is a smooth variety over $\Qp$ one has Faltings comparison isomorphism
   \begin{equation}
    \label{eq:dRcomparison}
    \comp_{\dR} : H^i_{\dR}(X_{\Qp}, \Qp) \otimes \mathbf{B}_{\dR} \cong H^i_{\et}(X_{\overline{\QQ}_p}, \Qp) \otimes \mathbf{B}_{\dR}
   \end{equation}
   where $\mathbf{B}_{\dR}$ is Fontaine's ring of periods; and this isomorphism is $\Gal(\overline{\QQ}_p / \Qp)$-equivariant, if we let the Galois group act trivially on $H^i_{\dR}(X_{\Qp} / \Qp)$ and via its native action on $H^i_{\et}(X_{\overline{\QQ}_p}, \Qp)$. This isomorphism is also compatible with the filtrations on both sides, and with the regulator maps $r_{\et}$, $r_{\dR}$ (for the variety $X_{\overline{\QQ}_p}$).

   If $X$ is itself proper, then rigid cohomology coincides with crystalline cohomology, and one has a comparison isomorphism refining \eqref{eq:dRcomparison},
   \begin{equation}
    \label{eq:criscomparison}
    \comp_{\mathrm{cris}} : H^i_{\rig}(X, \Qp) \otimes \mathbf{B}_{\mathrm{cris}} \cong H^i_{\et}(X_{\overline{\QQ}_p}, \Qp) \otimes \mathbf{B}_{\mathrm{cris}}
   \end{equation}
   compatible with $\Gal(\overline{\QQ}_p / \Qp)$ and with the Frobenius $\varphi$.

   \begin{remark}
    One can check that if $X$ is non-proper, but can be compactified to a smooth pair $\mathscr{X}$, then the \'etale cohomology of $X_{\Qp}$ is crystalline. Hence, combining equations \eqref{eq:rigidcomparison} and \eqref{eq:dRcomparison}, we have a canonical isomorphism of the form \eqref{eq:criscomparison}; but it is not clear if it commutes with the action of $\varphi$ when $X$ is non-proper.
   \end{remark}

   \begin{remark}
    \label{remark:nocoeffs}
    Note that we have not attempted to define a version of \eqref{eq:dRcomparison} or \eqref{eq:criscomparison} with coefficients, as it is not clear what the appropriate category of coefficient sheaves should be.
   \end{remark}

 \subsection{The Leray spectral sequence and its consequences}

  For any of the cohomology theories $\cT \in \{ \et, \bar\et, \dR, B, \AH, \rig, \syn\}$, and a variety $X$ with structure map $\pi: X \to \Spec R_{\cT}$ where $R_{\cT}$ is the appropriate base ring, we have a Leray spectral sequence
  \[ {}^\cT E_2^{ij} = H^i_{\cT}(\Spec R_{\cT}, R^j \pi_* \sF) \Rightarrow H^{i + j}_{\cT}(X, \sF).\]
  For the ``geometric'' theories $\cT \in \{ \bar\et, \dR, B, \rig \}$ this is not interesting, as the groups $H^i_{\cT}(\Spec R_{\cT}, -)$ are zero for $i \ne 0$. However, it is interesting for the ``absolute'' theories $\cT \in \{ \cD, \et, \syn \}$.

  \subsubsection*{Deligne cohomology and absolute Hodge cohomology}
  For a smooth variety $a:X_\RR\to\Spec \RR$
  one can define the Deligne-Beilinson cohomology
  groups $H^i_\cD(X_\RR,\RR(n))$  in terms of
  holomorphic differentials with logarithmic poles along
  a compactification. These groups are connected with
  de Rham and Betti cohomology via
  a long exact cohomology sequence
  \begin{equation}\label{eq:deligneexactseq}
  \ldots\to F^nH^i_\dR(X_\RR,\RR)\to H^{i}_B(X_\RR(\CC),\RR(n-1))^+\to H^{i+1}_\cD(X_\RR,\RR(n))\to F^nH^{i+1}_\dR(X_\RR,\RR)\to\ldots
  \end{equation}

  Recall the definition of absolute Hodge cohomology.
  Let $MHS_\RR^+$ be the category of mixed
 $\RR$-Hodge structures $M_\RR$ carrying an involution
 $F_\infty:M_\RR\to M_\RR$, which respects the weight filtration
 and such that $\overline{F}_\infty:M_\RR\otimes \CC\to
 M_\RR\otimes \CC$ respects the Hodge filtration.
 For any separated scheme $X_\RR\to \Spec \RR$ of finite type Beilinson \cite{Beilinson-Hodge}
 has defined a complex $R\Gamma(X_\RR,\RR(n))\in
 D^b(MHS_\RR^+)$ whose cohomology groups are the mixed
 Hodge structures
 $H^i(X_\RR(\CC),\RR(n))$ with the involution $\overline{F}_\infty$.  The absolute Hodge
cohomology of $X_\RR$ is by definition
\[
H^i_\AH(X_\RR, \RR(n)) \coloneqq R^i\Hom_{D^b(MHS_\RR^+)}(\RR(0),R\Gamma(X_\RR,\RR(n)))
\]
and one has  a short exact sequence
\begin{equation}\label{eq:absolute-Hodge-seq}
0\to \Ext^1_{MHS_\RR^+}(\RR(0),H^{i-1}_B(X_\RR(\CC),\RR(n)))
\to H^i_\AH(X_\RR, \RR(n))\to
\Hom_{MHS_\RR^+}(\RR(0),H^{i}_B(X_\RR(\CC),\RR(n)))\to 0.
\end{equation}
The computation of the Ext-groups of $M_\RR\in MHS_\RR^+$ is
standard and one has
\[ \Hom_{MHS_\RR^+}(\RR(0),M_\RR) = W_0M_{\RR}^+ \cap \Fil^0 M_{\CC},\quad \Ext^1_{MHS_\RR^+}(\RR(0), M_\RR) = \frac{W_0M_{\CC}^+}{W_0M_{\RR}^+ + \Fil^0 M_{\CC}^+}. \]
In the case where $a:X_\RR\to\Spec \RR$ is smooth
and the weights of $H^{i-1}_B(X_\RR,\RR(n))$
are $\le 0$ the absolute Hodge cohomology coincides
with the Deligne-Beilinson cohomology
$H^i_\cD(X_\RR,\RR(n))$. The advantage of
absolute Hodge cohomology is that one can define
this theory also with coefficients.

Let
$MHM_\RR(X_\RR)$ be the category of algebraic
$\RR$-mixed Hodge modules over $\RR$ of Saito
(this means a Hodge module over $X_\CC$ together
with an involution of $\overline{F}_\infty$, see
\cite[Appendix A]{HW} for more details). For
any $M_\RR\in MHM_\RR(X_\RR)$ one
defines
\[
H^i_\AH(X_\RR, M_\RR) \coloneqq R^i\Hom_{MHM_\RR(X_\RR)}(\RR(0),M_\RR)).
  \]
In the case where $M_\RR=\RR(n)$
one has the adjunction
\[
R\Hom_{MHM_\RR(X_\RR)}(\RR(0),\RR(n)))\isom
R\Hom_{D^b(MHS_\RR^+)}(\RR(0),Ra_*\RR(n)))
\]
and $Ra_*\RR(n)\isom R\Gamma(X_\RR,\RR(n))$.
This
interprets the above results in terms of the
Leray spectral sequence for $Ra_*$.

  \subsubsection*{Syntomic cohomology}

   The theory of syntomic cohomology, for smooth pairs $(X, \bar X)$ over $\Zp$, is closely parallel to that of absolute Hodge cohomology. An overconvergent filtered isocrystal on $\Spec \Zp$ is simply a filtered $\varphi$-module, in the sense of $p$-adic Hodge theory, and the cohomology groups $H^i_{\syn}(\Spec \Zp, D)$ are given by the cohomology of the 2-term complex $\Fil^0 D \rTo^{1-\varphi} D$; thus we have
   \begin{equation}
    \label{eq:hif}
    H^0_\syn(\Spec \Zp, D) = D^{\varphi = 1} \cap \Fil^0 D, \quad H^1_{\syn}(\Spec \Zp, D) = \frac{D}{(1 - \varphi) \Fil^0 D}.
   \end{equation}

   For a general smooth pair $(X, \bar X)$, and $\sF$ an admissible overconvergent filtered F-isocrystal on $X$, the comparison isomorphism $H^i_{\dR}(X_{\Qp}, \sF_{\dR}) \cong H^i_{\rig}(X, \sF_{\rig})$ allows us to interpret these spaces as a filtered $\varphi$-module, which we shall denote by $H^i_{\rig}(X, \sF)$; this is precisely the higher direct image $R^i \pi_* \sF$, where $\pi: X \to \Spec \Zp$ is the structure map. Exactly as in the complex case, the Leray spectral sequence becomes a long exact sequence
   \begin{equation}
    \label{eq:syntomicexactseq}
    \dots \to H^i_{\syn}(X, \sF) \to \Fil^0 H^i_{\dR}(X_{\Qp}, \sF_{\dR}) \rTo^{1 - \varphi} H^i_{\rig}(X, \sF_{\rig})
    \to \dots,
   \end{equation}
   which is the $p$-adic analogue of the long exact sequence \eqref{eq:deligneexactseq}.

  \subsubsection*{\'Etale cohomology}

   Let $X$ be a smooth variety over a field $K$ of characteristic 0, and $\sF$ a lisse \'etale sheaf on $X$. Then we have a Leray spectral sequence
   \[ {}^{\et} E_2^{ij} = H^i_{\et}(\Spec K, H^j(X_{\overline{K}}, \sF)) \Rightarrow H^{i+j}_{\et}(X, \sF).\]
   In general, this must be interpreted in terms of Jannsen's continuous \'etale cohomology \cite{jannsen88}; we shall only use this for $K = \Qp$, in which case continuous \'etale cohomology coincides with the usual \'etale cohomology.

   If $X_\QQ$ is a smooth variety over $\QQ$, then $X_{\QQ}$ admits a smooth model $X$ over $\ZZ[1/S]$ for some set of primes $S$ (with $p \in S$ without loss of generality); and if $\sF$ is of geometric origin, then it will extend to a lisse $\Qp$-sheaf on $X$ for large enough $S$. Then the cohomology groups $H^j(X_{\overline{\QQ}}, \sF)$ are unramified outside $S$, and this sequence becomes
   \[  H^i(\Gal(\QQ^S / \QQ), H^j(X_{\overline{\QQ}}, \sF)) \Rightarrow H^{i+j}_{\et}(X, \sF).\]


 \subsection{Compatibility of \'etale and syntomic cohomology}\label{section:etsyncompatible}

  Let $D$ be a filtered $\varphi$-module over $\Qp$. As noted above, we can regard $D$ as an overconvergent filtered $F$-isocrystal on $\Spec \Zp$, and the cohomology groups $H^i_{\syn}(\Spec \Zp, D)$ are given by the formulae \eqref{eq:hif}.

  If $D = \DD_{\mathrm{cris}}(V)$ for $V$ a crystalline $G_{\Qp}$-representation, then we have canonical maps
  \[ H^i_\syn(\Spec \Zp, D) \to H^i(\Qp, V)\]
  arising from the Bloch--Kato short exact sequence of $G_{\Qp}$-modules
  \begin{equation}
   \label{eq:fundamentalseq}
   0 \rTo V \rTo V \otimes \mathbf{B}_{\mathrm{cris}} \rTo (V \otimes \mathbf{B}_{\mathrm{cris}}) \oplus (V \otimes \mathbf{B}_{\dR}/\mathbf{B}_{\dR}^+) \rTo 0.
  \end{equation}
  The map $H^i_\syn(\Spec \Zp, D) \to H^i(\Qp, V)$ is an isomorphism for $i = 0$, and for $i = 1$ it is injective, with image the subspace $H^1_f(\Qp, V)$ parametrizing crystalline extensions of the trivial representation by $V$. The resulting isomorphism
  \[ \frac{D}{(1 - \varphi) \Fil^0 D} \rTo^\cong H^1_f(\Qp, V) \]
  is denoted by $\widetilde\exp_{\Qp, V}$ in \cite{LVZ}; it satisfies $\widetilde\exp_{\Qp, V} \circ (1 - \varphi) = \exp_{\Qp, V}$, where $\exp_{\Qp, V}$ is the Bloch--Kato exponential map.

  \begin{remark}
   If $D$ is a filtered $\varphi$-module over $\Qp$, then the space $D / (1 - \varphi)\Fil^0 D$ is easily seen to parametrize extensions (in the category of filtered $\varphi$-modules) of the trivial module by $D$, and the map $\widetilde\exp_{\Qp, V}$ is just the natural map $\Ext^1_{\varphi, \Fil}(\mathbf{1}, \DD_{\mathrm{cris}}(V)) \rTo \Ext^1_{G_{\Qp}}(\mathbf{1}, V)$.
  \end{remark}

  Then we have the following theorem:

  \begin{theorem}[Besser, Nizio\l]\mbox{~}
   \label{thm:absolutecomparison}
   \begin{enumerate}
    \item Suppose $\mathscr{X} = (X, \bar X)$ is a smooth pair over $\Zp$, with $\bar X$ projective. Then there is a natural map
    \[ \comp: H^i_{\syn}(X, \Qp(n)) \to H^i_{\et}(X_{\Qp}, \Qp(n)),\]
    for each $n$, fitting into a commutative diagram (functorial in $\mathscr{X}$) 
    \begin{equation}
     \label{eq:absolutecomparison}
     \begin{diagram}
      H^i_{\mot}(X, \QQ(n)) & \rTo^{r_\syn} & H^i_{\syn}(X, \Qp(n)) \\
      \dTo & & \dTo_{\comp} \\
      H^i_{\mot}(X_{\Qp}, \QQ(n))& \rTo^{r_{\et}} & H^i_{\et}(X_{\Qp}, \Qp(n)).
     \end{diagram}
    \end{equation}
    where the left vertical map is given by base extension.
    \item If $X$ is projective, there is a morphism of spectral sequences ${}^{\syn} E^{ij} \to {}^{\et} E^{ij}$ for each $n$, compatible with the morphism $\comp$ on the abutment; and on the $E_2$ page the morphisms
    \[ H^i_\syn(\Spec \Zp, H^j_{\rig}(X, \Qp(n))) \to H^i(\Qp, H^j_{\et}(X_{\overline{\QQ}_p}, \Qp(n))) \]
    are given by the exact sequence \eqref{eq:fundamentalseq} for $V = H^j_{\et}(X_{\overline{\QQ}_p}, \Qp(n))$, together with the Faltings comparison isomorphisms
    \( \comp_{\dR}: H^j_\rig(X, \Qp) \cong H^j_\dR(X, \Qp) \cong \DD_\dR(H^j_{\et}(X_{\overline{\QQ}_p}, \Qp)) \)
    of \eqref{eq:dRcomparison}.
   \end{enumerate}
  \end{theorem}

  \begin{proof}
   See \cite{Besser-rigid-syntomic}, Corollary 9.10 and Proposition 9.11.
  \end{proof}
 
  \begin{remark}
   Part (1) is stated in \emph{op.cit.} for general smooth quasiprojective $\Zp$-schemes $X$ (without the hypothesis that $X$ can be compactified to a smooth pair). This is incorrect, as the results of \cite{niziol97} quoted in \emph{op.cit.} only apply when $X$ is projective. However, for smooth pairs $(X, \bar X)$ one can define the comparison map in (1) by applying Nizio\l's comparison theorems to $\bar X$ considered as a log-scheme, with the log-structure associated to the divisor $\bar X - X$. We are grateful to Wieslawa Nizio\l\ for explaining this argument to us.
  \end{remark}
 
  We will use this in \S \ref{sect:ajdefs} below, to show that the Abel--Jacobi maps for syntomic and \'etale cohomology are related by the Bloch--Kato exponential map.

 \subsection{Finite-polynomial cohomology}
  \label{sect:fpdefs}

  In order to evaluate the syntomic regulators, we will make use of a family of cohomology theories defined by Besser \cite{Besser-Coleman-integration}, depending on a choice of polynomial $P \in 1 + T \Qp[T]$; these reduce to syntomic cohomology when $P(T) = 1 - T$. Besser's theory is described in \emph{op.cit.} for coefficient sheaves of the form $\Qp(n)$, and we briefly outline below how to extend this to more general coefficient sheaves.

  Let $X$ be a smooth $\Zp$-scheme, and $\sF$ an admissible overconvergent filtered $F$-isocrystal on $X$ (or, more precisely, on some smooth pair $(X, \bar X)$ compactifying $X$).

  \begin{definition}
   For a polynomial $P \in 1 + T \Qp[T]$, we define groups $H^i_{\fp}(X, \sF, P)$ by replacing $1 - \varphi$ with $P(\varphi)$ in the definition of rigid syntomic cohomology with coefficients (cf.~\cite[Appendix A]{BannaiKings}).

   We also define compactly-supported versions $H^i_{\fp,c}(X, \sF, P)$ similarly (cf.~\cite{Besser-K1-surface}).
  \end{definition}

  \begin{remark}
   When $\sF = \Qp(n)$ for some $n$, these groups reduce to Besser's original finite-polynomial cohomology (cf.~\cite{Besser-Coleman-integration}), but with a different numbering: in Besser's theory $H^i_{\syn}(X, \Qp(n))$ corresponds to taking $P(T) = 1 - T / p^n$, whereas if more general coefficients are allowed, it is more convenient to number in such a way that syntomic cohomology always corresponds to $P(T) = 1 - T$, whatever the value of $n$.
  \end{remark}

  Exactly as in the case of syntomic cohomology (see \eqref{eq:syntomicexactseq} above), we have a long exact sequence
  \[
   \dots \to H^i_{\fp}(X, \sF, P) \to \Fil^0 H^i_{\dR}(X_{\Qp}, \sF_{\dR}) \rTo^{P(\varphi)} H^i_{\rig}(X, \sF_{\rig}, P)
   \to \dots
  \]
  and similarly for the compactly-supported variant; and there are ``change-of-$P$'' maps fitting into a diagram
  \begin{diagram}
   \dots \to & H^i_{\fp}(X, \sF, P) & \rTo & \Fil^0 H^i_{\dR}(X_{\Qp}, \sF_{\dR}) & \rTo^{P(\varphi)} & H^i_{\rig}(X, \sF_{\rig}) & \to \dots\\
   & \dTo & & \dTo^{\mathrm{id}} && \dTo^{Q(\varphi)} \\
   \dots \to & H^i_{\fp}(X, \sF, PQ) & \rTo & \Fil^0 H^i_{\dR}(X_{\Qp}, \sF_{\dR}) & \rTo^{PQ(\varphi)} & H^i_{\rig}(X, \sF_{\rig}) & \to \dots\\
  \end{diagram}

  \begin{definition}
   \label{def:starproduct}
   Exactly as in the case of Tate-twist coefficients in \cite[\S 2]{Besser-K1-surface}, we define cup products
   \[ H^i_{\fp}(X, \sF, P)\times H^j_{\fp,c}(X, \mathscr{G},Q)\rTo^\cup H^{i+j}_{\fp,c}(X, \sF\otimes \mathscr{G}, P\star Q),\]
   where the polynomial $P\star Q$ is defined by the formula
   \[
    \left(\prod_i(1-\alpha_i T)\right)\star\left(\prod_j(1-\beta_j T)\right) = \prod_{i,j}(1-\alpha_i\beta_j T).
   \]
  \end{definition}

  These cup-products are compatible with the change-of-$P$ maps, in the obvious sense. They also satisfy a more subtle compatibility with the long exact sequence: the cup-product $H^i_{\fp}(X, \sF, P)\times H^j_{\fp,c}(X,\mathscr{G},Q)\rTo^\cup H^{i+j}_{\fp,c}(X,\sF \otimes \mathscr{G}, P\star Q)$ is compatible with the cup-products
  \[ H^u_{\fp}(\Spec \Zp, H^i_{\rig}(X \sF), P) \times H^v_{\fp}(\Spec \Zp, H^j_{\rig, c}(X, \mathscr{G}), Q) \to H^{u + v}_{\fp}(\Spec \Zp, H^{i + j}_{\rig, c}(X, \sF \otimes \mathscr{G}), P \star Q). \]

  If the polynomial $P$ satisfies $P(p^{-1}) \ne 0$, and $X$ is connected and has dimension $d$, then there is a canonical isomorphism
  \[
   \operatorname{tr}_{\fp, X}: H^{2d + 1}_{\fp, c}(X, \Qp(d + 1), P) \rTo^\cong \Qp
  \]
  given by $\frac{1}{P(p^{-1})} \operatorname{tr}_{\rig, X}$, where $\operatorname{tr}_{\rig, X}: H^{2d}_{\rig, c}(X, \Qp) \cong \Qp$ is the trace map for rigid cohomology; the inclusion of the factor $\frac{1}{P(p^{-1})}$ makes this map compatible with the change-of-$P$ maps. For polynomials $P, Q$ with $(P \star Q)(p^{-1}) \ne 0$, we thus have a pairing
  \[ \langle -, - \rangle_{\fp, X} : H^i_{\fp}(X, \sF, P) \times H^{2d + 1 - i}_{\fp, c}(X, \sF^\vee (d + 1), Q) \to \Qp \]
  given by composing the cup-product with the map $\operatorname{tr}_{\fp, X}$.


\section{Modular curves and modular forms}


 We now introduce the specific geometric objects to which we will apply the general theory of the previous section: the modular curves $Y_1(N)$, and various coefficient sheaves on these curves.

 \subsection{Symmetric tensors}

  If $H$ is an abelian group, we define the modules $\TSym^k H$, $k \ge 0$, of symmetric tensors with values in $H$ following \cite[\S 2.2]{Kings-Eisenstein}. By definition, $\TSym^k H$ is the submodule of $\frS_k$-invariant elements in the $k$-fold tensor product $H \otimes \dots \otimes H$ (while the more familiar $\Sym^k H$ is the module of $\frS_k$-coinvariants). If $H$ is free of finite rank, and $H^\vee$ is its dual, then there is a canonical isomorphism $\TSym^k(H)^\vee = \Sym^k (H^\vee)$.

  The direct sum $\TSym^\bullet H \coloneqq \bigoplus_{k \ge 0} \TSym^k H$ is equipped with a ring structure via symmetrization of the naive tensor product, so for $h \in H$ we write $h^{[r]}:=h^{\otimes r}$ and one has 
  \begin{equation}\label{TSymmult}
   h^{[m]} \cdot h^{[n]} = \frac{(m + n)!}{m! n!} h^{[m + n]}.
  \end{equation}

  Similarly, one can define $\TSym^k H$ for a module $H$ over any commutative ring $A$ (e.g.~for vector spaces over a field). There is a natural ring homomorphism $\Sym^\bullet H \to \TSym^\bullet H$ given by mapping $x^{k}$ to $k! x^{[k]}$, which is an isomorphism if $k!$ is invertible in $A$.
  
  In general $\TSym^k$ does not commute with base change, and hence does not sheafify well. In the cases where we consider $\TSym^k(H)$, $H$ is always a free module over the coefficient ring, so that this functor coincides with $\Gamma^k(H)$, the $k$-th divided power of $H$. This functor does sheafify (on an arbitrary site), so that the above definitions and constructions carry over to sheaves of abelian groups. Thus, for any of the cohomology theories $\cT \in \{B, \et,\dR, \syn, \rig,\cD\}$, and $\sF$ an object of the appropriate category of coefficient sheaves, we can define objects $\TSym^k \sF$. We use this to make the following key definition:

  \begin{definition}
   Let $\pi:\cE \to Y$ be an elliptic curve such that $\cE$ and $Y$ are regular. For $\cT \in \{B, \et,\dR, \syn, \rig,\cD\}$, we define an element of the appropriate category of coefficient sheaves on $Y$ by
   \[ \sH_{\cT} = (R^1 \pi_* \QQ_\cT)^\vee. \]
  \end{definition}

  We now suppose $Y$ is a $T$-scheme, for some base scheme $T$, and we let $\Delta$ be the diagonal embedding $Y \into Y^2 = Y \times_T Y$.
\begin{definition}\label{def:Sym-r-s}
We define sheaves on $Y^2$ by
  \begin{align*} \TSym^{[k, k']} \sH_{\cT}& \coloneqq \pi_1^* \left(\TSym^k \sH_{\cT}\right) \otimes \pi_2^* \left(\TSym^{k'} \sH_{\cT}\right)\\
\Sym^{(k, k')} \sH_{\cT}& \coloneqq \pi_1^* \left(\Sym^k \sH_{\cT}\right) \otimes \pi_2^* \left(\Sym^{k'} \sH_{\cT}\right)
,\end{align*}
  where $\pi_1$ and $\pi_2$ are the first and second projections $Y^2 \to Y$.
\end{definition}
Then the pullback of these sheaves along $\Delta$ are the sheaves on $Y$ \begin{align*}
\Delta^{*}\TSym^{[k, k']} \sH_{\cT}&\isom \TSym^{k} \sH_{\cT}\otimes
\TSym^{k'} \sH_{\cT}\\
\Delta^{*}\Sym^{(k, k')} \sH_{\cT}&\isom \Sym^{k} \sH_{\cT}\otimes
\Sym^{k'} \sH_{\cT}.
\end{align*}

  If $Y$ is smooth of relative dimension $d$ over $T$, then we obtain pushforward (Gysin) maps
  \begin{equation}\label{eq:gysin-map} \Delta_*: H^{i}_{\cT}(Y,  \TSym^{k} \sH_{\cT}\otimes
\TSym^{k'} \sH_{\cT}(n)) \to H^{i + 2d}_{\cT}(Y^2, \TSym^{[k, k']}\sH_{\cT}(n + d))\end{equation}
  for $\cT \in \{ B, \dR, \et, \bar\et, \rig, \syn, \cH \}$ (and any $i$ and $n$). 

 \subsection{Lieberman's trick}
  \label{sect:lieberman}

  \begin{definition}
   For an integer $k \ge 0$, let $\frS_k$ be the symmetric group on $k$ letters, and let $\frT_k$ be the semidirect product $\mu_2^k \rtimes \frS_k$. We define a character
   \begin{align*}
    \varepsilon_k: \frT_k &\to \mu_2\\
    (\eta_1,\ldots,\eta_k,\sigma) &\mapsto \eta_1\cdots\eta_k\sgn(\sigma).
   \end{align*}
   (Cf.~\cite[\S A.1]{Scholl-Kato-Euler-system}.)
  \end{definition}

  Let $\pi:\cE \to Y$ be an elliptic curve such that $\cE$ and $Y$ are regular. Let $\pi^k:\cE^k\to Y$ be the $k$-fold fibre product of $\cE$ over $Y$. On $\cE$ the group $\mu_2$ acts via $[-1]:\cE \to \cE$, and on $\cE^k$ the symmetric group $\frS_k$ acts by permuting the factors. This induces an action of the semi-direct product $\frT_k$ on $\cE^k$.

  \begin{definition}
   Let
   \[
    H^i_\mot(Y, \TSym^k\sH_\QQ(j))\coloneqq H^{i+k}_\mot(\cE^k,\QQ(j+k))(\varepsilon_k)
   \]
   be the $\varepsilon_k$-eigenspace.
  \end{definition}
%

  The next result, which is a standard application of Lieberman's trick, justifies the above notation. Cf.~\cite[Lemma 1.5]{BannaiKings}.

  \begin{theorem}
   \label{thm:lieberman}
   Let $\cT \in \{B,\et, \bar\et, \dR, \syn, \rig,\cD\}$, and $\QQ_\cT,\sH_\cT$ be the realizations of $\QQ,\sH$ in the respective categories. Then one has isomorphisms
   \[
    H^{i+k}_\cT(\cE^k,\QQ_\cT(j+k))(\varepsilon_k) \isom H^i_\cT(Y, \TSym^k\sH_{\cT}(j)).
   \]
   Moreover, the regulator map $r_{\cT}$ commutes with the action of $\frT_k$, and thus gives a map
   \[
    r_\cT:H^i_\mot(Y, \TSym^k\sH_\QQ(j)) \to H^i_\cT(Y, \TSym^k\sH_{\cT}(j)).
   \]
  \end{theorem}

  \begin{proof}
   This is immediate from the fact that $[-1]_*$ acts on $R^i\pi_*\QQ_\cT$ by multiplication with $[-1]^{2-i}$ and the isomorphism $\sH_{\cT}\isom R^1\pi_*\QQ_\cT(1)$ induced from the Tate pairing once one has a Leray spectral sequence.
  \end{proof}

  \begin{remark}
   In the case when $Y$ is a smooth $\Zp$-scheme, we can use Lieberman's trick to extend the comparison morphisms $\comp$ and $\comp_{\dR}$, defined above for cohomology with coefficients that are twists of the Tate object, to sheaves of the form $\TSym^k \sH$. By identifying $H^i_{\syn}(Y, \TSym^k(\sH)(r))$ with a direct summand of $H^{i + k}_{\syn}(\cE^k, \Qp(r + k))$, and similarly for \'etale and de Rham cohomology, we obtain maps
   \[ \comp: H^i_{\syn}(Y, \TSym^k(\sH_{\Qp})(r)) \to H^i_{\et}(Y_{\Qp}, \TSym^k(\sH_{\Qp})(r)) \]
   and isomorphisms
   \[ \comp_{\dR}: H^i_{\dR}(Y_{\Qp}, \TSym^k(\sH_{\Qp})(r)) \to \DD_{\mathrm{cris}}(H^i_{\et}(Y_{\overline{\QQ}_p}, \TSym^k(\sH_{\Qp})(r))). \]
  \end{remark}


 \subsection{Modular curves}
  \label{section:modularcurves}

  We recall some notations for modular curves, following \cite[\S\S 1--2]{Kato-p-adic}. For an integer $N \ge 5$, we shall write $Y_1(N)$ for the $\ZZ[1/N]$-scheme representing the functor
  \[
   S\mapsto\{
   \mbox{isomorphism classes $(E, P)$}\}
  \]
  where $S$ is a $\ZZ[1/N]$-scheme, $E/S$ is an elliptic curve, and $P \in E(S)$ is a section of exact order $N$ (equivalently, an embedding of the constant group scheme $\ZZ/N\ZZ$ into $E$).

  We use the same analytic uniformization of $Y_1(N)(\CC)$ as in \cite[1.8]{Kato-p-adic}:
  \begin{align*}
   \Gamma_1(N) \backslash \HH &\isom Y_1(N)(\CC) \\  \tau &\mapsto
  (\CC/(\ZZ\tau+\ZZ), 1/N)
  \end{align*}
  where $\HH$ is the upper half plane and $\Gamma_1(N)\coloneqq\{ \stbt 1 * 0 1 \bmod N \} \subseteq \SL_2(\ZZ)$.

  Let $\Tate(q)$ be the Tate curve over $\ZZ((q))$, with its canonical differential
  $\omega_\can$. Let $\zeta_N \coloneqq e^{2\pi i/N}$; then the pair $\left(\Tate(q), \zeta_N\right)$ defines a morphism $\Spec \ZZ[\zeta_N,1/N]((q))\to Y_1(N)$ and hence defines a cusp $\infty$.
  The $q$-development at $\infty$ is compatible with the Fourier series in the analytic
  theory if one writes $q \coloneqq e^{2\pi i\tau}$.

  \begin{remark}
   Note that the cusp $\infty$ is not defined over $\QQ$ in our model; so the $q$-expansions of elements of the coordinate ring of $Y_1(N)_{\QQ}$, or more generally of algebraic differentials on $Y_1(N)$ with values in the sheaves $\Sym^k \sH_{\QQ}^\vee$, do not have $q$-expansion coefficients in $\QQ$.
  \end{remark}

  The curves $Y_1(N)$ are equipped with Hecke correspondences $T_n$ and $T_n'$ for each integer $n \ge 1$, defined as in \cite[\S 2.9 \& \S 4.9]{Kato-p-adic}.

  \begin{remark}
   Note that the action of the operators $T_n$ on de Rham cohomology is given by the familiar $q$-expansion formulae, while the operators $T_n'$ are the transposes of the $T_n$, which have no direct interpretation in terms of $q$-expansions when $n$ is not coprime to the level.
  \end{remark}

 \subsection{Motives for Rankin convolutions}

  Let $f,g$ be normalized cuspidal new eigenforms of weight $k+2, k'+2$, levels $N_f,N_g$ and characters $\varepsilon(f),\varepsilon(g)$. We choose a number field $L$ containing the coefficients of $f$ and $g$.

  From the work of Scholl \cite{Scholl-motives}, one knows how to associate (Grothendieck) motives $M(f), M(g)$ with coefficients in $L$ to $f,g$. We denote by $M(f \otimes g)$ the tensor product of these motives (over $L$), which is a 4-dimensional motive over $\QQ$ with coefficients in $L$.

  \begin{definition}
   For $\cT \in \{ \dR, \rig, B, \bar\et\}$ (the ``geometric'' cohomology theories) we write $M_{\cT}(f \otimes g)$ for the $\cT$-realization of $M$, which is the maximal $L \otimes_{\QQ} \QQ_{\cT}$-submodule of
  \[
   H^2_\cT(Y_1(N_f) \times Y_1(N_g), \Sym^{(k,k')}\sH_\cT^\vee) \otimes_\QQ L
  \]
  on which the Hecke operators $(T_\ell, 1)$ and $(1, T_\ell)$ act as multiplication by the Fourier coefficients $a_\ell(f)$ and $a_\ell(g)$ respectively, for every prime $\ell$ (including $\ell \mid N_f N_g$).
  \end{definition}

  Note that this direct summand lifts, canonically, to a direct summand of $H^2_{c, \cT}$, since $f$ and $g$ are cuspidal.

  The Hodge filtration of $M_{\dR}(f \otimes g)$ is given by
  \[
   \dim_{L} \Fil^n M_{\dR}(f\otimes g) =
   \begin{cases}
    4& n\le 0\\
    3& 0< n \le\min\{k,k'\}+1\\
    2& \min\{k,k'\}+1<n \le \max\{k,k'\}+1\\
    1& \max\{k,k'\}+1<n\le k+k'+2\\
    0& k+k'+2<n.
   \end{cases}
  \]
  In particular, if $n = 1 + j$ with $0 \le j \le \min(k, k')$, the space $\Fil^{1 + j} M_{\dR}(f \otimes g) = \Fil^0 M_{\dR}(f \otimes g)(1 + j)$ has dimension 3 over $L$.

  \begin{definition}
   Dually, we write $M_{\cT}(f \otimes g)^*$ for the maximal \emph{quotient} of
  \[ H^2_\cT(Y_1(N_f) \times Y_1(N_g), \TSym^{[k,k']}\sH_\cT(2)) \otimes_\QQ L\]
  on which the dual Hecke operators $(T_\ell', 1)$ and $(1, T_\ell')$ act as multiplication by $a_\ell(f)$ and $a_\ell(g)$.
  \end{definition}

  \begin{remark}
   The twist by 2 implies that the Poincar\'e duality pairing
   \[ M_\cT(f \otimes g) \times M_{\cT}(f \otimes g)^* \to L \otimes_{\QQ} \QQ_{\cT}\]
   is well-defined and perfect, justifying the notation $M_{\cT}(f \otimes g)^*$.
  \end{remark}

  \begin{definition}\label{def:tangentspace}
   Let
   \begin{align*}
   t(M(f\otimes g)(j)) \coloneqq \frac{M_{\dR}(f\otimes g)(j)}{\Fil^0 M_{\dR}(f\otimes g)(j)}&&
   t(M(f\otimes g)^*(-j)) \coloneqq \frac{M_{\dR}(f\otimes g)^*(-j)}{\Fil^0 M_{\dR}(f\otimes g)^*(-j)}
   \end{align*}
   be the tangent spaces of the motives $M(f\otimes g)(j)$ and $M(f\otimes g)^*(-j)$.
  \end{definition}

  Note that $ t(M(f\otimes g)^*(-j))$ is the dual of the $L$-vector space $\Fil^{1 + j} M_{\dR}(f\otimes g)$; so for $0 \le j \le \min(k, k')$ it is also 3-dimensional over $L$. This tangent space will be the target of the Abel--Jacobi maps in \S \ref{sect:ajdefs} below.


 \subsection{Rankin L-functions}\label{section:Rankin-L-fct}

  Let $f$, $g$ be cuspidal eigenforms of weights $r, r' \ge 1$, levels $N_f, N_g$ and characters $\varepsilon_f,\varepsilon_g$. We define the Rankin $L$-function
  \[ L(f, g, s) = \frac{1}{L_{(N_f N_g)} (\varepsilon_f \varepsilon_g, 2s + 2 - r - r')} \sum_{n \ge 1} a_n(f) a_n(g) n^{-s}, \]
  where $L_{(N_f N_g)} (\varepsilon_f \varepsilon_g,s)$ denotes the Dirichlet $L$-function with the Euler factors at the primes dividing $N_f N_g$ removed.

  \begin{remark}
   This Dirichlet series differs by finitely many Euler factors from the $L$-function $L(\pi_f \otimes \pi_g, s)$ of the automorphic representation $\pi_f \otimes \pi_g$ of $\GL_2 \times \GL_2$ associated to $f$ and $g$. In particular, it has meromorphic continuation to all of $\CC$. It is holomorphic on $\CC$ unless $\langle \bar{f}, g \rangle \ne 0$, where $\bar{f} = \sum \bar{a_n}(f) q^n$, in which case it has a pole at $s = r$. If $f$ and $g$ are normalized newforms and $(N_f, N_g) = 1$, then $L(f, g, s) = L(\pi_f \otimes \pi_g, s)$.
  \end{remark}

  \begin{theorem}[{Shimura, \cite{Shimura-periods}}]
   For integer values of $s$ in the range $r' \le s \le r-1$, we have
   \[ \frac{L(f, g, s)}{\pi^{2s - r' + 1} \langle f, f \rangle_{N_f}} \in \overline{\QQ}, \]
   where $\langle f_1, f_2 \rangle_{N}$ is the Petersson scalar product of weight $r$ modular forms defined by
   \[ \int_{\Gamma_1(N) \backslash \HH} \overline{f_1(\tau)} f_2(\tau) \Im(\tau)^{r - 2}\, \mathrm{d}x \wedge \mathrm{d}y.\]
  \end{theorem}

  More precisely, we have the Rankin--Selberg integral formula
  \begin{equation}\label{eq:Rankin-formula}
   L(f, g, s) = N^{r + r' - 2s - 2} \frac{\pi^{2s - r' + 1} (-i)^{r - r'} 2^{2s + r - r'}}{\Gamma(s)\Gamma(s - r' + 1)}
      \left\langle \overline{f}, g E^{(r - r')}_{1/N}(\tau, s - r + 1)\right\rangle_N
  \end{equation}
  where $N \ge 1$ is some integer divisible by $N_f$ and $N_g$ and with the same prime factors as $N_f N_g$, and $E^{(r - r')}_{1/N}(\tau, s - r + 1)$ is a certain real-analytic Eisenstein series (cf.~\cite[Definition 4.2.1]{LLZ14}), whose values for $s$ in this range are nearly-holomorphic modular forms defined over $\overline{\QQ}$.

  The next theorem shows that the algebraic parts of the $L$-values $L(f, g, s)$, for $s$ in the above range, can be $p$-adically interpolated.

  \begin{theorem}[Hida]\label{thm:Hida}
   Let $p \nmid N_f N_g$ be a prime at which $f$ is ordinary. Then there is a $p$-adic $L$-function $L_p(f, g) \in \overline{\QQ}_p \otimes_{\Zp} \Zp[[
   \Gamma]]$ with the following interpolation property: for $s$ an integer in the range $r' \le s \le r-1$, we have
   \[ L_p(f, g, s) = \frac{\mathcal{E}(f, g, s)}{\mathcal{E}(f) \mathcal{E}^*(f)} \cdot \frac{\Gamma(s)\Gamma(s - r' + 1)} {\pi^{2s - r' + 1} (-i)^{r - r'} 2^{2s + r - r'} \langle f, f \rangle_{N_f}} \cdot L(f, g, s),\]
   where the Euler factors are defined by
   \begin{align*}
    \mathcal{E}(f) &= \left( 1 - \frac{\beta_f}{p \alpha_f}\right), \\
    \mathcal{E}^*(f) &= \left( 1 - \frac{\beta_f}{\alpha_f}\right),\\
    \mathcal{E}(f, g, s) &= \left( 1 - \frac{p^{s-1}}{\alpha_f \alpha_g}\right) \left( 1 - \frac{p^{s-1}}{\alpha_f \beta_g}\right) \left( 1 - \frac{\beta_f \alpha_g}{p^s}\right) \left( 1 - \frac{\beta_f \beta_g}{p^s}\right).
   \end{align*}
   Moreover, the function $L_p(f, g, s)$ varies analytically as $f$ varies in a Hida family, and if $g$ is ordinary it also varies analytically in $g$.
  \end{theorem}

  \begin{remark}\label{remark:DpversusLp}\mbox{~}
   \begin{enumerate}
    \item The $L$-function $L_p(f, g, s)$ considered here is $N^{2s + 2 - r - r'} \mathcal{D}_p(f, g, 1/N, s)$ in the notation of \cite[\S 5]{LLZ14}. We include the power of $N$ in the definition because it makes $L_p(f, g, s)$ independent of the choice of $N$.
    \item The complex $L$-function $L(f, g, s)$ is symmetric in $f$ and $g$, i.e.~we have $L(f, g, s) = L(g, f, s)$, but this is not true of $L_p(f, g, s)$.
    \item Removing the ordinarity condition on $f$ at $p$ has been the subject of several recent works; see Remark \ref{rem:nonord}(ii) below.
   \end{enumerate}
  \end{remark}


\section{Eisenstein classes on \texorpdfstring{$Y_1(N)$}{Y1(N)}}

 \subsection{Motivic Eisenstein classes}
  \label{sect:motivicEis}

  The fundamental input to the constructions of this paper are the following cohomology classes first constructed by Beilinson:

  \begin{theorem}
   \label{thm:motivic-eisenstein-classes}
   Let $N \ge 5$ and let $b \in \ZZ / N\ZZ$ be nonzero. Then there exist nonzero cohomology classes (``motivic Eisenstein classes'')
   \[ \Eis^k_{\mot, b, N} \in H^1_{\mot}(Y_1(N), \TSym^k \sH_{\QQ}(1))\]
   for all integers $k \ge 0$, satisfying the following residue formula: we have
   \[ \res_{\infty}\left(\Eis^k_{\mot, b, N}\right) = -N^k \zeta(-1-k). \]
  \end{theorem}

  \begin{proof}
   By the construction in \cite[\S 6.4]{BeilinsonLevin} there
   is associated to the canonical order $N$ section $t_N$ of the universal elliptic curve $\cE$ over $Y_1(N)$
   a class
   $ \sE^{k+2}_{\mot,b}\in H^1_{\mot}(Y_1(N), \TSym^k \sH_{\QQ}(1))$, which is essentially the specialization of the elliptic
   polylogarithm at $b t_N$. It has
   the property that \cite[6.4.5]{BeilinsonLevin}
   \[
    r_\et(\sE^{k+2}_{\mot,b})=-N^{k-1}\operatorname{contr}_{\sH_{\QQ_p}}((bt_N)^*\operatorname{pol}^{k+1})
   \]
   where the right hand side is the notation of \cite[\S 4.2]{Kings-Eisenstein}. Now we set
   $\Eis^k_{\mot,b,N}\coloneqq-N\sE^{k+2}_{\mot,b}$; by the residue formula of \cite[Theorem 5.2.2]{Kings-Eisenstein}
   we see that $\Eis^k_{\mot,b,N}$ has the stated residue at $\infty$.
  \end{proof}

  \begin{remark}
   Note that our definition of the residue map is a little different from that of \cite{Kings-Eisenstein}. Firstly, the residue map in \emph{op.cit.} is defined using $Y(N)$, rather than $Y_1(N)$, which introduces a factor of $N$. Secondly, we directly use a trivialisation of $\TSym^k \sH_{\QQ}$ at the cusp $\infty$, rather than using the isomorphism to $\Sym^k \sH_{\QQ}$, which introduces a factor of $k!$ (cf. \cite[Lemma 5.1.6]{Kings-Eisenstein}).
  \end{remark}

  \begin{remark}
   For $k = 0$, we have $H^1_{\mot}(Y_1(N), \QQ(1)) = \cO(Y_1(N))^\times \otimes \QQ$, and the Eisenstein class $\Eis^k_{\mot, b, N}$ is simply the Siegel unit $g_{0, b/N}$ in the notation of \cite{Kato-p-adic}.
  \end{remark}


 \subsection{Eisenstein classes in other cohomology theories}
  \label{sect:generalEis}

  As a consequence of the existence of the motivic Eisenstein class, we immediately obtain Eisenstein classes in all the other cohomology theories introduced above: we define
  \[ \Eis^k_{\cT, b, N} = r_{\cT}\left(\Eis^k_{\mot, b, N}\right).\]

  In particular, we have the following:
  \begin{itemize}
   \item An \'etale Eisenstein class
   \[ \Eis^k_{\et, b, N} \in H^1_{\et}\left(Y_1(N)[1/p], \TSym^k \sH_{\Qp}(1)\right). \]
   \item A de Rham Eisenstein class
   \[ \Eis^k_{\dR, b, N} \in H^1_\dR\left(Y_1(N)_{\QQ}, \TSym^k \sH_{\QQ}(1)\right).\]
   \item An Eisenstein class in absolute Hodge cohomology
   \[ \Eis^k_{\cH, b, N} \in H^1_{\cH}\left(Y_1(N)_{\RR}, \TSym^k \sH_{\RR}(1)\right)\]
   whose image in $H^1_\dR(Y_1(N)_{\RR}, \TSym^k \sH_{\RR}(1))$ coincides with the image of $\Eis^k_{\dR, b, N}$ under $\otimes \RR$.
   \item A syntomic Eisenstein class
   \[ \Eis^k_{\syn, b, N} \in H^1_{\syn}\left(\sY, \TSym^k \sH_{\Qp}(1)\right),\]
   for any $p \nmid N$, where $\sY$ denotes the smooth pair $(Y_1(N)_{\Zp}, X_1(N)_{\Zp})$, whose image in the group $H^1_\dR(Y_1(N)_{\Qp}, \TSym^k \sH_{\Qp}(1))$ coincides with the image of $\Eis^k_{\dR, b, N}$ under $\otimes \Qp$, and whose image under the map
   \[ \comp: H^1_{\syn}\left(\sY, \TSym^k \sH_{\Qp}(1)\right) \to H^1_{\et}\left(Y_1(N)_{\Qp}, \TSym^k \sH_{\Qp}(1)\right)\]
   is the localization at $p$ of the \'etale Eisenstein class $\Eis^k_{\et, b, N}$.
  \end{itemize}

  We shall give explicit formulae for the de Rham, absolute Hodge, and syntomic Eisenstein classes in \S\S \ref{sect:deRhamEis}--\ref{sect:syntomicEis} below; these formulae involve Eisenstein series -- classical algebraic Eisenstein series for the de Rham case, and real-analytic and p-adic Eisenstein series in the absolute Hodge and syntomic cases respectively.


 \subsection{The de Rham Eisenstein class}
  \label{sect:deRhamEis}

  We give an explicit formula for the de Rham Eisenstein class $\Eis^k_{\dR, b, N}$, in terms of certain modular forms which are Eisenstein series of weight $k + 2$.

  \begin{definition}
   \label{def:algebraic-eisenstein-series}
   Write $\zeta_N \coloneqq e^{2\pi i/N}$ and $q \coloneqq e^{2\pi i\tau}$.
   For $k \ge -1$ and $b\in \ZZ/N\ZZ$ not equal to zero, we define an algebraic Eisenstein series
   \[
    F_{k+2,b} \coloneqq \zeta (-1-k)+
    \sum_{n>0}q^n
    \sum_{\stackrel{dd'= n}{d,d'>0}}d^{k+1}(\zeta_N^{bd'} + (-1)^k\zeta_N^{-bd'})
   \]
   where $\zeta(s)\coloneqq\sum_{n>0}\frac{1}{n^s}$ is the Riemann zeta function.
  \end{definition}

  \begin{remark}
  The Eisenstein series $F_{k+2,b}$ is $F^{(k+2)}_{0,b/N}$ in the notation of \cite[Proposition 3.10]{Kato-p-adic}.
  \end{remark}

  As in Section \ref{section:modularcurves}, denote by $\Tate(q)$ the Tate elliptic curve over $\ZZ((q))$. Endowing $\Tate(q)$ with the order $N$ section corresponding to $\zeta_N \in \mathbf{G}_m / q^\ZZ$ gives a point of $Y_1(N)$ over $\ZZ((q)) \otimes_{\ZZ} \ZZ[1/N, \zeta_N]$. Moreover, the canonical differential $\omega_{\can}$ on $\Tate(q)$ gives a section trivializing $\Fil^1 \sH^\vee |_{\Tate(q)}$, and hence of $\Fil^0 \sH |_{\Tate(q)}$, via the isomorphism $\sH \cong \sH^\vee(1)$; we write $v^{[0, k]}$ for the $k$-th tensor power of this section, regarded as a section of $\TSym^k \sH$ over $\Tate(q)$.

  \begin{proposition}
   \label{prop:de-rham-eis}
   The pullback of the de Rham Eisenstein class to $(\Tate(q), \zeta_N)$ is given by
   \[ \Eis^k_{\dR, b, N} = -N^k F_{k+2,b} \cdot v^{[0, k]} \otimes \frac{\mathrm{d}q}{q}.\]
  \end{proposition}

  \begin{proof}
   See \cite{BannaiKings}. (Our normalizations are slightly different from those of \emph{op.cit.}, but it is clear that the above class has the correct residue at $\infty$.)
  \end{proof}


 \subsection{The Eisenstein class in absolute Hodge cohomology}
  \label{sect:DeligneEis}

  The results in this section are well-known and due to Beilinson and Deninger.
  As we have to express the computations with our normalizations anyway, we decided
  to show how the Eisenstein class in absolute Hodge cohomology can be
  computed very easily as in the syntomic case by solving explicitly a differential
  equation. This transfers the main idea in \cite{BannaiKings} from syntomic
  cohomology to the case of
  absolute Hodge cohomology. The aim of this section is to give an explicit description of the class $\Eis^k_{\cH,b,N}\coloneqq r_{\cH}(\Eis^k_{\mot,b,N})$.

  \begin{proposition}
   \label{prop:explicitdelignecoho}
   The group $H^1_\cH(Y_1(N)_\RR, \TSym^k\sH_\RR(1))$ is the group of equivalence
   classes of pairs $(\alpha_\infty,\alpha_\dR)$, where
   \[
    \alpha_\infty\in \Gamma(Y_1(N)(\CC), \TSym^k\sH_\RR\otimes \sC^\infty)
   \]
   is a $\sC^\infty$-section of $\TSym^k(\sH_\RR)(1)$ and
   $\alpha_\dR\in \Gamma(X_1(N)_\RR, \TSym^k\underline{\omega}\otimes \Omega^1_{X_1(N)}(C))$
   is an algebraic section with logarithmic poles along $C \coloneqq X_1(N)\setminus Y_1(N)$,
   such that
   \[
    \nabla(\alpha_\infty)=\pi_1(\alpha_\dR).
   \]
   A pair $(\alpha_{\infty}, \alpha_{\dR})$ is equivalent to $0$ if we have
   \[ (\alpha_{\infty}, \alpha_{\dR}) = (\pi_1(\beta), \nabla(\beta)) \text{ for some }\beta \in \Gamma(X_1(N)_\RR, \TSym^k(\underline{\omega})(C)).\]
   Here $\pi_1:\TSym^k\sH_\CC\isom \TSym^k\sH_\RR\otimes\CC\to \TSym^k\sH_\RR(1)$
   is induced by the projection $\CC\to \RR(1)$, $z\mapsto (z-\overline{z})/2$.
  \end{proposition}

  \begin{proof}
   Can be deduced either from the explicit description as group of extensions
   of mixed Hodge modules or from the standard description of
   $H^{k+1}_\cD(\cE^k_\RR,\RR(k+1))$ and application of the projector $\varepsilon_k$.
  \end{proof}

  Consider the covering $\HH\to Y_1(N)(\CC)$, which maps $\tau\mapsto (\CC/(\ZZ\tau+\ZZ),1/N)$.
  Over $\HH$ we have the standard section $\omega=dz$ of $\sH^\vee_{\CC}$, where
  $z$ is the coordinate on $\CC$. Denote by $\langle\quad,\quad\rangle$ the Poincar\'e duality pairing
  on $\sH^\vee_\CC$ and by $\omega^\vee$, $\overline{\omega}^\vee$ the
  basis dual to $\omega,\overline{\omega}$.  We have the following formulae:
  \begin{align*}
  \langle\overline{\omega},\omega\rangle&=\frac{\tau-\overline{\tau}}{2\pi i}&&& \langle\omega,\overline{\omega}\rangle&=-\langle\overline{\omega},\omega\rangle\\
  \omega^\vee&=\frac{2\pi i\overline{\omega}}{\tau-\overline{\tau}}&&&\overline{\omega}^\vee&=
  -\frac{2\pi i{\omega}}{\tau-\overline{\tau}}\\
  \nabla(\omega)&=\frac{\omega-\overline{\omega}}{\tau-\overline{\tau}}\mathrm{d}\tau&&&
  \nabla(\overline{\omega})&=\frac{\omega-\overline{\omega}}{\tau-\overline{\tau}}
  \mathrm{d}\overline{\tau}\\
  \nabla(\omega^\vee)&=-\frac{\omega^\vee \mathrm{d}\tau+\overline{\omega}^\vee \mathrm{d}\overline{\tau}}{\tau-\overline{\tau}}&&&
  \nabla(\overline{\omega}^\vee)&=\frac{\omega^\vee \mathrm{d}\tau+\overline{\omega}^\vee \mathrm{d}\overline{\tau}}{\tau-\overline{\tau}}
  \end{align*}

  \begin{definition}
   Let
   $w^{(r,s)}\coloneqq\omega^r\overline{\omega}^s\in \Sym^k\sH^\vee_\CC$ and
   $w^{[r,s]}\coloneqq\omega^{\vee,[r]}\overline{\omega}^{\vee,[s]}\in \TSym^k\sH_\CC$, where we use the notation of \eqref{TSymmult} and the product in 
$\TSym^\bullet\sH_\CC$.
  \end{definition}

  One has $\overline{w^{(r,s)}}=w^{(s,r)}$, and $\overline{w^{[r,s]}}=(-1)^{r+s}w^{[s,r]}$. Also, one has
  \[ \nabla(w^{[r, s]}) = \frac{1}{\tau - \bar\tau}\left( (-r d\tau + s d\bar\tau) w^{[r, s]} + (r + 1) w^{[r + 1, s-1]} - (s + 1) w^{[r-1, s+1]}\right),\]
  and hence, if $D_j$ are $C^\infty$ functions of $\tau$ satisfying the symmetry relation
  \[ (2\pi i)^{k-j} (\tau - \bar \tau)^{k-j} D_{k - j} = \overline{ (2\pi i)^j (\tau - \bar \tau)^{j} D_j},\]
  we have
  \begin{multline}
   \label{eq:deligne-eisenstein-class2}
   \nabla\left( \sum_{j = 0}^k (2\pi i)^j (\tau - \bar \tau)^j D_j w^{[k-j, j]} \right)\\
   = \sum_{j = 0}^k \Big[(2\pi i)^{j + 1}(\tau - \bar\tau)^j \left(\delta_{2j - k}(D_j) + (k - j) D_{j + 1}\right)\,\mathrm{d}\tau + \left(\dots\right) \mathrm{d}\bar\tau\Big] w^{[k-j, j]},
  \end{multline}
  where $\dots$ indicate the term derived from the previous one by interchanging $j$ and $k-j$ and applying complex conjugation. Here $\delta_r$ denotes the Maass--Shimura differential operator $(2\pi i)^{-1} \left(\frac{\partial}{\partial \tau} + \frac{r}{\tau - \bar \tau}\right)$.

  We define real analytic Eisenstein series:
  \begin{definition}
   \label{def:real-eisenstein-series}
   For $t\ge 0$ and $s\in\CC$ with $t+2\Re(s)>2$ and $b\in\ZZ/N\ZZ$, we define
   \[
   F^\an_{t,s,b} \coloneqq
   (-1)^t\frac{\Gamma(s+t)}{(2\pi i)^{s+t}}
   \sum_{(m,n)\in \ZZ^2\setminus \{(0,0)\}}
   \frac{e^{2\pi i mb/N}(\tau-\overline{\tau})^s}{(m\tau +n)^{t}|m\tau+n|^{2s}}.
   \]
  \end{definition}

  \begin{remark}
   In \cite[4.2.1]{LLZ14} this Eisenstein series was denoted by $F^{(t)}_{b/N}(\tau,s)$. Note also that $F^\an_{k+2,0,b}$ coincides with the algebraic Eisenstein series $F_{k+2,b}$ of Definition \ref{def:algebraic-eisenstein-series}.
  \end{remark}

  With these definitions the Eisenstein class in absolute Hodge cohomology is given
  as follows:

  \begin{proposition}
   \label{prop:deligne-eisenstein}
   The class $\Eis^k_{\cH,b,N}=(\alpha_{\infty},\alpha_{\dR})\in
   H^1_\cH(Y_1(N)_\RR, \TSym^k\sH_\RR(1))$ is given by
   \[
   \alpha_{\infty}\coloneqq
   \frac{-N^k}{2}
   \sum_{j=0}^k(-1)^{j}(k-j)!(2\pi i)^{j-k}(\tau-\overline{\tau})^{j}F^{\an}_{2j-k, k+1-j, b} w^{[k-j,j]}
   \]
   and
   \[
   \alpha_{\dR}\coloneqq
   (-1)^{k+1} N^kF^\an_{k+2, 0, b}(\tau-\overline{\tau})^kw^{[0,k]} \otimes 2\pi i\, \mathrm{d}\tau.
   \]
  \end{proposition}

  \begin{proof}
   One checks easily that $\alpha_\infty$ takes values in $\TSym^k(\sH_{\RR})(1)$. Moreover, an elementary but lengthy computation using Equation \eqref{eq:deligne-eisenstein-class2} and the relation $\delta_t(F^\an_{t,s,b})= F^\an_{t + 2, s - 1, b}$  (cf.~\cite[Proposition 4.2.2(iii)]{LLZ14}) shows that $\nabla(\alpha_\infty) = \pi_1(\alpha_{\dR})$, so that $(\alpha_\infty, \alpha_{\dR})$ does indeed define a class in absolute Hodge cohomology. 
   
   For $k > 0$, the Eisenstein class $\Eis^k_{\cH,b,N}$ is uniquely determined by
   $\alpha_\dR=\Eis^k_{\dR,b,N}$; and $\alpha_{\dR}$ as defined above satisfies this, since $F^\an_{k+2, 0, b}$ coincides with the holomorphic Eisenstein series $F_{k + 2, b}$ of the previous section (and $\tfrac{\mathrm{d}q}{q} = 2\pi i\, \mathrm{d}\tau$, $v^{[0, k]} = (-1)^k (\tau - \bar\tau)^k w^{[0, k]}$).
 
    For $k = 0$ the Eisenstein class is characterized by
   $\Eis^0_{\cH,b,N}=(\log|g_{0,b/N}|,d\log(g_{0,b/N}))$, and we have
   \[ F^{\an}_{0, 1, b} = -2 \log |g_{0, b/N}|.\qedhere\]
  \end{proof}

  \begin{remark}
   Note that there is a typographical error in formula (iii) of \cite[(3.8.4)]{Kato-p-adic} (a minus sign is missing), which we incautiously reproduced without checking in \cite[Proposition 4.2.2(v)]{LLZ14}.
  \end{remark}


 \subsection{The syntomic Eisenstein class on the ordinary locus}
  \label{sect:syntomicEis}

  We review the description from \cite{BannaiKings} of the syntomic Eisenstein class
  $\Eis^k_{\syn, b, N} \in  H^1_{\syn}(\sY^\ord, \TSym^k \sH_{\Qp}(1))$
  in terms of $p$-adic Eisenstein series. We assume here that $p \nmid N$.

  In this section we let $X = X_1(N)_{\Zp}$, $Y = Y_1(N)_{\Zp}$, and $Y^\ord$ the open subscheme of $Y$ where the Eisenstein series
  $E_{p-1}\in \Gamma(Y,\underline{\omega}^{\otimes p-1})$ is invertible.
  Let $\cY$, $\cX$ be the formal completions of $Y$ and $X$ with respect to the special
  fibre and $\cY_\Qp$, $\cX_\Qp$ be the associated rigid analytic spaces.
  We also let $j:\cY_\Qp\to \cX_\Qp$ be the open immersion and
  $Y_\Qp^\an$ be the rigid analytic space associated to $Y_\Qp$ so
  that $Y_\Qp^\an$ is a strict neighbourhood of $j$. We denote by $j^{\dagger}$ Berthelot's overconvergent sections functor and let
\[ 
\TSym^{k}\sH_{\rig}:=j_{Y_\Qp^\an}^\dagger\TSym^k\sH\mid_{Y_\Qp^\an}.
\]

  Let $\sY$ and $\sY^\ord$ be the smooth pairs $(Y, X)$ and $(Y^\ord,X)$. We shall give an explicit formula for the image of $\Eis^k_{\syn, b, N}$ under
  the restriction map
  \[
  H^1_{\syn}(\sY, \TSym^k \sH_{\Qp}(1))\to
  H^1_{\syn}(\sY^\ord, \TSym^k \sH_{\Qp}(1)).
  \]

  These syntomic cohomology groups have an explicit presentation exactly parallel to Proposition \ref{prop:explicitdelignecoho} above:

  \begin{proposition}[{\cite[Proposition A.16]{BannaiKings}}]
   A class in $H^1_{\syn}(\sY^\ord, \TSym^k \sH_{\Qp}(1))$ is given by a pair of sections $(\alpha_\rig, \alpha_\dR)$, where
   \[ \alpha_\rig \in \Gamma(\cX_{\Qp},\TSym^k\sH_\rig\mid_{Y_\Qp^\an})\]
   is an overconvergent section,
   \[\alpha_\dR \in \Gamma(X_\Qp, \TSym^k\underline{\omega}\otimes
   \Omega^1_{X_\Qp}(\Cusp))\]
   is an algebraic section with logarithmic poles along $\Cusp\coloneqq X_\Qp\setminus Y_\Qp$, and we have the relation
   \[ \nabla(\alpha_\rig)=(1-\varphi)\alpha_\dR.\]
  \end{proposition}

  The natural map $H^1_{\syn}(\sY, \TSym^k \sH_{\Qp}(1)) \to H^1_\dR(Y_{\Qp}, \TSym^k \sH_{\Qp}(1))$ is given by mapping $(\alpha_\rig, \alpha_\dR)$ to the class of $\alpha_\dR$; thus $\alpha_\dR$ must be the unique algebraic differential with logarithmic poles at the cusps representing $\Eis^k_{\dR, b, N}$, whose $q$-expansion was given in Proposition \ref{prop:de-rham-eis} above.

  Before we can write down the explicit formula for $\alpha_\rig$,
  we need to introduce certain $p$-adic Eisenstein series, and a certain
  trivialization of $\sH_\Qp$.

  \begin{definition}
   Write $\zeta_N\coloneqq e^{2\pi i/N}$ and $q\coloneqq e^{2\pi i\tau}$.
   For $t, s \in \ZZ$ with $t + s \ge 1$ we set
   \[
    F^{(p)}_{t, s, b} \coloneqq \sum_{n > 0} q^n
    \sum_{\substack{ dd'=n \\p\nmid d'\\ d,d'>0}}
    d^{t + s - 1} (d')^{-s}(\zeta_N^{bd'}+(-1)^t\zeta_N^{-bd'})
    \in \Zp[\zeta_N][[q]].
   \]
  \end{definition}

  \begin{remark}
   For $(t, s)$ satisfying the inequalities $s + t \ge 1$, $t \ge 1$, $s \le 0$, the real-analytic Eisenstein series $F^{\an}_{t, s, b}$ is an algebraic nearly-holomorphic modular form defined over $\QQ$, and the $p$-adic one $F^{(p)}_{t, s, b}$ is the $p$-stabilization of this form; cf.~\S 5.2 of \cite{LLZ14}.
  \end{remark}

  Let $\widetilde{\cY}$ be the formal scheme
  which classifies elliptic curves over $p$-adic rings with a $\Gamma_1(N)$-structure together with
  an isomorphism $\eta: \widehat{\mathbb{G}}_m\isom \widehat{\cE}$
  of formal groups. The
  elements in $\Gamma(\widetilde{\cY},\cO_{\widetilde{\cY}})$ are called
  Katz $p$-adic modular forms.  The discussion in \cite[\S 5.2]{BannaiKings}
  shows that $F^{(p)}_{t, s, b}$ is the $q$-expansion of a $p$-adic modular form.

  Denote by $\widetilde{\sH}_\Qp$ the pull-back of $\sH_\Qp$
  to $\widetilde{\cY}$. Its dual $\sH^\vee_\Qp$ contains a canonical
  section $\widetilde{\omega}$ with $\eta^*(\widetilde{\omega})=dT/(1+T)$.

  \begin{definition}
   Denote by $\xi\in \Omega^1_{\widetilde{\cY}/\Zp}$ the differential form which corresponds to $ \widetilde{\omega}^{\otimes 2}$ under the Kodaira-Spencer isomorphism  $\underline{\omega}^{\otimes 2}\isom \Omega^1_{\widetilde{\cY}/\Zp}$  and denote by $\theta$ the differential operator dual to $\xi$.
  \end{definition}

  Then define $\widetilde{u}\coloneqq\nabla(\theta)(\widetilde{\omega})\in
  \widetilde{\sH}_\Qp^\vee$. The element $\widetilde{u}$ is a generator of
  the unit root space. We denote by $\widetilde{\omega}^\vee,\widetilde{u}^\vee$
  the basis dual to $\widetilde{\omega},\widetilde{u}$.

  The actions of $\nabla$ and $\varphi$ on these vectors are given by the formulae
  \begin{align*}
   \nabla(\tilde\omega) &= \tilde u \otimes \xi, &
   \nabla(\tilde \omega^\vee) &= 0, \\
   \nabla(\tilde u) &= 0, &
   \nabla(\tilde u^\vee) &= \tilde \omega^\vee \otimes \xi,\\
   \varphi(\tilde \omega) &= p \tilde\omega, &
   \varphi(\tilde \omega^\vee) &= p^{-1} \tilde\omega^\vee,\\
   \varphi(\tilde u) &= \tilde u,&
   \varphi(\tilde u^\vee) &= \tilde u^\vee.
  \end{align*}

  We are interested in the sheaves $\Sym^k \sH^\vee$ and $\TSym^k \sH$. The pullbacks of these to $\widetilde{\mathcal Y}$ have bases of sections given, respectively, by the sections $v^{(r,s)} \coloneqq \tilde{\omega}^r \tilde{u}^s$ with $r + s = k$, and by the $v^{[r,s]}\coloneqq (\tilde{\omega}^\vee)^{[r]} (\tilde{u}^\vee)^{[s]}$ with $r + s = k$. The pairing is given by $\langle v^{[r, s]}, v^{(r', s')}\rangle = \delta_{rr'} \delta_{ss'}$, so these two bases are dual to each other.

  We have
  \begin{align*}
   \nabla(v^{(r,s)}) &= r v^{(r-1,s+1)} \otimes \xi, &
   \nabla(v^{[r,s]}) &= (r + 1) v^{[r+1,s-1]} \otimes \xi, \\
   \varphi(v^{(r,s)}) &= p^r v^{(r,s)}, &
   \varphi(v^{[r,s]}) &= p^{-r} v^{[r,s]}.\\
  \end{align*}

  \begin{remark}
   The Tate curve $\Tate(q)$, equipped with its natural isomorphism of formal groups to $\mathbb{G}_m$ and the point $\zeta_N$ of order $N$, defines a point of $\widetilde{\mathcal{Y}}$ over the ring $\Zp[\zeta_N][[q]]$; the pullback of $\xi$ is $\tfrac{\mathrm{d}q}{q}$ and $\theta$ acts as the differential operator dual to this, which is $q \frac{\mathrm{d}}{\mathrm{d}q}$. This identifies our sections $\tilde \omega$ and $\tilde u$ with the $\omega_{\can}$ and $\eta_{\can}$ of \cite[Eq.~(22)]{DR-diagonal-cycles-I}, and $\theta$ with the Serre differential $d$ of \S 2.4 of \emph{op.cit.}
  \end{remark}

  Since $\theta$ acts on $q$-expansions as $q \frac{\mathrm{d}}{\mathrm{d}q}$, we have the formula
  \[ \theta(F^{(p)}_{t, s, b})=F^{(p)}_{t + 2, s - 1, b}\]
  for any $t, s$ with $t + s \ge 1$.

  \begin{definition}
   Let
   \[
   \alpha_\rig \coloneqq -N^k \sum_{j=0}^k(-1)^{k - j}(k-j)!
   F^{(p)}_{2j-k, k + 1 - j, b} v^{[k-j, j]}
   \in \Gamma(\widetilde{\cY},\TSym^k \widetilde{\sH}_\Qp)
   \]
   and (as above) let
   \[
   \alpha_\dR \coloneqq - N^k F_{k+2,b} v^{[0,k]} \otimes \xi \in \Gamma(X_\Qp, \TSym^k\underline{\omega}\otimes\Omega^1_{X_\Qp}(\Cusp))
   \]
   be the section representing $\Eis^k_{\dR, b, N}$.
  \end{definition}

  It is shown in \cite[Lemma 5.10]{BannaiKings} that  $\alpha_\rig \in \Gamma(\cX_{\Qp},\TSym^k\sH_\rig)$. With  these notations we have the following explicit description of the class $\Eis^k_{\syn, b, N}$; note its similarity to the description of $\Eis^k_{\cH, b, N}$ in Proposition \ref{prop:deligne-eisenstein}:

  \begin{theorem}[{\cite[Theorem 5.11]{BannaiKings}}]
   \label{thm:syntomicformula}
   The class
   $\Eis^k_{\syn, b, N} \in  H^1_{\syn}(\sY^\ord, \TSym^k \sH_{\Qp}(1))$
   is given by $\left(\alpha_\rig, \alpha_\dR\right)$ defined above.
  \end{theorem}

  \begin{proof}
   An immediate calculation gives
   \[ \nabla\left(\alpha_\rig \right) = - N^k F^{(p)}_{k + 2, 0, b} v^{[0, k]} \otimes \xi, \]
   and we calculate that
   \[ (1 - \varphi)\left(\alpha_\dR\right) = -N^k F^{(p)}_{k+2, 0, b} v^{[0, k]} \otimes \xi,\]
   Thus the pair $(\alpha_\rig, \alpha_\dR)$ does define a class in syntomic cohomology, which maps to $\Eis^k_{\dR, b, N}$ in de Rham cohomology.

   For $k > 0$ this is enough to uniquely characterize the syntomic Eisenstein class, as the map from syntomic to de Rham cohomology is injective in this case (see \cite[Proposition 4.1]{BannaiKings}). For $k = 0$ this does not hold; but the motivic Eisenstein class $\Eis^0_{\mot, b, N}$ is just the Siegel unit $g_{0, b/N}$, and one knows that
   \[ r_{\syn}(g_{0, b/N}) = \left( (1 - \varphi) \log g_{0, b/N}, \operatorname{dlog} g_{0, b/N}\right).\]
   An easy series calculation shows that we have
   \[ \alpha_\rig = -F^{(p)}_{0, 1, b} = (1 - \varphi) \log g_{0, b/N}\]
   as rigid-analytic sections of the sheaf $\Qp(1)$ on $Y^{\ord}$, as required.
  \end{proof}


\section{Rankin--Eisenstein classes on products of modular curves}


 We shall now define ``Rankin--Eisenstein classes'' as the pushforward of Eisenstein classes along maps arising from the diagonal inclusion $Y_1(N) \into Y_1(N)^2$.

 \subsection{The Clebsch-Gordan map}
  \label{sect:clebschgordan}

  In this section, we'll establish some results on tensor products of the modules $\TSym^k H$. Let $H$ be any abelian group.

  Let $k,k',j$ be integers satisfying the inequalities
  \begin{equation}
   \label{eq:inequalities}
   k \ge 0, \quad k' \ge 0, \quad 0 \le j \le \min(k, k').
  \end{equation}
  By definition, we have
  \[ \TSym^{k + k' - 2j} H \subseteq \TSym^{k - j} H \otimes \TSym^{k' - j} H.\]
  On the other hand, the map
  \[ \wedge^2 H \to H \otimes H, \]
  given by mapping $x \wedge y$ to the antisymmetric tensor $x \otimes y - y \otimes x$, gives a map
  \[ \TSym^j\left(\wedge{}^2 H\right) \to \TSym^j H \otimes \TSym^j H\]
  by raising to the $j$-th power. Taking the tensor product of these two maps and using the multiplication in the tensor algebra $\TSym^\bullet H$, we obtain a map
  \[  CG^{[k, k', j]}: \TSym^{k + k - 2j}(H) \otimes \TSym^j\left(\wedge^2 H\right) \to \TSym^k\left(H\right) \otimes \TSym^{k'}\left(H\right).\]

  We are interested in the case where $H \cong \ZZ^2$, in which case $\wedge^2 H = \det(H)$.

  \begin{definition}
   Define
   \[ CG^{[k, k', j]} :  \TSym^{k + k' - 2j}(H) \to \TSym^{k }(H) \otimes \TSym^{k'}(H) \otimes \det(H)^{-j} \]
   to be the map defined by the above construction.
  \end{definition}

  We will need the following explicit formula for a piece of the Clebsch--Gordan map. Composing the Clebsch--Gordan map $CG^{[k, k', j]}$ with the natural contraction map
  \[ \left(\Sym^k H^\vee\right) \otimes \left(\Sym^{k'} H^\vee\right) \otimes \left( \TSym^k(H) \otimes \TSym^{k'}(H)\right) \to \ZZ \]
  gives a trilinear form
  \begin{equation}
   \label{eq:trilinear}
   \left(\Sym^k H^\vee\right) \otimes \left(\Sym^{k'} H^\vee\right) \otimes \left(\TSym^{k + k' - 2j} H\right) \to \det(H)^{-j}.
  \end{equation}

  Let us fix a basis $u, v$ of $H$ and write $w^{[r, s]} = u^{[r]} v^{[s]} \in \TSym^{r + s} H$. We let $u^\vee, v^\vee$ be the dual basis of $H^\vee$, and write $w^{(r, s)} = (u^\vee)^r (v^\vee)^s \in \Sym^{r + s} H^\vee$, so the bases $\{ w^{[r, s]} : r + s = k\}$ of $\TSym^k H$ and $\{ w^{(r, s)}: r + s = k\}$ of $\Sym^k H^\vee$ are dual to each other. We let $e_1$ be the basis $u \wedge v$ of $\det H$, and $e_j = e_1{}^{\otimes j}$.

  \begin{proposition}
   \label{prop:clebsch-gordan-coefficients}
   The trilinear form \eqref{eq:trilinear} sends the basis vector
   \[ w^{(0, k)} \otimes w^{(k', 0)} \otimes w^{[s, t]}\]
   to zero unless $(s, t) = (k' - j, k-j)$, in which case it is mapped to
   \[ \frac{k! (k')!}{j! (k-j)! (k' - j)!} \otimes e_{-j}.\]
  \end{proposition}

  \begin{proof}
   An unpleasant computation shows that for $0 \le s \le k + k' - 2j$, the Clebsch--Gordan map sends the basis vector $w^{[s, k + k' - 2j - s]}$ of $\TSym^{k + k' - 2j} H$ to the element
   \[ \sum_{r + r' = s} \sum_{i = 0}^j (-1)^i \frac{(r + i)!(k -r + i)!(r' + j - i)!(k' - r' + j-i)!}{r!(r')!(k-r-j)!(k'-r'-j)! i! (j-i)!} w^{[r + i, k-r-i]} \otimes w^{[r' + j - i, k' -j + i - r']} \otimes e_{-j}\]
   of $\TSym^k H \otimes \TSym^{k'} H \otimes \det(H)^{-j}$. The vector $w^{[r + i, k-r-i]} \otimes w^{[r' + j - i, k' -j + i - r']}$ pairs nontrivially with $w^{(0, k)} \otimes w^{(k', 0)}$ if and only if $r = i = 0$ and $r' = s = k' - j$, and substituting these values gives the formula claimed.
  \end{proof}

 \subsection{Geometric realization of the Clebsch--Gordan map}

  The constructions of the previous section can clearly also be carried out with sheaves of abelian groups; so for any of our cohomology theories $\cT \in \{ B, \dR, \et, \bar\et, \rig, \syn, \cH \}$ for which we have a well-behaved category of coefficients, and $E \to Y$ an elliptic curve with $E$ and $Y$ regular, we obtain Clebsch--Gordan maps
  \[
   CG^{[k, k', j]}_{\cT}: \TSym^{k + k' - 2j}\sH_\cT \to \TSym^{k} \sH_\cT\otimes \TSym^{k'} \sH_\cT(-j).
  \]

  This Clebsch-Gordan map can also realized geometrically using Lieberman's trick, as follows. Recall the group  $\mathfrak{T}_k = \mu_2^k \rtimes \mathfrak{S}_k$ and the character $\varepsilon_k$ from \S \ref{sect:lieberman} above. We saw that there are isomorphisms
  \[ H^{i}_{\cT}(Y, \TSym^k \sH_{\cT}(n)) \cong H^{i + k}_{\cT}(\cE^k, \QQ_{\cT}(k + n))(\varepsilon_k), \]
  and the same argument also gives isomorphisms
  \[ H^{i}_{\cT}(Y, \TSym^{k} \sH_\cT\otimes \TSym^{k'} \sH_\cT(n)) \cong H^{i + k + k'}_{\cT}(\cE^{k + k'}, \QQ_{\cT}(k + k' + n))(\varepsilon_k \times \varepsilon_{k'}),\]
  where we consider $\varepsilon_k \times \varepsilon_{k'}$ as a character of $\mathfrak{T}_k \times \mathfrak{T}_{k'} \subseteq \mathfrak{T}_{k + k'}$.

  \begin{lemma}
   For any $(k, k, j)$ satisfying the inqualities \eqref{eq:inequalities}, we can find a finite set of triples $(\lambda_t, \xi_t, \eta_t)$, where $\lambda_t \in \QQ$ and $\xi_t: \cE^{k + k' - j} \to \cE^{k + k' - 2j}$ and $\eta_t: \cE^{k + k' - 2j} \to \cE^{k + k'}$ are morphisms of $Y$-schemes, such that for any of the cohomology theories $\cT$, the map
   \[H^{k+k'-2j+i}_\cT(\cE^{k + k' - 2j},\QQ_\cT(k+k'-2j+n))
    \to H^{k+k'+i}_\cT(\cE^{k'},\QQ_\cT(k+k'-j+n))
   \]
   given by $\sum_t \lambda_t (\eta_t)_* \circ (\xi_t)^*$ sends the $\varepsilon_{k + k' - 2j}$-eigenspace to the $(\varepsilon_k \times \varepsilon_{k'})$-eigenspace, and coincides on these eigenspaces with the map
   \[  H^{i}_{\cT}(Y, \TSym^{k + k' - 2j} \sH_{\cT}(n)) \to H^{i}_{\cT}(Y, \TSym^{k} \sH_\cT\otimes \TSym^{k'} \sH_\cT(n-j))\]
   induced by $CG^{[k, k', j]}_{\cT}$.
  \end{lemma}

  \begin{proof}
   In all the cohomology theories $\cT$ we consider, for any morphism of $Y$-schemes $f: X \to X'$, there are relative pullback and (if $f$ is proper) pushforward morphisms between the sheaves on $Y$ obtained as the higher direct images of $\QQ_{\cT}$ along the structure maps of $X$ and $X'$. These are compatible with the absolute pullback and pushforward via the Leray spectral sequence. So it suffices to show that we may find $(\lambda_t, \xi_t, \eta_t)$ such that the sum of the \emph{relative} pushforward and pullback maps coincides with $CG^{[k, k', j]}_{\cT}$. This then gives the result of the proposition (for all values of $i$ and $n$ simultaneously).

   We consider first the extreme cases $j = 0$ and $k = k' = j$. In the former case, $CG^{[k, k', 0]}_{\cT}$ is just the natural inclusion $\TSym^{k + k'}(\sH_{\cT}) \subseteq \TSym^k(\sH_{\cT}) \otimes \TSym^{k'}(\sH_{\cT})$ (compatible with the inclusion of both sheaves into $(\sH_{\cT})^{\otimes (k + k')}$); so it is compatible with the natural inclusion of cohomology groups
   \[ H^{k+k'+i}_\cT(\cE^{k + k'},\QQ_\cT(k+k'+n))(\varepsilon_{k + k'}) \to H^{k+k'+i}_\cT(\cE^{k'},\QQ_\cT(k+k'+n))(\varepsilon_k \times \varepsilon_{k'}). \]

   For the case $k = k' = j$, we note that pullback along the structure map $\pi_j: \cE^j \to Y$, composed with pushforward along the diagonal inclusion $\delta_j: \cE^j \to \cE^{2j}$, gives a map of sheaves on $Y$ (a ``relative cycle class'')
   \[ \QQ_{\cT} \to R^{2j} (\pi_{2j})_* \QQ_{\cT}(j) \]
   where $\pi_{2j}$ is the projection $\cE^{2j} \to Y$. Projecting to the subsheaf on which all the $[-1]$ endomorphisms on the fibres act as $-1$, we obtain a map
   \[ \QQ_{\cT} \to \left(R^{1} \pi_* \QQ_{\cT}\right)^{\otimes 2j}(j) = \sH_{\cT}^{\otimes 2j}(-j). \]
   Projecting to the direct summand $\TSym^{[j, j]} \sH_{\cT}(-j)$ gives a geometric realization of $CG^{[j, j, j]}_{\cT}$, which is given concretely as a formal linear combination of translates of $(\delta_j)_* (\pi_j)^*$ by elements of the group $\frT_j \times \frT_j$.

   For a general $(k, k', j)$, we take the product of the above maps for the triples $(j, j, j)$ and $(k-j, k'-j, 0)$, and average over the cosets of $(\mathfrak{T}_j \times \mathfrak{T}_{k-j}) \times (\mathfrak{T}_j \times \mathfrak{T}_{k'-j})$ in $\mathfrak{T}_k \times \mathfrak{T}_{k'}$. Since the map $CG^{[k, k', j]}_{\cT}$ is likewise built up from $CG_{\cT}^{[j, j, j]}$ and $CG_{\cT}^{[k-j, k'-j, 0]}$ via the symmetrized tensor product, this gives the required compatibility.
  \end{proof}

  \begin{corollary}
   There exists a morphism
   \[ CG^{[k, k', j]}_{\mot}: H^{i}_{\mot}(Y, \TSym^{k + k' - 2j} \sH_{\QQ}(n)) \to H^{i}_{\mot}(Y, \TSym^{k} \sH_\QQ\otimes \TSym^{k'} \sH_\QQ(n-j)) \]
   compatible with the maps $CG^{[k, k', j]}_{\cT}$, for $\cT \in \{ B, \dR, \et, \bar\et, \rig, \syn, \cH \}$, under the regulator maps $r_{\cT}$.
  \end{corollary}

  \begin{proof}
   We simply define $CG^{[k, k', j]}_{\mot}$ to be the map $\sum_t \lambda_t (\eta_t)_* \circ (\xi_t)^*$, which is well-defined since motivic cohomology with $\QQ(n)$ coefficients has pullback and proper pushforward maps, compatible with the other theories via the maps $r_{\cT}$.
  \end{proof}
With these constructions we can extend the Gysin map of \eqref{eq:gysin-map} to the motivic setting. If $Y$ is a smooth $T$-scheme of relative dimension $d$, we consider the pushforward along the closed embedding $\cE^{k+k'}=\cE^{k}\times_Y\cE^{k'}$ into $\cE^{k}\times_T\cE^{k'}$. One gets
\[ 
\Delta_*: H^{i}_{\mot}(Y, \TSym^{k} \sH_\QQ\otimes \TSym^{k'}\sH_{\QQ}(n)) \to
 H^{i+2d}_{\mot}(Y\times_TY, \TSym^{[k+k']} \sH_{\QQ}(n+d)),
\]
where $\TSym^{[k+k']} \sH_{\QQ}$ is as in Definition \ref{def:Sym-r-s}.


 \subsection{Rankin--Eisenstein classes}

  We now come to the case which interests us: we consider the scheme $S = Y_1(N)$ over $T = \Spec \ZZ[1/N]$.

  \begin{definition}\label{def:Rankin-Eisenstein-class}
   For $k, k', j$ satisfying the inequalities \eqref{eq:inequalities}, and $\cT \in \{\mot, \et, \cH, \syn\}$, we define
   \[
    \Eis^{[k, k', j]}_{\cT, b, N} \coloneqq \left(\Delta_* \circ CG^{[k, k', j]}_{\cT}\right)\left(\Eis^{k + k' - 2j}_{\cT, b, N}\right)\\
    \in H^3_{\cT}(Y_1(N)^2, \TSym^{[k, k']}(\sH_{\cT})(2-j)).
   \]
  \end{definition}

  \begin{remark}
   The classes $\Eis^{[k, k', j]}_{\cT, b, N}$ can also be defined for the ``geometric'' theories $\cT = \{\bar\et, \dR, B, \rig\}$, but these are automatically zero, since $Y_1(N)^2$ is an affine surface, and thus its $H^3$ vanishes.
  \end{remark}


 \subsection{Abel--Jacobi maps}
  \label{sect:ajdefs}

  Let $f$, $g$ be newforms of weights $k + 2, k' + 2$ and levels $N_f, N_g$ dividing $N$. The aim of this section is the construction of Abel--Jacobi maps $\AJ_{\cT, f,g}$, for $\cT\in \{\cH, \syn, \et\}$, which we will use to interpret the Rankin--Eisenstein classes as linear functionals on de Rham cohomology.

  \subsubsection*{Absolute Hodge cohomology} Since $Y_1(N)^2$ is a smooth affine variety of dimension 2, its de Rham (or Betti) cohomology vanishes in degrees $\ge 3$. Consequently, the exact sequence \eqref{eq:deligneexactseq} for absolute Hodge cohomology gives an isomorphism
  \[
   H^3_{\cH}(Y_1(N)^2_{\RR}, \TSym^{[k, k']}(\sH_\RR)(2-j)) \cong
   H^1_{\cH}\left(\Spec \RR, H^2_B(Y_1(N)^2_{\RR}, \TSym^{[k, k']}(\sH_\RR)(2-j))\right)
  \]
  for any $j$. Via projection to the $(f, g)$-isotypical component we obtain a natural map
  \begin{equation*}
   H^3_{\cH}(Y_1(N)^2_{\RR}, \TSym^{[k, k']}(\sH_\RR)(2-j)) \rTo H^1_{\cH}(\Spec \RR, M_B(f\otimes g)^*(-j)).
  \end{equation*}
  The comparison isomorphisms between Betti and de Rham cohomology
  induce two \emph{period maps}
  \begin{align*}
   \alpha_{M(f\otimes g)(n)}: M_B(f\otimes g)(n)_{\RR}^+ &\to  t(M(f\otimes g)(n))_\RR\\
\alpha_{M(f\otimes g)^{*}(n)}: M_B(f\otimes g)^{*}(n)_{\RR}^+ &\to  t(M(f\otimes g)^{*}(n))_\RR
  \end{align*}
with the tangent spaces from \ref{def:tangentspace}
  The $L$-vector spaces $M_B(f\otimes g)(n)^+$ and
  $M_B(f\otimes g)^{*}(n)^+$  have dimension $2$ and in the case that $0\le j\le \min\{k,k'\}$ we get that $\ker(\alpha_{M(f\otimes g)(j+1)})$
  is one-dimensional. Poincar\'e duality induces a perfect pairing
\[
\ker(\alpha_{M(f\otimes g)(j+1)})\times \coker(\alpha_{M(f\otimes g)^{*}(-j)})\to L\otimes_\QQ \RR
\]
which together with the isomorphism coming from sequence \eqref{eq:absolute-Hodge-seq}
\[
H^1_\cH(\Spec\RR, M_B(f\otimes g)^*(-j))\isom \coker(\alpha_{M(f\otimes g)^{*}(-j)})
\]
induces the isomorphism
  \[
  H^1_\cH(\Spec\RR, M_B(f\otimes g)^*(-j))\isom
\ker(\alpha_{M(f\otimes g)(j+1)})^{*}.
  \]
Putting these isomorphisms together,
we obtain the \emph{Abel--Jacobi map} for absolute Hodge cohomology,
  \[ \AJ_{\cH, f, g}: H^3_{\cH}(Y_1(N)^2_{\RR}, \TSym^{[k, k']}(\sH_\RR)(2-j)) \rTo (\ker \alpha_{M(f\otimes g)(1+j)})^*.\]
  Taking duals we get
  an exact sequence
  \begin{equation}
     \label{eq:deligne-cohomology}
   0\to \Fil^{-j}M_\dR(f\otimes g)^*_\RR \to M_B(f\otimes g)^*(-j-1)_\RR^+ \to
     H^1_\cH(\Spec\RR, M_B(f\otimes g)^*(-j)_\RR)\to 0,
  \end{equation}
  which is crucial for the interpretation of the leading terms of $L(f,g,s)$ at
  $s=j+1$ in the Beilinson conjecture.

  \subsubsection*{Syntomic cohomology}

  Similarly, for a prime $p \nmid N$, the exact sequence \eqref{eq:syntomicexactseq} for syntomic cohomology gives isomorphisms
  \[ H^3_{\syn}(Y_1(N)^2_{\Zp}, \TSym^{[k, k']}(\sH_\Qp)(2-j) \cong H^1_\syn\left(\Spec \Zp, H^2_\rig(Y_1(N)^2_{\Zp}, \TSym^{[k, k']}(\sH_\rig)(2-j))\right), \]
  and projecting to the $(f, g)$-isotypical component we obtain a map
  \begin{multline*}
   H^3_{\syn}(Y_1(N)^2_{\Zp}, \TSym^{[k, k']}(\sH_\Qp)(2-j)) \rTo H^1_\syn\left(\Spec \Zp, M_{\rig}(f \otimes g)^*(-j)\right) \\
   = \frac{M_{\rig}(f\otimes g)^*(-j)_{\Qp}}{(1 - \varphi) \Fil^0 M_{\dR}(f\otimes g)^*(-j)_{\Qp}}.
  \end{multline*}
  Since $1 - \varphi$ is an isomorphism on $M_{\rig}(f\otimes g)^*(-j)_{\Qp}$ for $0 \le j \le \min(k, k')$, the right-hand side can be identified with $t(M(f\otimes g)^*(-j))_{\Qp}$, which is free of rank 3 over $L \otimes \Qp$. This gives the \emph{syntomic Abel--Jacobi map}
  \[  \AJ_{\syn, f, g}: H^3_{\syn}(Y_1(N)^2_{\Zp}, \TSym^{[k, k']}(\sH_\Qp)(2-j)) \rTo t(M(f\otimes g)^*(-j))_{\Qp}.\]

  \subsubsection*{\'Etale cohomology} The \'etale spectral sequence ${}^{\et} E^{ij}$ for $Y_1(N)^2_{\Qp}$ degenerates at $E_3$ (since $H^i(\Qp, -)$ is the the zero functor for $i \ne \{0, 1, 2\}$), so we obtain a natural map (not an isomorphism in general)
  \[ H^3_{\et}(Y_1(N)^2_{\Qp},  \TSym^{[k, k']}(\sH_\Qp)(2-j)) \to H^1_{\et}\left(\Spec \Qp, H^2_{\bar\et}(Y_1(N)^2, \TSym^{[k, k']}(\sH_\Qp)(2-j))\right).\]
  Projecting to the $(f, g)$-isotypical component gives an \'etale Abel--Jacobi map
  \[
   \AJ_{\et, f, g}: H^3_{\et}(Y_1(N)^2_{\Qp},  \TSym^{[k, k']}(\sH_\Qp)(2-j))
   \rTo H^1(\Qp, M_{\overline{\et}}(f \otimes g)^*(-j)).
  \]

  The following proposition gives a relation between the syntomic and \'etale Abel--Jacobi maps:

  \begin{proposition}
   The maps $\AJ_{\et, f, g}$ and $\AJ_{\syn, f, g}$ fit into a commutative diagram
   \begin{diagram}
    H^3_{\mot}(Y_1(N)^2_{\Zp}, \TSym^{[k, k']}(\sH_\Qp)(2-j)) & \rTo & H^3_{\mot}(Y_1(N)^2_{\Qp}, \TSym^{[k, k']}(\sH_\Qp)(2-j)) \\
    \dTo^{r_{\syn}} & & \dTo^{r_{\et}}\\
    H^3_{\syn}(Y_1(N)^2_{\Zp}, \TSym^{[k, k']}(\sH_\Qp)(2-j)) & & H^3_{\et}(Y_1(N)^2_{\Qp},  \TSym^{[k, k']}(\sH_\Qp)(2-j))\\
    \dTo^{\AJ_{\syn, f, g}} & & \dTo_{\AJ_{\et, f, g}} \\
    t(M(f\otimes g)^*(-j))_{\Qp} & \rTo^{\exp \circ \comp_{\dR}} & H^1(\Qp, M_{\overline{\et}}(f \otimes g)^*(-j)).
   \end{diagram}
   where the top horizontal arrow is given by base-extension, and in the bottom horizontal arrow,
   \[ \comp_{\dR}: M_{\dR}(f\otimes g)^*_{\Qp} \cong \DD_{\dR}(M_{\overline{\et}}(f \otimes g)^*)\]
   is the Faltings comparison isomorphism, and $\exp$ is the Bloch--Kato exponential map (c.f. Section \ref{section:etsyncompatible}).
  \end{proposition}

  \begin{proof}
   This follows from Theorem \ref{thm:absolutecomparison}(1) applied to the variety $X = \cE^k \times \cE^{k'}$, together with the observation that the $(f, g)$-eigenspaces in de Rham, syntomic, and \'etale cohomology all lift isomorphically to the cohomology of the product of Kuga--Sato varieties $\overline{\cE}^k \times \overline{\cE}^{k'}$, which is projective, so we may apply Theorem \ref{thm:absolutecomparison}(2) (with $X = \overline{\cE}^k \times \overline{\cE}^{k'}$).
  \end{proof}

  We shall give formulae for the images of the Eisenstein class under $\AJ_{\cH, f, g}$ and $\AJ_{\syn, f, g}$ in the following section, and by the preceding proposition, the latter will also give a formula for $\AJ_{\et, f, g}$. (We will not use $\AJ_{\et, f, g}$ directly in the present paper, but it will play a central role in the sequel.)


\section{Regulator formulae}


 \subsection{Differentials and rationality}
  \label{sect:periods}

  Recall that we have fixed newforms $f$, $g$ of levels $N_f, N_g$ and weights $k + 2, k' + 2$, whose $q$-expansions have coefficients in some number field $L$. We let $\omega_f \in H^0(X_1(N)(\CC), \Sym^k \sH_{\CC}^\vee \otimes \Omega^1) \otimes_{\QQ} L$ denote the holomorphic $(\Sym^k \sH_{\CC}^\vee)$-valued differential whose pullback to the upper half-plane is given by
  \[ \omega_f = (2 \pi i)^{k + 1} f(\tau)\, w^{(k, 0)}\, \mathrm{d}\tau = (2 \pi i)^{k} f(\tau)\, w^{(k, 0)} \tfrac{\mathrm{d}q}{q}, \]
  where as in \S \ref{sect:DeligneEis}, we have $w^{(k, 0)} = (\mathrm{d}z)^k$, for $\mathrm{d}z$ the standard basis of $\Fil^1 \sH_{\CC}^\vee$.

  With our conventions, the field of definition of the differential $\omega_f$ is not necessarily the same as the field $L$ generated by the $q$-expansion coefficients of $f$. As shown by the following lemma, the translation between the two different $L$-structures is given by multiplying by a Gauss sum:

  \begin{lemma}
   \label{lemma:gauss-sum}
   Let $f = \sum_{n>0} a_n q^n$ with $a_n \in L$ and $\varepsilon_f$ the associated character.
   We denote by $N_{\varepsilon_f}$ its conductor.
   Let
   \[
    G(\varepsilon_f^{-1}) \coloneqq \sum_{x\in \ZZ/N_{\varepsilon_f}\ZZ}\varepsilon_f^{-1}(x)e^{2\pi ix/N_{\varepsilon_f}}
   \]
   be the \emph{Gauss sum} of $\varepsilon_f^{-1}$. Then the differential
   \[ \omega'_f  \coloneqq G(\varepsilon_f^{-1})\omega_f \]
   is a section of $\Sym^k \sH_{\dR}^\vee$ over $L$, and its class in de Rham cohomology is a basis of the 1-dimensional $L$-vector space $\Fil^1 M_\dR(f)$.
  \end{lemma}

  \begin{proof}
   Note that the cusp $\infty$ is not defined over $\QQ$. By the $q$-expansion principle, we have to check that for any $\sigma_d\in \Gal(\QQ(\zeta_N)/\QQ) \isom (\ZZ/N\ZZ)^*$, we have $\sigma_d(a_n)=a_n$. But the action of $\sigma_d$ is given by ${\stbt 1 0 0 d}^*f=\varepsilon_f(d)f$, which implies that the coefficients $a_n$ have to be in the $\varepsilon$-eigenspace of $L \otimes_{\QQ} \QQ(\zeta_N)$, which is generated over $L$ by $G(\varepsilon_f^{-1})$.
  \end{proof}

  We also want to have a basis of the space $M(f) / \Fil^1$, for which we need to introduce the form conjugate to $f$:

  \begin{definition}
   \label{def:f-star}
   We define $f^* \in S_{k+2}(\Gamma_1(N), \CC)$ to be the
   form with $q$-expansion $\sum_{n > 0} \overline{a_n} q^n$, so that
   \[ f^*(\tau) = \overline{f(-\overline{\tau})}. \]
  \end{definition}

  The class of the $C^\infty$ differential form
  \[ \overline{\omega}_{f^*} = (-2\pi i)^{k + 1} f(-\bar\tau)\, w^{(0, k)} \mathrm{d}\bar\tau\]
  has the same Hecke eigenvalues as $f$, so it defines an element of $M_{\dR}(f) \otimes_{\QQ} \CC$. In order to renormalise this appropriately, we introduce a pairing on $\Sym^k \sH_{\CC}^\vee$ as the composite
  \begin{equation}
   \label{eq:symk-pairing} 
   \Sym^k \sH^\vee_{\CC} \otimes \Sym^k \sH^\vee_{\CC} \rTo \TSym^k \sH_{\CC}^\vee \otimes \Sym^k \sH_{\CC}^\vee \rTo^{\tfrac{1}{k!}\langle-,-\rangle} \CC(-k),
  \end{equation}
  where the first map is the canonical inclusion $\Sym^k \into \TSym^k$ and the second map is given by the isomorphisms $\TSym^k(\sH^\vee) = (\Sym^k \sH)^\vee$ and $\sH^\vee \cong \sH(-1)$. The presence of the factor $\tfrac{1}{k!}$ implies that we have $\langle x^k, y^k\rangle = \langle x, y \rangle^k$ for sections $x, y$ of $\sH^\vee_{\CC}$.
  
  Via Poincar\'e duality, this pairing on the coefficients induces a pairing on $H^1_{\dR, c}(Y_1(N_f), \Sym^k \sH^\vee_{\CC})$, which we denote by $\langle-,-\rangle_{Y_1(N_f)}$.
  
  \begin{proposition}\label{prop:eta-prime}
   The class modulo $\Fil^1$ of the element
   \[ \frac{G(\varepsilon_f^{-1})}{\langle \omega_f, \bar\omega_f \rangle_{Y_1(N_f)}}\, \overline{\omega}_{f^*}\]
   is non-zero and defined over $L$, and thus defines a basis $\eta_f'$ of the $L$-vector space $\frac{M_{\dR}(f)}{\Fil^1 M_{\dR}(f)}$.
  \end{proposition}

  \begin{proof}
   It suffices to check that this differential pairs to an element of $L$ with the basis vector  $\omega'_{f^*} = G(\varepsilon_f) \omega_{f^*}$ of $\Fil^1 M_{\dR}(f^*)$. Since $\langle \omega_{f^*}, \bar\omega_{f^*}\rangle_{Y_1(N_f)} = \langle \omega_{f}, \bar\omega_{f}\rangle_{Y_1(N_f)}$, this pairing evaluates to $G(\varepsilon_f) G(\varepsilon_f^{-1}) = (-1)^k N_{\varepsilon_f} \in L$, as required.
  \end{proof}

  \begin{remark} \
   \begin{enumerate}
    \item The constant $\langle \omega_f, \bar\omega_f \rangle_{Y_1(N_f)}$ is equal to $(-4\pi)^{k + 1} \|f\|$, where
    \[ \|f\| = \int_{\Gamma_1(N_f) \backslash \mathfrak{H}} |f(x + iy)|^2 y^k\, \mathrm{d}x\, \mathrm{d}y\]
    is the Petersson norm of $f$.
    \item In \cite{LLZ14} we worked directly with the classes $\omega_f$ and $\eta_f = \frac{1}{\langle \omega_{f}, \bar\omega_{f}\rangle} \bar{\omega}_{f^*}$. However, these classes are not defined over $L$ but only over $L \otimes \QQ(\mu_N)$ for a suitable $N$; so one has to extend scalars to this field in order to evaluate the regulators. (This base-extension was inadvertently omitted from the statement of Theorem 5.4.6 of \emph{op.cit.}.) In the present paper we want to verify the conjectures of Beilinson and Perrin-Riou, which predict the values of regulators up to a factor in $L^\times$, so it is more convenient to work with the $L$-rational classes $\eta_f'$ and $\omega_f'$.
   \end{enumerate}
  \end{remark}

 \subsection{The regulator formula in absolute Hodge cohomology}
  \label{sect:delignereg}

  Let $0\le j\le \min\{k,k'\}$ and write $Y \coloneqq Y_1(N)_{\RR}$ and $Y^2 \coloneqq Y_1(N)_{\RR} \times Y_1(N)_{\RR}$. As in \S \ref{sect:ajdefs}, let $f$, $g$ be newforms of weights $k + 2, k' + 2$ and levels $N_f, N_g$ dividing $N$. In this section we want to  give a formula for the image of $\Eis_{\cH,b,N}^{[k,k',j]}$ under the Abel--Jacobi map
  \[
   \AJ_{\cH, f, g}:
   H^3_\cH(Y^2_\RR,\TSym^{[k,k']}\sH_\RR(2-j) ) \rTo
   (\ker \alpha_{M(f\otimes g)(1+j)})^*.
  \]
  By \ref{eq:deligne-cohomology} one has an exact sequence
  \[
   0 \to F^{-j}M_\dR(f\otimes g)^*_\RR \to M_B(f\otimes g)^*(-j-1)_\RR^+
   \to (\ker \alpha_{M(f\otimes g)(1+j)})^* \to 0
  \]
  and we will compute a representative for $\Eis_{\cH,b,N}^{[k,k',j]}$ in $ M_B(f\otimes g)^*(-j-1)_\RR^+$. To compute the projection $\pr_{f,g}$ we use the perfect pairing
  \[
   \langle\quad ,\quad\rangle_{Y^{2}}:
   M_B({f}\otimes {g})^*_\CC/F^{-j}
   \times F^{1+j}M_B(f\otimes g)_\CC\to \CC\otimes_\QQ L.
  \]

  \begin{remark}
   This pairing is in fact the product of the pairings
   \begin{align*}
    \langle \quad,\quad\rangle_{Y}:& M_B({f})^*_\CC\times M_B(f)_\CC\to \CC\otimes_\QQ L,\\
    \langle\quad,\quad\rangle_{Y}:& M_B({g})^*_\CC\times M_B(g)_\CC\to \CC\otimes_\QQ L.
   \end{align*}
  \end{remark}

  We want to give a formula for the pairing of
  \[
   \Eis^{[k, k', j]}_{\cH, b, N} \coloneqq \left(\Delta_* \circ CG^{[k, k', j]}_{\cH}\right)\left(\Eis^{k + k' - 2j}_{\cH, b, N}\right)
  \]
  with a class in $ F^{1+j}M_B(f\otimes g)_\CC^+$ which reduces the computation to an integral on $Y(\CC)$. For this we consider $\Eis^{[k, k', j]}_{\cH, b, N}$ as a class in $H^2_B(Y^2(\CC),\TSym^{[k,k']}\sH_\CC)^+/F^{2-j}$. Consider the pull-back $\Delta^*$ composed with the dual of $CG^{[k, k', j]}_B$,
  \[
   ^\vee CG^{[k, k', j]}_B\circ\Delta^*:
   F^{1+j}H^2_{c,B}(Y^2(\CC),\Sym^{(k,k')}\sH_\CC^\vee)\to
   F^{1+j}H^2_{c,B}(Y(\CC),\Sym^{k+k'-2j}\sH_\CC^\vee).
  \]
  Recall from Proposition \ref{prop:deligne-eisenstein} that $\Eis^{k + k' - 2j}_{\cH, b, N}$ is represented by a pair of forms $(\alpha_\infty,\alpha_\dR)$. Then we have:

  \begin{proposition}\label{prop:integral-formula}
   For any cohomology class $[\omega_c]\in F^{1+j}M_\dR(f\otimes g)_\RR$ one has the formula
   \[
    \langle \Eis^{[k, k', j]}_{\cH, b, N},\omega_c\rangle_{Y^{2}}=
    \frac{1}{2\pi i}\int_{Y(\CC)}{^\vee C}G^{[k, k', j]}_B\circ\Delta^*(\omega_c)\wedge\alpha_\infty.
   \]
  \end{proposition}

  \begin{remark}
   Compare \cite[Theorem 4.3.1]{LLZ14}, which is the case of trivial coefficients. Note that the proof will actually show that the integral is well defined, i.e., does not depend on the choice of a differential $\omega_c$ representing the class $[\omega_c]$.
  \end{remark}

  \begin{proof}
   The push-forward along the diagonal $\Delta$ is defined via Deligne homology (see \cite{Jannsen-Hodge} for the definitions). In fact one has by general properties for a Bloch-Ogus cohomology theory
   \[
    H^1_\cH(Y_\RR,\TSym^{k+k'-2j}\sH_\RR(1))\isom
    H^\cH_1(Y_\RR,\Sym^{k+k'-2j}\sH_\RR^\vee)
   \]
   and
   \[
    H^3_\cH(Y^2_\RR,\TSym^{[k,k']}\sH_\RR(2-j) )\isom
    H^\cH_1(Y^2_\RR,\Sym^{(k,k')}\sH_\RR^\vee(j)).
   \]
   With these isomorphisms the map $\Delta_* \circ CG^{[k, k', j]}_{\cH}$ is just the functoriality for the homology combined with the Clebsch-Gordan map. These homology groups have an interpretation in terms of currents. Let $T_{\alpha_\infty}$ and $T_{\alpha_\dR}$ be the currents associated    to $\alpha_\infty$ and $\alpha_\dR$. As $H^3_\dR(Y^2_\RR,\TSym^{[k,k']}\sH_\RR )=0$ the current   $\Delta_* \circ CG^{[k, k', j]}_{\cH}(T_{\alpha_\dR})$ is a trivial cohomology class, so that there exists a current $\rho$ (with logarithmic singularities) with $d\rho=\Delta_* \circ CG^{[k, k', j]}_{\cH}(T_{\alpha_\dR})$. It follows that $\Delta_* \circ CG^{[k, k', j]}_{\cH}(T_{\alpha_\infty})-\rho$ defines a cohomology class, which represents a lift of $\Eis^{[k, k', j]}_{\cH, b, N}$ to $H^2_B(Y^2(\CC),\TSym^{[k,k']}\sH_\CC)$. This gives
   \[
    \langle \Eis^{[k, k', j]}_{\cH, b, N},\omega_c\rangle_{Y^{2}}=
    (\Delta_* \circ CG^{[k, k', j]}_{\cH}(T_{\alpha_\infty})-\rho)(\omega_c).
   \]
   By construction $\Delta_* \circ CG^{[k, k', j]}_{\cH}(T_{\alpha_\dR})$ is a current in the zeroth step of the Hodge filtration (see \cite[1.4]{Jannsen-Hodge}) so that also $\rho$ is in $F^0$. As $\omega_c$ is in $F^{1+j}M_B(f\otimes g)_\CC$ the evaluation $\rho(\omega_c)$ is in $F^{1+j}$ and hence zero as $1+j>0$. This gives
   \[
    (\Delta_* \circ CG^{[k, k', j]}_{\cH}(T_{\alpha_\infty})-\rho)(\omega_c)=
    \Delta_* \circ CG^{[k, k', j]}_{\cH}(T_{\alpha_\infty})(\omega_c)
    =
    T_{\alpha_\infty}(^\vee CG^{[k, k', j]}_B\circ\Delta^*(\omega_c))
   \]
   where the last equality is the definition of the push-forward. Finally, by definition of $T_{\alpha_\infty}$, we get
   \[
    T_{\alpha_\infty}(^\vee CG^{[k, k', j]}_B\circ\Delta^*(\omega_c))
    =
    \frac{1}{2\pi i}\int_{Y(\CC)}{^\vee C}G^{[k, k', j]}_B\circ\Delta^*(\omega_c)\wedge\alpha_\infty.
   \]
   As all these equalities hold for any closed form $\omega_c$ in the $1+j$-step of the Hodge filtration and because $\langle \Eis^{[k, k', j]}_{\cH, b, N},\omega_c\rangle_B$ is independent of the representative of $[\omega_c]$, the same is true for the integral.
  \end{proof}

  For the explicit computations we use the bases
  \[
   \{\omega_f\otimes \omega_g,\omega_f\otimes \bar{\omega}_{{g^*}},\bar{\omega}_{{f^*}}\otimes \omega_g,\bar{\omega}_{{f^*}}\otimes \bar{\omega}_{{g^*}}\}\quad \text{and} \quad
   \{\omega_{f^*}\otimes \omega_{g^*},\omega_{f^*}\otimes \bar{\omega}_{{g}},\bar{\omega}_{{f}}\otimes \omega_{ g^*},\bar{\omega}_{{f}}\otimes \bar{\omega}_{{g}}\}
  \]
  of $M_B(f\otimes g)_\CC$ and $M_B(f^*\otimes g^*)_\CC$ respectively.

  \begin{remark}
   The natural map $M_{B}(f^* \otimes g^*)(k + k' + 2) \to M_{B}(f \otimes g)^*$ is an isomorphism, so we may interpret the latter vectors as a basis of $M_{B}(f \otimes g)^*_{\CC}$.
  \end{remark}

  Note that one has $\overline{F}_\infty^*(\omega_{f^*}\otimes \omega_{g^*})
  =\bar{\omega}_{{f}}\otimes \bar{\omega}_{{g}}$ and
  $\overline{F}_\infty^*(\omega_{f^*}\otimes \bar{\omega}_{{g}})=\bar{\omega}_{{f}}\otimes \omega_{ g^*}$ so that $\frac{1}{2}(\omega_{f^*}\otimes \bar{\omega}_{{g}}+(-1)^{-j-1}\bar{\omega}_{{f}}\otimes \omega_{ g^*})$ is a basis of
  \[
   M_B(f\otimes g)^*(-j-1)_\RR^+ /F^{-j}M_\dR(f\otimes g)^*_\RR\isom (\ker \alpha_{M(f\otimes g)(1+j)})^*.
  \]

  \begin{lemma}
   The element
   \[
   \frac{-1}{\langle\omega_f,\bar{\omega}_f\rangle_Y \langle\omega_g,\bar{\omega}_g \rangle_Y}
   \left(\bar\omega_{f^*}\otimes {\omega}_g+(-1)^{j+1}\omega_f\otimes\bar \omega_{g^{*}} \right)
   \]
   in $\ker \alpha_{M(f\otimes g)(1+j)}$ is the dual basis of
    $\frac{1}{2}(\omega_{f^*}\otimes \bar{\omega}_{{g}}+(-1)^{-j-1}\bar{\omega}_{{f}}\otimes \omega_{ g^*})$.
  \end{lemma}

  \begin{proof}
   This follows from the formulae
   \begin{align*}
    \langle \bar\omega_{f^*}\otimes {\omega}_g,{\omega}_{f^*} \otimes \bar{\omega}_{ g}\rangle_{Y}&=
    -{\langle\omega_{f^*}, \bar{\omega}_{f^*}\rangle_Y\langle\omega_g, \bar{\omega}_g\rangle_Y}\\
    \langle {\omega}_f\otimes \bar{\omega}_{ g^*},\bar{\omega}_f \otimes  {\omega}_{g^*}\rangle_{Y}&=
    -{\langle\omega_f, \bar{\omega}_f\rangle_Y\langle\omega_{g^*}, \bar{\omega}_{g^*}\rangle_Y},
   \end{align*}
   as $\langle\omega_{f^*}, \bar{\omega}_{f^*}\rangle_Y=\langle\omega_f, \bar{\omega}_f\rangle_Y$ and
   $\langle\omega_g, \bar{\omega}_g\rangle_Y=\langle\omega_{g^*}, \bar{\omega}_{g^*}\rangle_Y$.
  \end{proof}

  \begin{definition}
   We write
    \[
    \bar{\omega}_{{f^*}}\wedge \omega_g\wedge
         \alpha_\infty
    \]
    for the form on $Y(\CC)$ obtained from the form
    $ {^\vee C}G^{[k, k', j]}_B\circ\Delta^*(\bar{\omega}_{{f^*}}\otimes \omega_g)\wedge\alpha_\infty$ by using the evaluation
    pairing  $\TSym^{k+k'-2j}\sH_\CC\otimes\Sym^{k+k'-2j}\sH^\vee_\CC\to \CC$
and similar for ${\omega}_f\wedge\bar{ \omega}_{{g^*}}\wedge \alpha_\infty$.
  \end{definition}

  \begin{proposition}
   \label{prop:deligne-regulator-formula}
   Let $0\le j\le\min\{k,k'\}$ and $\Eis_{\cH,b,N}^{k+k'-2j}=(\alpha_\infty,\alpha_\dR)$. Then
   \begin{multline*}
    \left\langle\AJ_{\cH,f,g}\left(\Eis_{\cH,b,N}^{[k,k',j]}\right),
    \frac{-1}{\langle\omega_f,\bar{\omega}_f\rangle_Y \langle\omega_g,\bar{\omega}_g \rangle_Y}
    \left(\bar\omega_{f^*}\otimes {\omega}_g+(-1)^{j+1}\omega_f\otimes\bar \omega_{g^{*}} \right)
    \right\rangle=\\
    \frac{-1}{(2\pi i)\langle\omega_f,\bar{\omega}_f\rangle_Y \langle\omega_g,\bar{\omega}_g \rangle_Y}
        \int_{Y(\CC)}(\bar{\omega}_{{f^*}}\wedge \omega_g
    +(-1)^{j+1} {\omega}_f\wedge\bar{ \omega}_{{g^*}})\wedge
    \alpha_\infty.
   \end{multline*}
  \end{proposition}

  \begin{proof}
   We have to compute the pairing of $\AJ_{\cH,f,g}\left(\Eis_{\cH,b,N}^{[k,k',j]}\right)$
   with
   \[
   \frac{-1}{\langle\omega_f,\bar{\omega}_f\rangle_Y \langle\omega_g,\bar{\omega}_g \rangle_Y}
   \left(\bar\omega_{f^*}\otimes {\omega}_g+(-1)^{j+1}\omega_f\otimes\bar \omega_{g^{*}} \right)
   \]
   which by Proposition \ref{prop:integral-formula} reduces to the
   integral in the proposition.
  \end{proof}

  We will compute the integral
  appearing in Proposition \ref{prop:deligne-regulator-formula} in terms of special values of the $L$-function $L(f,g,s)$. We write
  \[
   \omega_{f^*}={f^*}(2\pi i)^kw^{(k,0)}\frac{dq}{q} \quad
   \text{and}\quad \omega_g=g(2\pi i)^{k'} w^{(k',0)}\frac{dq}{q}
  \]
  so that
  \[\bar\omega_{f^*}\wedge \omega_g=(-1)^{k+1}\bar f^*g(2\pi i)^{k+k'+2} w^{(0, k)} \otimes w^{(k', 0)} d\tau d\bar\tau.\]
  Recall from Proposition
  \ref{prop:deligne-eisenstein} the formula
  \begin{align*}
   \alpha_{\infty}&\coloneqq \frac{-N^{k+k'-2j}}{2}\times\\
   &\sum_{m=0}^{k+k'-2j}(-1)^{m}(k+k'-2j-m)!(2\pi i)^{m-k-k'+2j}(\tau-\overline{\tau})^{m}F^{\an}_{2m-k-k'+2j, k+k'-2j+1-m, b} w^{[k+k'-2j-m,m]}
  \end{align*}

  \begin{proposition}
   One has
   \begin{align*}
    \bar\omega_{{f^*}}\wedge {\omega}_g\wedge \alpha_\infty&=
    \frac{(-1)^{j} N^{k+k'-2j}}{2} \binom{k}{j} k'!(2\pi i)^{k+2}(\tau-\bar \tau)^k
    \bar f^* gF^\an_{k-k',k'-j+1,b}d\tau d\bar \tau
    \\
    \omega_{{f}}\wedge \bar{\omega}_{g^*}\wedge\alpha_\infty&=
     \frac{(-1)^{k+k'+1}N^{k+k'-2j}}{2}\binom{k'}{j}k!(2\pi i)^{k'+2}(\tau-\bar \tau)^{k'}
         f \bar g^*F^\an_{k'-k,k-j+1,b}d\tau d\bar \tau
   \end{align*}
  \end{proposition}

  \begin{proof}
   From Proposition \ref{prop:clebsch-gordan-coefficients} one sees that
   \[
   {^\vee C}G^{[k, k', j]}_B(w^{(0,k)}\otimes w^{(k',0)})\wedge w^{[s,t]}=w^{(0,k)}\otimes w^{(k',0)}
   \wedge CG^{[k, k', j]}_B(w^{[s,t]})
   \]
   is zero unless $(s,t)=(k'-j,k-j)$, in which case one gets
   \[
   \frac{k! (k')!}{j! (k-j)! (k' - j)!}\left(\frac{\tau-\bar{\tau}}{2\pi i}\right)^j.
   \]
   Collecting terms gives the first formula. The second formula is obtained in
   a similar way, observing that
   \[
    w^{(0,k)}\otimes w^{(k',0)} \wedge CG^{[k, k', j]}_B(w^{[s,t]})
     =(-1)^{s+t-j}
    \overline{w^{(k,0)}\otimes w^{(0,k')} \wedge CG^{[k, k', j]}_B(w^{[t,s]})}.
    \qedhere
   \]
  \end{proof}

  \begin{theorem}
   \label{thm:deligne-regulator-formulae}
   Let $0\le j\le \min\{k,k'\}$ and $b=1$. One has
   \begin{align*}
    \frac{1}{2\pi i}\int_{Y(\CC)}\bar{\omega}_{f^*}\wedge \omega_g\wedge
    \alpha_\infty=\frac{(-1)^{k-j}}{2}(2\pi i)^{k+k'-2j}
    \frac{k!k'!}{(k-j)!(k'-j)!}L'(f,g,j+1)\\
    \frac{1}{2\pi i}\int_{Y(\CC)}{\omega}_{f}\wedge \bar\omega_{g^*}\wedge
    \alpha_\infty=\frac{(-1)^{k-1}}{2}(2\pi i)^{k+k'-2j}
    \frac{k!k'!}{(k-j)!(k'-j)!}L'(f, g, j+1),
   \end{align*}
   where $L'(f,g,j+1) \coloneqq \lim_{s\to 0}\frac{L(f,g,j+1+s)}{s}$. In particular,
   \begin{multline*}
    \left\langle\AJ_{\cH,f,g}\left(\Eis_{\cH,1,N}^{[k,k',j]}\right),
    \frac{-1}{\langle\omega_f,\bar{\omega}_f\rangle_Y \langle\omega_g,\bar{\omega}_g \rangle_Y}
    \left(\bar\omega_{f^*}\otimes {\omega}_g+(-1)^{j+1}\omega_f\otimes\bar \omega_{g^{*}} \right)
    \right\rangle=\\
    \frac{(-1)^{k-j+1}(2\pi i)^{k+k'-2j}}{2\langle\omega_f,\bar{\omega}_f\rangle_Y \langle\omega_g,\bar{\omega}_g \rangle_Y}
    \frac{k!k'!}{(k-j)!(k'-j)!}L'(f, g, j+1).
   \end{multline*}
  \end{theorem}

  \begin{proof}
   This follows from the formulae above, the functional equation
   \[
    F^\an_{k-k',k'-j+1,1}=E^{(k-k')}_{1/N}(\tau,j-k)
   \]
   in the notation of \cite[Definition 4.2.1]{LLZ14} and the formulae in
   \eqref{eq:Rankin-formula} with $r=k+2$, $r'=k'+2$.
  \end{proof}


 \subsection{Finite-polynomial cohomology}

  We now compute the image of the Eisenstein class under the $p$-adic syntomic Abel--Jacobi map, following \cite{BDR-BeilinsonFlach} and \cite{DR-diagonal-cycles-I}. In order that we can apply Besser's rigid syntomic cohomology, we need to assume that $p \nmid N$.

  We fix integers $(k, k', j)$ satisfying our usual inequalities \eqref{eq:inequalities}. Let $Y = Y_1(N)_{\Zp}$ and $S = Y \times_{\Spec \Zp} Y$, $Y_{\Qp}$ and $S_{\Qp}$ their generic fibres, and $\sY$ and $\sS$ the smooth pairs $(Y, X)$ and $(Y \times Y, X \times X)$, where $X = X_1(N)_{\Zp}$.

  Let $\sF$ be the sheaf $\TSym^{[k, k']} \sH_{\Qp}  \coloneqq \TSym^k \sH_{\Qp} \boxtimes \TSym^{k'} \sH_{\Qp}$ on $S$ (regarded as an overconvergent filtered $F$-isocrystal), and $\sF^\vee$ its dual.

  As we saw in Section \ref{sect:ajdefs} above, there is a natural map
  \[ \AJ_{\syn, f, g}: H^3_{\mot}(S, \sF(2-j)) \rTo t(M(f \otimes g)^*(-j))_{\Qp} = \left(\Fil^{1 + j} M_{\dR}(f \otimes g)_{\Qp}\right)^*.\]
  Our inequalities on $j$ imply that $\Fil^{1 + j} M_\dR(f \otimes g)$ is 3-dimensional over $L$.

  Suppose we are given an element
  \[ \lambda \in \Fil^{0} M_\dR(f \otimes g)(1 + j)_{\Qp}. \]
  We want to evaluate the pairing
  \[ \left\langle \AJ_{\syn, f, g}\left(\Eis^{[k, k', j]}_{\syn, b, N}\right), \lambda \right\rangle.\]
  Since $1 - p^{-1} \varphi^{-1}$ is an isomorphism on $M_{\dR}(f \otimes g)(1 + j)$, we may find a polynomial $P \in 1 + T L_{\frP}[T]$ such that $P(p^{-1}) \ne 0$ and $P(\varphi)(\lambda) = 0$. Thus $\lambda$ lifts to a class
  \[ \tilde\lambda \in H^2_{\fp, c}\left(S, \sF^\vee(1 + j), P\right);\]
  and since $H^1_{\dR,c}(S, \sF^\vee(1 + j))$ is zero (being Poincar\'e dual to $H^3$ with non-compact supports) this lift is uniquely determined.

  \begin{proposition}
   We have
   \[ \left\langle \AJ_{\syn, f, g}\left(\Eis^{[k, k', j]}_{\syn, b, N}\right), \lambda \right\rangle = \big\langle \Eis^{[k, k', j]}_{\syn,b,N}, \tilde \lambda \big\rangle_{\fp, S}, \]
   where the pairing $\langle -, - \rangle_{\fp, S}$ is as defined in \S \ref{sect:fpdefs} above.
  \end{proposition}

  \begin{proof}
   This is an instance of the general statement that the cup-product on finite-polynomial cohomology is compatible with the Leray spectral sequence, as observed in \S \ref{sect:fpdefs}.
  \end{proof}

  \begin{corollary}
   We have
   \[ \left\langle \AJ_{\syn, f, g}\left(\Eis^{[k, k', j]}_{\syn, b, N}\right), \lambda\right\rangle = \left\langle CG^{[k, k', j]}_{\syn}(\Eis^{k + k' - 2j}_{\syn, b, N}), \Delta^*(\tilde \lambda) \right\rangle_{\fp, Y}.\]
  \end{corollary}

  \begin{proof}
   Recall that $\Eis^{[k, k', j]}_{\syn, b, N}=(\Delta_* \circ CG^{[k, k', j]}_{\syn})\left(\Eis_{\syn,b,N}^{k + k' - 2j}\right)$. The pushforward in fp-cohomology satisfies the projection formula (\cite[Theorem 5.2]{Besser-K1-surface}), so
   \[\left\langle \Eis^{[k, k', j]}_{\syn,b,N}, \tilde \lambda \right\rangle_{\fp, S} = \left\langle CG^{[k, k', j]}_{\syn}(\Eis^{k + k' - 2j}_{\syn,b,N}), \Delta^*(\tilde \lambda) \right\rangle_{\fp, Y}.\qedhere\]
  \end{proof}

  This reduces the computation of the syntomic regulator to a calculation on $Y$ alone. We now give a more explicit recipe when $\lambda$ is of a special form. We suppose that
  \[ \lambda = \eta \sqcup \omega  \coloneqq \pi_1^*(\eta) \cup \pi_2^*(\omega),\]
  where $\pi_1, \pi_2$ are the two projections $S \to Y$, and $\eta \in \Fil^0 M_{\dR}(f)_{\Qp}$, $\omega \in \Fil^0 M_{\dR}(g)_{\Qp}(1 + j)$. We fix polynomials $P_{\eta}$ (pure of weight $k + 1$) and $P_\omega$ (pure of weight $k' + 1 - 2j$) annihilating $\eta$ and $\omega$ respectively. Then $\eta$ and $\omega$ lift uniquely to classes
  \begin{align*}
   \tilde{\eta} &\in H^1_{\fp, c}(Y, \Sym^k \sH_{\Qp}^\vee, P_\eta),\\
   \tilde{\omega} &\in H^1_{\fp, c}(Y, \Sym^k \sH_{\Qp}^\vee(1 + j), P_\omega)
  \end{align*}
  and we clearly have $\tilde\lambda = \tilde\eta \sqcup \tilde\omega$ and hence
  \[ \Delta^*(\tilde{\lambda}) = \Delta^*(\tilde{\eta}\sqcup \tilde{\omega})=\tilde{\eta}\cup \tilde{\omega}.\]
  Thus we have
  \[  \left\langle CG^{[k, k', j]}_{\syn}(\Eis^{k + k' - 2j}_{\syn,b,N}), \Delta^*(\tilde \lambda) \right\rangle_{\fp, Y} = \operatorname{tr}_{\fp, Y}\left( \sigma \cup \tilde{\eta}\cup \tilde{\omega}\right), \]
  where we have written $\sigma = CG^{[k, k', j]}_{\syn}(\Eis^{k + k' - 2j}_{\syn,b,N})$ for brevity, and $\operatorname{tr}_{\fp, Y}$ is as in \S \ref{sect:fpdefs} above.

  \begin{proposition}
   \label{prop:fpcohoformula}
   The natural map
   \[ H^1_{\rig}(Y, \TSym^k \sH_{\rig}(2)) \to H^2_{\fp}(\sY, \TSym^k \sH(2), P_\omega) \]
   is surjective; and if $\Xi \in H^1_{\rig}(Y, \TSym^k \sH_{\rig}(2))$ is any preimage of $\sigma \cup \tilde \omega$, then we have
   \[ \left\langle CG^{[k, k', j]}_{\syn}(\Eis^{k + k' - 2j}_{\syn,b,N}), \tilde\eta \cup \tilde\omega\right\rangle_{\fp, Y} = \left\langle P_\omega(p^{-1} \varphi^{-1})^{-1} \eta, \Xi\right\rangle_{\rig, Y}.\]
  \end{proposition}

  \begin{proof}
   This natural map is clearly surjective (because the cokernel is a subspace of $H^2_{\rig}(\sY, \TSym^k \sH_{\rig}(2))$, which is zero because $Y$ is affine). The equality of cup-products is another instance of the compatibility of finite-polynomial cup-products with the Leray spectral sequence, as noted above.
  \end{proof}

  In the next section we'll compute such a lifting $\Xi$ explicitly.

 \subsection{Restriction to the ordinary locus}

  As in Section \ref{sect:syntomicEis}, we denote by $Y^{\ord}$ the open subscheme of $Y$ where the Eisenstein series $E_{p-1}$ is invertible, so $\sY^{\ord} = (Y^{\mathrm{ord}}, X)$ is also a smooth pair. Note that $Y_{\Qp}^{\ord}$ is the complement of a finite set of points in $Y_{\Qp}$, possibly defined over the unramified quadratic extension of $\Qp$ (one for each supersingular point of the special fibre).

  \begin{proposition}
   \label{prop:liftordlocus}
   The rigid realization $M_{\rig}(f)$ lifts canonically to a subspace of $H^1_{\rig, c}(Y^{\ord}, \Sym^{k} \sH_{\rig})$, commuting with the action of $\varphi$.
  \end{proposition}

  \begin{proof}
   We have an exact sequence
   \[ H^0_{\rig}(Z, \Sym^{k} \sH_{\rig}) \rTo H^1_{\rig, c}(Y^{\ord}, \Sym^{k} \sH_{\rig}) \rTo H^1_{\rig, c}(Y, \Sym^{k} \sH_{\rig}) \rTo 0, \]
   where $Z = Y - Y^{\ord}$ is the scheme of supersingular points. (The surjectivity of the last map follows from the fact that $Z$ is 0-dimensional, so $H^1_{\rig}(Z, \Sym^{k} \sH_{\rig})$ vanishes.)
   
   A theorem due to Deuring and Serre (see \cite[\S 9]{serre96}) shows that the supersingular points of $Y_1(N)_{\overline{\FF}_p}$ are in bijection with the double coset space $B^\times \backslash (B \otimes \mathbb{A}_f)^\times / K$, where $B$ is the definite quaternion algebra ramified at $\{p, \infty\}$ and $K$ is an appropriate open compact subgroup. This bijection is compatible with the action of Hecke correspondences away from $p$, and induces an isomorphism of Hecke modules between $H^0_{\rig}(Z, \Sym^{k} \sH_{\rig})$ and a space of automorphic forms for $B^\times$. (Compare \cite[Proposition 4.10]{tianxiao16}, which is a generalisation to Hilbert modular varieties.) Since the systems of prime-to-$p$ Hecke eigenvalues appearing in the latter are all associated to $p$-new cusp forms of level $\Gamma_1(N) \cap \Gamma_0(p)$, which are disjoint from those appearing in level $\Gamma_1(N)$, it follows that $M_{\rig}(f)$ has a unique Hecke-equivariant lifting to $H^1_{\rig, c}(Y^{\ord}, \Sym^{k} \sH_{\rig})$, whose uniqueness implies it must be compatible with $\varphi$.
  \end{proof}

  We let $\eta^{\ord}$ be the image of $\eta$ under this lifting, and let $\Xi$ be as in Proposition \ref{prop:fpcohoformula}. Then we have
  \[  \left\langle P_\omega(p^{-1} \varphi^{-1})^{-1} \eta, \xi\right\rangle_{\rig, Y} =  \left\langle P_\omega(p^{-1} \varphi^{-1})^{-1} \eta^{\ord}, \Xi |_{Y^{\ord}} \right\rangle_{\rig, Y^{\ord}}. \]

  We now compute a representative for $\Xi |_{Y^{\ord}}$.

  \begin{definition}
   Choose a $\nabla$-closed algebraic section of $\Omega^1 \otimes \Fil^0 \sH^\vee(1 + j)$ over $Y_{\Qp}$ representing the class of $\omega$ (which we shall denote, abusively, by the same letter $\omega$), and let
    \[ F_\omega \in \Gamma\left(X_{\Qp}^{\rig},  \sH_{\rig}^\vee(1 + j)\big|_{Y^{\ord}}\right)\]
   be a rigid-analytic primitive of $P_\omega(\varphi)(\omega)$, so that the class of the pair $(\omega, F_\omega)$ is a lift of $\omega$ to $H^1_{\fp}(\sY^{\ord}, \sH_{\Qp}^\vee(1 + j), P_\omega)$.
  \end{definition}

  As we saw above, the restriction to $Y^{\ord}$ of the syntomic Eisenstein class $ \Eis^{k + k' - 2j}_{\syn,b,N}$ is represented by the pair $\left( \alpha_\rig, \alpha_\dR \right)$
  defined in Theorem \ref{thm:syntomicformula} above, satisfying $\nabla(\alpha_\rig) = (1 - \varphi)\alpha_\dR$. Let
  \begin{align*}
   \sigma_{\rig}& =CG^{[k, k', j]}_{\rig}(\alpha_{\rig}),\\
   \sigma_{\dR} & = CG^{[k, k', j]}_{\dR}(\alpha_{\dR}),
  \end{align*}
  so the restriction to $Y^{\ord}$ of the class $CG^{[k, k', j]}_{\syn}(\Eis^{k + k' - 2j}_{\syn,b,N})$ is given by the pair $(\sigma_{\rig},\sigma_{\dR})$. The definition of the cup-product in finite-polynomial cohomology now gives the following:

  \begin{proposition}
   The class $\Xi |_{Y^{\ord}}$ is represented by the $\TSym^k \sH_{\Qp}(2)$-valued overconvergent 1-form on $Y^{\ord}$ defined by
   \[\Xi  \coloneqq \bigcup\left[a(\varphi_1, \varphi_2) (F_{\omega} \otimes \sigma_\dR)
   - b(\varphi_1, \varphi_2) (\omega \otimes \sigma_\rig)\right], \]
   where $a(x, y)$ and $b(x, y)$ are any two polynomials such that
   \[ P_\omega(xy) = a(x, y)P_\omega(x) + b(x, y) (1-y). \qed\]
  \end{proposition}

  Note that such polynomials do exist, since the polynomial $P(X) = 1 - X$ is the identity for the $\star$ operation of Definition \ref{def:starproduct}. The form $\Xi$ is evidently $\nabla$-closed (since $Y$ is 1-dimensional) and hence defines a class in $H^1_\rig(Y^{\ord}, \TSym^k \sH_{\Qp}(2))$; and this class is well-defined, since changing the polynomials $(a, b)$ changes $P_\omega$ by an exact form.

  We now evaluate the right-hand side of the formula in Proposition \ref{prop:fpcohoformula} in terms of $p$-adic modular forms. We take for $\omega$ the class $\omega'_g(1 + j)$, where $\omega_g' \in \Fil^{1 + j} M_{\dR}(g)$ is as in \ref{sect:periods}; and we take for $P_\omega$ the polynomial
  \[P_g(X) = \left(1 - \frac{p^{1 + j} X}{\alpha_g}\right)\left(1 - \frac{p^{1 + j} X}{\beta_g}\right)\]
  where $\alpha_g, \beta_g$ are the roots of the Hecke polynomial of $g$ at $p$.

  We denote also by $\omega'_g$ the unique regular algebraic differential on $X_1(N_f)$ representing this class, whose pullback to the formal scheme $\tilde{\mathcal{Y}}$ of \S \ref{sect:syntomicEis} is given by
  \[ \omega_g' = G(\varepsilon_g^{-1}) g\, v^{(k',0)} \otimes \xi \otimes e_{1 + j}. \]

  \begin{remark}
   Note that this class $\omega_g'$ does not generally have $q$-expansion in $L[[q]]$, owing to the presence of the Gauss sum, but it is nonetheless defined over $L$; recall that the cusp $\infty$ is not rational on our model of $Y_1(N)$.
  \end{remark}

  \begin{remark}
   More generally, one can replace $g$ with the any holomorphic form $\breve g$ of level $N$ having $q$-expansion in $(L \otimes \Qp)[[q]]$ and the same Hecke eigenvalues as $g$ away from $N$. We give the argument for $\breve g = g$ for simplicity of notation.
  \end{remark}

  As a rigid 1-form we have
  \[ P_g(\varphi)(\omega_g') = G(\varepsilon_g^{-1}) g^{[p]}\, v^{(k',0)} \otimes \xi \otimes e_{1 + j},\]
  where $g^{[p]}$ is the ``$p$-depletion'' of $g$ (the $p$-adic modular form whose $q$-expansion is $\sum_{p \nmid n} a_n(g) q^n$). If we write $G$ for an overconvergent primitive of $P_g(\varphi)(\omega_g')$ vanishing at the cusp $\infty$, then as above we have
   \[ \Xi|_{Y^{\ord}} = \bigcup \left[ a(\varphi_1, \varphi_2)(G \otimes \sigma_\dR) - b(\varphi_1, \varphi_2)(\omega_g' \otimes \sigma_\rig)\right],\]
  Here $a(X, Y)$ and $b(X, Y)$ may be any polynomials such that $P_g(XY) = a(X, Y)P_g(X) + b(X, Y)(1 - Y)$, but we shall make the following choice
  \[ a(X, Y) = 1, \quad b(X, Y) = \frac{P_g(XY) - P_g(X)}{1 - Y},\]
  so that
  \[ \Xi|_{Y^{\ord}} = G \cup \sigma_\dR - \bigcup \left[ b(\varphi_1, \varphi_2)(\omega_g' \otimes \sigma_\rig)\right].\]

  \begin{lemma}
   If $G$ vanishes at the cusp $\infty$, then modulo $\nabla$-exact 1-forms we have
   \[ G \cup \sigma_\dR = G \cup (1 - \varphi)\sigma_\dR = G \cup \nabla(\sigma_\rig).\]
  \end{lemma}

  \begin{proof}
   It suffices to show that $G \cup \varphi(\sigma_\dR)$ is $\nabla$-exact. We know that the $q$-expansion of $G$ is $p$-depleted (the coefficient of $q^n$ is zero if $p \mid n$) while $\varphi(\sigma_{\dR})$ is a power series in $q^p$. Hence the $q$-expansion of $G \cup \varphi(\sigma_\dR)$ is also $p$-depleted, and thus lies in the kernel of the operator $U_p$ defined on $q$-expansions by $U_p\left(\sum a_{n} q^n\right) = \sum a_{np} q^n$. Since $U_p \circ \varphi$ coincides with the diamond operator $\langle p \rangle$, and $\langle p \rangle$ and $\varphi$ act bijectively on the cohomology groups, any differential in the kernel of $U_p$ must be $\nabla$-exact.
  \end{proof}

  As a corollary, we obtain the following explicit formula for $\Xi|_{Y^{\ord}}$:

  \begin{corollary}
   Modulo $\nabla$-exact 1-forms, we have
   \begin{equation}\label{eq:xi}
    \Xi|_{Y^{\ord}}=- \left( 1 - \frac{p^{k + 2} \varphi^2}{\langle p \rangle}\right)\left( \omega_g' \otimes \sigma_\rig\right).
   \end{equation}
  \end{corollary}

  \begin{proof}
   Observe that the 1-form
   \[ G \cup \nabla(\sigma_\rig) + \nabla(G) \cup \sigma_\rig = \nabla\left( G \cup \sigma_\rig\right)\]
   is $\nabla$-exact. Hence modulo $\nabla$-exact forms we have
   \[ G \cup \nabla(\sigma_\rig) = -\nabla(G) \cup \sigma_\rig = -\bigcup \left[ P_g(\varphi_1) (\omega_g' \otimes \sigma_\rig) \right].\]
   Thus
   \[ \Xi|_{Y^{\ord}} = -\bigcup \left[ c(\varphi_1, \varphi_2) (\omega_g' \otimes \sigma_\rig)\right]\]
   where
   \[ c(X, Y) = P_g(X) + b(X, Y) = \frac{P_g(XY) - Y P_g(X)}{1 - Y}.\]
   We evaluate explicitly:
   \[ c(X, Y) = 1 - \frac{p^{2 + 2j} X^2 Y}{\alpha_g \beta_g}.\]
   Hence
   \[ \Xi|_{Y^{\ord}} = \omega_g' \cup \sigma_\rig- \frac{p^{2 + 2j}}{\alpha_g \beta_g} \varphi \left( \varphi(\omega_g') \cup \sigma_\rig\right).\]

   We now note that the relation
   \[
    \sigma_\rig = p^{1 + k + k' - 2j} \langle p \rangle^{-1} \varphi(\sigma_\rig)
   \]
   holds modulo exact forms (as one sees by an explicit $q$-expansion calculation). Hence
   \[
    \Xi|_{Y^{\ord}} = - \left( 1 - \frac{p^{k + k' + 3} \varphi^2}{\alpha_g \beta_g \langle p \rangle}\right)\left(\langle p \rangle \omega_g' \otimes  \sigma_\rig\right) = - \left( 1 - \frac{p^{k + 2} \varphi^2}{\langle p \rangle}\right)\left( \omega_g' \otimes \sigma_\rig\right).\qedhere
   \]
  \end{proof}


 \subsection{Relation to a p-adic L-value}

  Recall that above $\eta$ was an arbitrary element of the free rank 2 $(L \otimes \Qp)$-module
  \[ \Fil^0 M_{\dR}(f)_{\Qp} = M_{\dR}(f)_{\Qp}.\]
  We now specify $\eta$ more precisely. Recall that in \S \ref{sect:periods} we defined a canonical $L$-basis $\eta_f'$ of $M_{\dR}(f) / \Fil^1$.

  Choose a prime $\frP$ of $L$ above $p$, and (after extending $L$ if necessary) choose a root $\alpha_f \in L_{\frP}$ of the Hecke polynomial $X^2 - a_p(f) X + p^{k + 1} \varepsilon_p(f)$ of $f$. Assume that $\alpha$ has $p$-adic valuation $< k+1$ (i.e.~it is not the non-unit root associated to an ordinary form). Then the $\varphi = \alpha$ eigenspace in $M_{\rig}(f) \otimes_{L \otimes \Qp} L_{\frP}$ is complementary to the $\Fil^1$ subspace.

  \begin{definition}
   \label{def:etaf1}
   We let $\eta_f^{\alpha}$ be the unique lifting of $\eta_f' \in M_{\dR}(f) / \Fil^1$ to an element of $M_{\dR}(f) \otimes_L L_{\frP}$ satisfying $\varphi(\eta_f^{\alpha}) = \alpha_f \eta_f^{\alpha}$.
  \end{definition}

  \begin{remark}
   \label{remark:etagausssum}
   If $f$ is ordinary (so $\alpha_f$ is necessarily the unit root) the class $\eta_f^\alpha$ coincides with $G(\varepsilon_f^{-1}) \eta_f^{\mathrm{ur}}$ where $\eta_f^{\mathrm{ur}}$ is the class considered in \cite[Proposition 4.6]{DR-diagonal-cycles-I}, \cite[Theorem 5.6.2]{LLZ14}.
  \end{remark}

  By Proposition \ref{prop:liftordlocus} we may lift $\eta_f^{\alpha}$ uniquely to a compactly-supported class $\eta_f^{\alpha, \ord}$ in the rigid cohomology of $Y^\ord$, having the same Hecke eigenvalues outside $p$ as $f$.

  \begin{remark}
   The data implicit in a lifting of $\eta_f^{\alpha}$ to the compactly-supported cohomology of $Y^{\ord}$ is a choice of local integrals of $\eta_f^{\alpha}$ on some sufficiently small annulus along the boundary of each supersingular residue disc.
  \end{remark}

  \begin{proposition}
   We have
   \[ \langle \eta_f^\alpha, \Xi\rangle_{\rig, Y}  = - \left( 1 - \frac{\beta_f}{p\alpha_f}\right)\langle \eta_f^{\alpha, \ord}, \omega_g' \cup \sigma_\rig\rangle_{\rig, Y^{\ord}},\]
   where $\beta_f = p^{k + 1} \varepsilon_f(p) / \alpha_f$ is the other root of the Hecke polynomial of $f$.
  \end{proposition}

  \begin{proof}
   This follows from Equation \eqref{eq:xi}, since we have
   \begin{align*}
    \langle \eta_f^\alpha, \Xi\rangle_{\rig, Y} &= \langle \eta_f^{\alpha, \ord}, \Xi|_{Y^{\ord}} \rangle_{\rig, Y^{\ord}} \\
    &=  -\left\langle \eta_f^{\alpha, \ord}, \left( 1 - \frac{p^{k + 2} \varphi^2}{\langle p \rangle}\right)\left( \omega_g' \cup \sigma_\rig\right)\right\rangle_{\rig, Y^{\ord}} \\
    &=  -\left\langle  \left( 1 - \frac{p^{k + 2} (p^{-1} \varphi^{-1})^2}{\langle p^{-1} \rangle}\right) \eta_f^{\alpha, \ord}, \omega_g' \cup \sigma_\rig \right\rangle_{\rig, Y^{\ord}} \\
    &= - \left( 1 - \frac{\beta_f}{p\alpha_f}\right)\langle \eta_f^{\alpha, \ord},  \omega_g' \cup \sigma_\rig\rangle_{\rig, Y^{\ord}}.\qedhere
   \end{align*}
  \end{proof}

  To proceed further we need a more explicit description of the class $\omega_g' \cup \sigma_\rig$. It turns out that we only need to consider the image of $\omega_g' \cup \sigma_\rig$ under a certain projection operator:

  \begin{definition}[{cf.~\cite[\S 2.4]{DR-diagonal-cycles-I}}]
   The \emph{unit root splitting} of $\sH_{\rig} |_{Y^{\ord}}$ is the map
   \[ \operatorname{spl}^{ur}: \sH_{\rig} |_{Y^{\ord}} \to \Fil^0 \sH_{\rig} |_{Y^{\ord}} \]
   whose kernel is the unit root subspace for the action of Frobenius.
  \end{definition}
  
  Passing to the $k$-th tensor power, we obtain a projection $\TSym^k \sH_{\rig} \to \TSym^k(\Fil^0 \sH^{\rig})$ over $Y^{\ord}$. After passing to the covering space $\widetilde{\mathcal Y}$ of $Y^{\mathrm{ord}}$ defined in \S \ref{sect:syntomicEis}, the line bundle $\TSym^k(\Fil^0 \sH^{\rig}) \otimes \Omega^1$ becomes trivial, via the basis vector $v^{[0, k]} \otimes \xi$.
 
  \begin{proposition}
   The image of $\omega_g' \cup \sigma_\rig$ under $\operatorname{spl}^{ur}$ is represented by the $(\TSym^k \sH_{\rig})$-valued differential on $Y^{\ord}$ whose pullback to $\widetilde{\mathcal{Y}}$ is $H \cdot v^{[0, k]} \otimes \xi$, where $H$ is the $p$-adic modular form
   \[ -N^{k + k' - 2j}(-1)^{k' - j - 1}  (k')! \binom{k}{j} G(\varepsilon_g^{-1}) \left(g \cdot F^{(p)}_{k - k', k' - j + 1, b}\right).\]
  \end{proposition}

  \begin{proof}
   Recall that $\sigma_{\rig}=CG_{\syn}^{[k,k',j]}(\alpha_{\rig})$, and the pullback of $\alpha_{\rig}$ to the formal scheme $\cY$  is given by the sum
   \[
    \alpha_{\rig}=-N^{k+k'-2j}\sum_{i=0}^{k+k'-2j}(-1)^{k+k'-2j-i}(k+k'-2j-i)!
    F^{(p)}_{2i-(k+k'-2j),k+k'-2j+1-i,b} v^{[k+k'-2j-i,i]}.
   \]
   In terms of the basis $\{ v^{[r, s]} : r +s = k\}$ of $\TSym^{k} \sH |_{\widetilde{\mathcal Y}}$, the unit-root splitting $\operatorname{spl}^{ur}$ sends all basis vectors to zero except $v^{[0, k]}$. Hence, by the definition of the trilinear pairing and Proposition \ref{prop:clebsch-gordan-coefficients}, we see that the linear map given by pairing with $G(\varepsilon_g^{-1}) g \otimes v^{(k', 0)}$ and applying $\operatorname{spl}^{ur}$ sends all terms in this sum to zero except the term for $i = k-j$; this term is given by
   \[ -N^{k + k' - 2j} (-1)^{k'-j-1}(k' - j)! F^{(p)}_{k - k', k'-j+1, b} v^{[k'-j, k-j]}, \]
   and Proposition \ref{prop:clebsch-gordan-coefficients} gives the factor $\frac{k! (k')!}{j!(k-j)!(k'-j)!}$, which gives the claimed formula.
  \end{proof}

   The restriction of the unit-root splitting to the overconvergent $(\TSym^k \sH)$-valued differentials on $Y^{\ord}$ is injective, and (after tensoring with $\QQ(\mu_N)$ to rectify issues with rationality of cusps) its image is the space $M_{k + 2}^{\mathrm{n-oc}}(N, L_{\frP})$ of ``nearly overconvergent'' $p$-adic modular forms, in the sense of \cite[Definition 2.6]{DR-diagonal-cycles-I}. Thus $H$ is a nearly-overconvergent $p$-adic modular form.

   There is a map, the \emph{overconvergent projector},
   \[ \Pi^{oc}: M_{k + 2}^{\mathrm{n-oc}}(N, L_{\frP}) \to \frac{M_{k + 2}^{\dagger}(N, L_{\frP})}{\theta^{k + 1} M_{-k}^{\dagger}(N, L_{\frP})} \]
   (cf.~\cite[Equation (44)]{DR-diagonal-cycles-I}). Composing this with the unit-root splitting gives an isomorphism
   \[ H^1_{\rig}(Y^{\ord}, \TSym^k \sH_{\rig}) \otimes \QQ(\mu_N) \cong \frac{M_{k + 2}^{\dagger}(N, L_{\frP})}{\theta^{k + 1} M_{-k}^{\dagger}(N, L_{\frP})} \otimes \QQ(\mu_N),\]
   whose inverse is given by mapping an overconvergent modular form $\mathcal{F}$ to the class of the differential form $\mathcal{F} \cdot v^{[0, k]} \otimes \xi$.

   By construction, pairing with $\eta_f^{\alpha, \ord}$ defines a linear functional on $H^1_{\rig}(Y^{\ord}, \TSym^k \sH_{\rig})$. In terms of the above isomorphism, we can interpret it as a linear functional on $M_{k + 2}^{\dagger}(N, L_{\frP}) \otimes \QQ(\mu_N)$, which is characterised uniquely by the following properties:
   \begin{itemize}
    \item it factors through the maximal quotient of $M_{k + 2}^{\dagger}(N, L_{\frP})$ on which the Hecke operators away from $p$ act as they do on the conjugate form $f^*$, and the $U_p$ operator acts as $\varepsilon_f(p)^{-1} \alpha_f$; 
    \item it maps the normalised eigenform $f^*$ to $[\Gamma_1(N_f): \Gamma_1(N)] \times G(\varepsilon_f^{-1})$. 
   \end{itemize}
  (The factor $[\Gamma_1(N_f): \Gamma_1(N)]$ appears because $\eta_f^{\alpha}$ is the pullback of a class at level $N_f$, and pullback multiplies the Poincar\'e pairing by the degree of the map.)

  \begin{corollary}
   \label{cor:syntomicreg-pioc}
   We have
   \[ \langle \eta_f^\alpha, \Xi\rangle_{\rig, Y} = (-1)^{k' - j + 1} N^{k + k' - 2j} (k')! \binom{k}{j} G(\varepsilon_g^{-1})\left( 1 - \frac{\beta_f}{p\alpha_f}\right)  \left\langle \eta_f^{\alpha, \ord},  \Pi^{oc}\left(g \cdot F^{(p)}_{k - k', k' - j + 1, b}\right)\right\rangle,\]
   and hence
   \begin{align*}
    \left\langle \AJ_{\syn, f, g}\left(\Eis^{[k, k', j]}_{\syn, b,N}\right), \eta_f^\alpha \otimes \omega_g'\right\rangle =& \frac{1}{P_g(p^{-1}\alpha_f^{-1})} \langle \eta_f^{\alpha,\ord}, \Xi\rangle_{\rig, Y} \\
    = & (-1)^{k' - j + 1} N^{k + k' - 2j} (k')! \binom{k}{j}G(\varepsilon_g^{-1}) \frac{\left( 1 - \frac{\beta_f}{p\alpha_f}\right)}{\left( 1 - \frac{p^j}{\alpha_f \alpha_g}\right)\left( 1 - \frac{p^j}{\alpha_f \beta_g}\right)} \\
    & \times \left\langle \eta_f^{\alpha,\ord},  \Pi^{oc}\left(g \cdot F^{(p)}_{k - k', k' - j + 1, b}\right)\right\rangle.
   \end{align*}
  \end{corollary}

  In order to make the link to $p$-adic $L$-functions, we would like a version of this with the Eisenstein series $F^{(p)}_{t, s, b}$ replaced with the ``$p$-depleted'' one
  \[ F^{[p]}_{t,s,b} = \sum_{\substack{n \ge 0\\ p \nmid n}} q^n \sum_{d d' = n} d^{t - 1 + s}(d')^{-s} (\zeta_N^{bd'} + (-1)^{t} \zeta_N^{-bd'}).\]
  This is the specialization of a 2-variable $p$-adic analytic family, with $t, s$ varying over characters of $\ZZ_p^\times$. It is easy to see that $F^{(p)}_{k-k', k' - j + 1, b}$ is an eigenvector for $U = U_p$ with eigenvalue $p^{k-j}$, and evidently
  \[ (1 - V U) F^{(p)}_{k-k', k' - j + 1, b}= F^{[p]}_{k-k', k' - j + 1, b}\]
  where $V$ is the right inverse of $U$ given by $q \mapsto q^p$ on $\ZZ[[q]]$. We also have the following power-series identity:

  \begin{lemma}
   Let $R$ be a commutative ring and let $A, B \in R[[q]]$ be such that $U A = \lambda A - \mu V A$ and $U B = \nu B$. Then we have
   \[ A \cdot (1 - VU) B = (1 - \lambda \nu V + \mu \nu^2 V^2) \left( A \cdot B\right) + (1 - VU)\left[\mu V^2(A) B - A V(B)\right].\]
    \qed
   \end{lemma}

   Applying the lemma with $A = g$ and $B = F^{(p)}_{k - k', k' - j + 1, 1}$, and noting that anything in the image of $1 - VU$ is annihilated under projection to a finite-slope $U_p$-eigenspace, we can rearrange Corollary \ref{cor:syntomicreg-pioc} to give the following formula: if $\cE(f, g, 1 + j) \ne 0$ then
  \begin{multline}
   \label{eq:syntomicreg}
   \left\langle \AJ_{\syn,f, g} \left(\Eis^{[k, k', j]}_{\syn, b}\right), \eta_f^\alpha \otimes \omega_g'\right\rangle = (-1)^{k' - j + 1} N^{k +k' - 2j} (k')! \binom{k}{j}G(\varepsilon_g^{-1}) \frac{\cE(f)}{\mathcal{E}(f, g, 1 + j)}\\
    \times \left\langle \eta_f^{\alpha,\ord},  \Pi^{oc}\left(g \cdot F^{[p]}_{k - k', k' - j + 1, b}\right)\right\rangle,
  \end{multline}
  where $\cE(f)$ and $\cE(f, g, s)$ are as defined in Theorem \ref{thm:Hida} (and we assume $\mathcal{E}(f, g, 1 + j)$ does not vanish).

  \begin{theorem}
   \label{thm:syntomicreg}
   Suppose that $f$ is ordinary, with $\alpha_f$ be the unit root. If $\mathcal{E}(f, g, 1 + j) \ne 0$ then we have
   \[
    \left\langle \AJ_{\syn,f, g} \left(\Eis^{[k, k', j]}_{\syn, 1,N}\right), \eta_f^\alpha \otimes \omega_g'\right\rangle
    =  (-1)^{k' - j + 1} (k')! \binom{k}{j} G(\varepsilon_f^{-1}) G(\varepsilon_g^{-1})\frac{\mathcal{E}(f) \mathcal{E}^*(f)}{\mathcal{E}(f, g, 1 + j)}
    L_p(f, g, 1 + j),
   \]
   where the $p$-adic $L$-function $L_p(f, g, 1 + j)$ and the factors $\mathcal{E}(f)$ and $\cE^*(f)$ are as defined in Theorem \ref{thm:Hida}.
  \end{theorem}

  \begin{proof}
   Since $\alpha_f$ is the unit root, pairing with $\eta_f^\alpha$ factors through the Hida ordinary idempotent $e_{\ord}$; and we have $e_{\ord} \phi = e_\ord(\Pi^{\mathrm{oc}} \phi)$ for any nearly-overconvergent form $\phi$ \cite[Lemma 2.7]{DR-diagonal-cycles-I}. By the construction of the $p$-adic $L$-function, we have
   \[ N^{k +k' - 2j} \left\langle \eta_f^{\alpha,\ord},  e_{\ord} \left(g \cdot F^{[p]}_{k - k', k' -j + 1, 1}\right)\right\rangle
    = G(\varepsilon_f^{-1}) \cE^*(f) L_p(f, g, 1 + j).
   \]
   (See \cite[Proposition 5.4.1]{LLZ14}; the power of $N$ and the Gauss sum appear because our normalizations are slightly different from \emph{op.cit.}, see Remark \ref{remark:etagausssum} and Remark \ref{remark:DpversusLp}(1) above). Substituting this into equation \eqref{eq:syntomicreg} gives the stated formula.
  \end{proof}

  \begin{remark}\label{rem:nonord}\mbox{~}
   \begin{enumerate}[(i)]
    \item The non-vanishing of $\cE(f, g, 1 + j)$ is automatic, for weight reasons, unless $k = k' = j$.
    \item The $p$-adic Eisenstein series $F^{[p]}_{k - k', k' -j + 1, b}$ is the same as the one denoted by $\mathcal{E}_{b/N}(j - k' + 1, k - j)$ in \cite[Definition 5.3.1]{LLZ14}.
    \item This argument works without the ordinary assumption on $f$, if we use Urban's definition of the $p$-adic Rankin--Selberg $L$-function \cite{Urban-nearly-overconvergent} (see \cite{loeffler-note} and \cite{andreattaiovita-triple-product} for corrections). This construction is only written up for $N = 1$, but the extension to general $N$ is immediate.
    \item The factors $\cE(f)$ and $\cE^*(f)$ can be interpreted as Euler factors attached to the adjoint $L$-function of $f$, which measures the difference between the period $\langle \omega_f, \bar\omega_f\rangle$ used in defining $\eta_f'$ and a ``correctly normalized'' period.
   \end{enumerate}
  \end{remark}


\section{The Perrin-Riou conjecture}

 \subsection{The Beilinson conjecture}

  In this section we consider $0 \le j \le \min\{k,k'\}$ and let
  $f$ and $g$ be new forms of weights $k$ and $k'$ and
levels $N_f$, $N_g$ dividing $N$, respectively. We want to prove part of Beilinson's conjecture for $L(f,g,s)$ at $s=j+1$.

With our conditions on $j$, it follows from the functional
equation that $L(f,g,s)$ has a zero of order $1$ at $s=j+1$. Denote
by $L'(f,g,j+1)$ the value of the derivative of $L(f,g,s)$ at
$j+1$. Beilinson conjectures the following interpretation
of this value:

One expects that there is a Chow motive $M(f\otimes g)^{*}$ underlying
the realizations discussed so far in this paper,
whose motivic cohomology $H^{1}_{\mot}(\QQ,M(f\otimes g)^{*}(-j))$
should be a direct summand of $H^{3}_\mot(Y_1(N)^{2},
\TSym^{[k,k']}\sH(2-j))$ defined by an idempotent in the Hecke algebra. Unfortunately, the ex\-ist\-ence of the Chow motive $M(f\otimes g)^{*}$ is
only known in the case of $k=k'=0$, i.e., if $f$ and $g$ have weight
$2$. Beilinson conjectures further the existence of a subspace of integral elements
\[
H^{1}_{\mot}(\ZZ,M(f\otimes g)^{*}(-j))\subset
H^{1}_{\mot}(\QQ,M(f\otimes g)^{*}(-j))
\]
of the motivic cohomology of $M(f\otimes g)^{*}$, which should have
$L$-dimension $1$ and such that the regulator induces an isomorphism
\[
r_\cH:
H^{1}_{\mot}(\ZZ,M(f\otimes g)^{*}(-j))\otimes_\QQ \RR\isom
H^1_\cH(\RR, M_B(f\otimes g)^*(-j)_\RR).
\]
In our results, which will be formulated below, we do not have to
say anything about the dimension of the motivic cohomology.
Recall from \ref{eq:deligne-cohomology} the exact sequence
  \begin{equation}
   0\to \Fil^{-j}M_\dR(f\otimes g)^*_\RR \to M_B(f\otimes g)^*(-j-1)_\RR^+ \to
     H^1_\cH(\RR, M_B(f\otimes g)^*(-j)_\RR)\to 0
  \end{equation}
and the isomorphism $H^1_\cH(\RR, M_B(f\otimes g)^*(-j)_\RR)\isom
(\ker\alpha_{M(f\otimes g)(j+1)})^{*}$.
The exact sequence induces an isomorphism
\[
{\det}_{\RR\otimes L}(M_B(f\otimes g)^*(-j-1)_\RR^+ )\otimes
{\det}_{\RR\otimes L}(\Fil^{-j}M_\dR(f\otimes g)^*_\RR)^{-1}\isom
{\det}_{\RR\otimes L}(H^1_\cH(\RR, M_B(f\otimes g)^*(-j)_\RR)).
\]
Beilinson defines an $L$-structure on the left hand side as follows:
\begin{conjecture}[Beilinson]
Denote by
\[
\cR(M(f\otimes g)^{*}(-j)) \coloneqq {\det}_L(M_B(f\otimes g)^*(-j-1))
\otimes {\det}_L(\Fil^{-j}M_\dR(f\otimes g)^*)
\]
the one-dimensional $L$-vector space of
${\det}_{\RR\otimes L}(H^1_\cH(\RR, M_B(f\otimes g)^*(-j)_\RR))$.
Then
\[
r_\cH(H^{1}_{\mot}(\ZZ,M(f\otimes g)^{*}(-j)))=
L'(f,g,j+1)\cR(M(f\otimes g)^{*}(-j)).
\]
\end{conjecture}
Recall from  \ref{def:Rankin-Eisenstein-class} the Rankin--Eisenstein class
    \[
       \AJ_{\cH,f,g}(\Eis^{[k, k', j]}_{\cH, 1, N})\in  H^1_\cH(\RR, M_B(f\otimes g)^*(-j)_\RR),
    \]
which is by construction in the image of the regulator map
\[
r_\cH:
H^{1}_{\mot}(\QQ,M(f\otimes g)^{*}(-j))\otimes_\QQ \RR\to
H^1_\cH(\RR, M_B(f\otimes g)^*(-j)_\RR).
\]
\begin{remark} Scholl has defined in \cite{scholl-integral}
a subspace of integral elements in motivic cohomology, which coincides with Beilinson's integral elements if the variety admits a projective, flat and regular model.
We remark that the motivic Eisenstein
class $\Eis^{k+k'-2j}_{\mot, 1, N}$ can be defined integrally
over the regular model $Y_1(N)_\ZZ$ over $\ZZ$ by considering the moduli space
of elliptic curves with a point of exact order $N$ in the sense of
Drinfeld. With the techniques in loc. cit. one can show that
$\Eis^{[k, k', j]}_{\mot, 1, N}$ lies in the subspace of Scholl's
integral elements.
\end{remark}
Before we formulate our results, we choose bases to make the involved $L$-structures more explicit. Recall that $f^*$ and $g^*$ denote the forms with complex conjugate Fourier coefficients and that we have associated differential forms
  $\omega_f=f(\tau)(2\pi)^{k+1}w^{(k,0)}d\tau$.

  Observe that $\varepsilon_{f^*}=\varepsilon_f^{-1}$ and
  that $\overline{G(\varepsilon_f)}=\varepsilon_f(-1)G(\varepsilon_f^{-1})=(-1)^{k+2}G(\varepsilon_f^{-1})$. Moreover one has $G(\varepsilon_f)\overline{G(\varepsilon_f)}=N_{\varepsilon_f}$.
  Let $M(\varepsilon_f)$ the Artin motive associated to the character $\varepsilon_f$.
  The cup-product between cohomology in
  degree $0$ and $1$ induces an isomorphism of motives
  \begin{equation}
  M(\varepsilon_f)(k+1)\otimes M(f)\isom M(f^*)(k+1)\isom M(f)^*.
  \end{equation}

  \begin{definition}
   Denote by $\omega_{\varepsilon_f}\in M_\dR(\varepsilon_f)$ and $\delta_{\varepsilon_f}\in M_B(\varepsilon_f)$ generators  such that $\omega_{\varepsilon_f}=G(\varepsilon_f)\delta_{\varepsilon_f}$. This is possible by
  \cite[Section 6.4]{Deligne-L-values}.
  \end{definition}

  \begin{note}
   We have $(2\pi i)^{k+1}\delta_{\varepsilon_f}\in M_B(\varepsilon_f)(k+1)^-$ because $\varepsilon_f$ has parity $k+2$.
  \end{note}

  \begin{definition}
   Choose a basis $\delta_f^\pm$ of $M_B(f)^\pm$ and $\delta_g^\pm$ of $M_B(g)^\pm$,
   so that
   \[
   \{\delta_f^+\otimes\delta_g^+,\delta_f^+\otimes\delta_g^-,
   \delta_f^-\otimes\delta_g^+,\delta_f^-\otimes\delta_g^-\}
   \]
   is an $L$-basis of $M_B(f\otimes g)$.
  \end{definition}

  \begin{definition}
  Let
  \begin{align*}
  \delta_{f^*}^\pm& \coloneqq (2\pi i)^{k+1}\delta_{\varepsilon_f}\delta_f^\mp\in M_B(f^*)(k+1)^\pm\\
  \delta_{g^*}^\pm& \coloneqq (2\pi i)^{k'+1}\delta_{\varepsilon_g}\delta_g^\mp\in M_B(g^*)(k'+1)^\pm.
  \end{align*}
  We normalize $\delta_{\varepsilon_f}$ and $\delta_{\varepsilon_g}$ such
  that $\langle \delta_{f^*}^\pm, \delta_f^\pm\rangle=1$ and
  $\langle \delta_{g^*}^\pm, \delta_g^\pm\rangle=1$.
  \end{definition}

   \begin{definition}
    Let
    \begin{align*}
    \widetilde{\omega}_{f^*}& \coloneqq G(\varepsilon_f^{-1})\omega_{\varepsilon_f}\omega_f\in M_\dR(f^*)\\
    \widetilde{\omega}_{g^*}& \coloneqq G(\varepsilon_g^{-1})\omega_{\varepsilon_g}\omega_g\in M_\dR(g^*).
    \end{align*}
    \end{definition}
  The next result is crucial for the period computation.
  \begin{lemma}\label{lemma:period-comp}
  For the Poincar\'e duality pairing $\langle\quad ,\quad\rangle$
  the following identities hold:
  \begin{align*}
  \langle\widetilde{\omega}_{f^*},\delta_f^\pm\rangle&=(2\pi i)^{-k-1}(-1)^k N_{\varepsilon_f}\langle {\omega}_{f},\delta_{f^*}^\mp\rangle=\mp
(2\pi i)^{-k-1}(-1)^{k} N_{\varepsilon_f}\langle \overline{\omega}_{f^*},\delta_{f^*}^\mp\rangle\\
   \langle\widetilde{\omega}_{g^*},\delta_g^\pm\rangle&=(2\pi i)^{-k'-1}(-1)^{k'}N_{\varepsilon_g}\langle{\omega}_{g},\delta_{g^*}^\mp\rangle=\mp(2\pi i)^{-k'-1}(-1)^{k'}N_{\varepsilon_g}\langle\overline{\omega}_{g^{*}},\delta_{g^*}^\mp\rangle.
  \end{align*}
  \end{lemma}

  \begin{proof}
  Using the definitions, one sees that
  \begin{align*}
  \langle \widetilde{\omega}_{f^*},\delta_f^\pm\rangle
  &=\langle G(\varepsilon_f^{-1})\omega_{\varepsilon_f}\omega_f,(2\pi i)^{-k-1}\delta_{\varepsilon_f}^{-1}\delta_{f^*}^\mp\rangle\\
  &=(-1)^{k}N_{\varepsilon_f}\langle \delta_{\varepsilon_f}\omega_f,(2\pi i)^{-k-1}\delta_{\varepsilon_f}^{-1}\delta_{f^*}^\mp\rangle\\
  &=(2\pi i)^{-k-1}(-1)^{k}N_{\varepsilon_f}\langle \omega_f,\delta_{f^*}^\mp\rangle.
  \end{align*}
  As $\overline{F_\infty}^*\omega_f=\overline\omega_{f^*}$ and $\overline{F_\infty}^*\delta_{f^*}^\mp=\mp\delta_{f^*}^\mp$ one gets also
  \[
   \langle \widetilde{\omega}_{f^*},\delta_f^\pm\rangle=\mp(2\pi i)^{-k-1}
   (-1)^{k-1} N_{\varepsilon_f}\langle \overline{\omega}_{f^*},\delta_{f^*}^\mp\rangle.\qedhere
   \]
  \end{proof}
  Finally, we define a basis for $M_B(f\otimes g)^*(-j-1)^+\isom M_B(f^*\otimes g^*)(k+k'-j+1)^+$.
  \begin{definition}
  Let $\gamma_j^*,\delta_j^*$ be the $L$-basis of $M_B(f^*\otimes g^*)(k+k'-j+1)^+$ defined by
  \begin{align*}
  \gamma_j^*& \coloneqq (2\pi i)^{-j-1}\delta_{f^*}^+\otimes \delta_{g^*}^{(-1)^{j+1}}\\
  \delta_j^*& \coloneqq (2\pi i)^{-j-1}\delta_{f^*}^-\otimes \delta_{g^*}^{(-1)^j}.
  \end{align*}
  \end{definition}

  We consider $\widetilde{\omega}_{f^*}\otimes \widetilde{\omega}_{g^*}$  as a
  basis of $\Fil^{-j}M_\dR(f\otimes g)^*$ via the isomorphism
  $M_\dR(f\otimes g)^*\isom M_\dR(f^*\otimes g^*)(k+k'+2)$.
With these definitions we are able to give an explicit generator
of the $L$-vector space $\cR(M(f\otimes g)^{*}(-j-1))$.
\begin{proposition}\label{prop:special-basis}
In the case where $j+1$ is even (resp. odd) the image of the element
\begin{align*}
(2\pi i)^{-j-1}(\langle\widetilde{\omega}_{f^*},\delta_f^+\rangle\langle\widetilde{\omega}_{g^*},\delta_g^+\rangle)^{-1}\delta_j^*&&
(\mbox{resp. } (2\pi i)^{-j-1}(\langle\widetilde{\omega}_{f^*},\delta_f^+\rangle\langle\widetilde{\omega}_{g^*},\delta_g^-\rangle)^{-1}\delta_j^*)
\end{align*}
in $ H^1_\cH(\RR, M_B(f\otimes g)^*(-j)_\RR)$ is an $L$-basis of
$\cR(M(f\otimes g)^{*}(-j-1))$.
\end{proposition}
\begin{proof}
We treat only the case
   of $j+1$ even. The odd case is entirely similar.
   Write
   \begin{align*}
   \widetilde{\omega}_{f^*}&=\langle\widetilde{\omega}_{f^*},\delta_f^+\rangle \delta_{f^*}^+ +\langle\widetilde{\omega}_{f^*},\delta_f^-\rangle\delta_{f^*}^-\\
  \widetilde{\omega}_{g^*}&=\langle\widetilde{\omega}_{g^*},\delta_g^+\rangle \delta_{g^*}^+ +\langle\widetilde{\omega}_{g^*},\delta_g^-\rangle\delta_{g^*}^-
   \end{align*}
   and let $\pi_{-j-1}:\CC\to \RR(-j-1)$ be the projection $z\mapsto \frac{1}{2}(z+(-1)^{-j-1}z)$. Then one has (because $j+1$ is even)
   \begin{align*}
   \pi_{-j-1}( \widetilde{\omega}_{f^*}\otimes  \widetilde{\omega}_{g^*})&=
   \langle\widetilde{\omega}_{f^*},\delta_f^+\rangle\langle\widetilde{\omega}_{g^*},\delta_g^+\rangle\delta_{f^*}^+\otimes \delta_{g^*}^+
   +\langle\widetilde{\omega}_{f^*},\delta_f^-\rangle\langle\widetilde{\omega}_{g^*},\delta_g^-\rangle\delta_{f^*}^-\otimes \delta_{g^*}^-\\
   &=(2\pi i)^{j+1}(\langle\widetilde{\omega}_{f^*},\delta_f^+\rangle\langle\widetilde{\omega}_{g^*},\delta_g^+\rangle\gamma_j^*+\langle\widetilde{\omega}_{f^*},\delta_f^-\rangle\langle\widetilde{\omega}_{g^*},\delta_g^-\rangle\delta_j^*).
   \end{align*}
   The image of the element $(2\pi i)^{-j-1}(\langle\widetilde{\omega}_{f^*},\delta_f^+\rangle\langle\widetilde{\omega}_{g^*},\delta_g^+\rangle)^{-1}\delta_j^*$ in
   $ H^1_\cH(\RR, M_B(f\otimes g)^*(-j)_\RR)$ is a basis whose
   determinant with $ \pi_{-j-1}( \widetilde{\omega}_{f^*}\otimes  \widetilde{\omega}_{g^*})$
   and $\gamma_j^*,\delta_j^*$ is $1$.
\end{proof}
With the above notations we get the formula for $L'(f,g,j+1)$ as
in Beilinson's conjecture. In the case $k=k'=0$, i.e. $j=0$, it was first proved by
   Beilinson \cite{Beilinson-L-values} and for general $k,k'$
   and $0\le j\le \min\{k,k'\}$ it was announced by Scholl (unpublished, but see \cite{Kings-Higher-Regulators} for closely
related results in the Hilbert--Blumenthal case).
\begin{theorem}\label{thm:Beilinson}
Let  $0\le j\le \min\{k,k'\}$, then $\AJ_{\cH,f,g}(\Eis^{[k, k', j]}_{\cH, 1, N})$ generates the $L$-subspace
\[
L'(f,g,j+1)\cR(M(f\otimes g)^{*}(-j-1))
\]
of $ H^1_\cH(\RR, M_B(f\otimes g)^*(-j)_\RR)$.
\end{theorem}
\begin{proof}
For simplicity we again treat only the case $j+1$ even. From Theorem
\ref{thm:deligne-regulator-formulae} we get
\begin{multline*}
  \left\langle\AJ_{\cH,f,g}\left(\Eis_{\cH,1,N}^{[k,k',j]}\right),
\frac{-1}{\langle\omega_f,\bar{\omega}_f\rangle_Y \langle\omega_g,\bar{\omega}_g \rangle_Y}
\left(\bar\omega_{f^*}\otimes {\omega}_g+(-1)^{j+1}\omega_f\otimes\bar \omega_{g^{*}} \right)
\right\rangle=\\
\frac{(-1)^{k-j+1}(2\pi i)^{k+k'-2j}}{2\langle\omega_f,\bar{\omega}_f\rangle_Y \langle\omega_g,\bar{\omega}_g \rangle_Y}
    \frac{k!k'!}{(k-j)!(k'-j)!}L'(f, g, j+1).
\end{multline*}
and a straightforward computation with the basis $(2\pi i)^{-j-1}(\langle\widetilde{\omega}_{f^*},\delta_f^+\rangle\langle\widetilde{\omega}_{g^*},\delta_g^+\rangle)^{-1}\delta_j^*$ from
Proposition \ref{prop:special-basis} using Lemma \ref{lemma:period-comp} gives
\begin{multline*}
\left\langle (2\pi i)^{-j-1}(\langle\widetilde{\omega}_{f^*},\delta_f^+\rangle\langle\widetilde{\omega}_{g^*},\delta_g^+\rangle)^{-1}\delta_j^*,
\frac{-1}{\langle\omega_f,\bar{\omega}_f\rangle_Y \langle\omega_g,\bar{\omega}_g \rangle_Y}
\left(\bar\omega_{f^*}\otimes {\omega}_g+(-1)^{j+1}\omega_f\otimes\bar \omega_{g^{*}} \right)
\right\rangle=\\
\frac{(-1)^{k+k'}2(2\pi i)^{k+k'-2j}}{N_{\varepsilon_f}N_{\varepsilon_g}\langle\omega_f,\bar{\omega}_f\rangle_Y \langle\omega_g,\bar{\omega}_g \rangle_Y}.
\end{multline*}
This gives
\[
\AJ_{\cH,f,g}\left(\Eis_{\cH,1,N}^{[k,k',j]}\right)=L'(f,g,j+1)\left(
\frac{(-1)^{k'-j+1}N_{\varepsilon_f}N_{\varepsilon_g}k!k'!}{4(k-j)!(k'-j)!}
\right)(2\pi i)^{-j-1}(\langle\widetilde{\omega}_{f^*},\delta_f^+\rangle\langle\widetilde{\omega}_{g^*},\delta_g^+\rangle)^{-1}\delta_j^*,
\]
which implies the assertion of the theorem.
\end{proof}
For the Perrin-Riou conjecture it is necessary to reformulate the
above theorem in terms of a period.
   \begin{definition}\label{def:infty-period} Let $0\le j\le \min\{k,k'\}$.
   The \emph{$\infty$-period} $\Omega_\infty(j+1)$ of the motive $M(f\otimes g)(j+1)$ is the
   element in $(L\otimes_\QQ\RR)^\times$ given by the determinant of
   \[
     0\to \Fil^{-j}M_\dR(f\otimes g)^*_\RR \to M_B(f\otimes g)^*(-j-1)_\RR^+ \to
          H^1_\cH(\RR, M_B(f\otimes g)^*(-j)_\RR)\to 0
      \]
   with respect to the bases $\widetilde{\omega}_{f^*}\otimes \widetilde{\omega}_{g^*}$,
   $\gamma_j^*,\delta_j^*$ and $ \AJ_{\cH,f,g}(\Eis^{[k, k', j]}_{\cH, b, N})$.
   \end{definition}
   \begin{remark}
   Note that $\Omega_\infty(j+1)$ is independent of the choice of bases up to an element
   in $L^\times$. The condition in the definition means that under the isomorphism
   \begin{equation}\label{eq:determinant}
   \det(\Fil^{-j}M_\dR(f\otimes g)^*_\RR)\otimes \det( M_B(f\otimes g)^*(-j-1)_\RR^+ )^{-1}
   \otimes \det(H^1_\cH(\RR, M_B(f\otimes g)^*(-j)_\RR))\isom L\otimes_\QQ\RR
   \end{equation}
   the determinants of the corresponding bases map to $\Omega_\infty(j+1)$.
   \end{remark}
   The next theorem is a reformulation of the Beilinson conjecture for the
   motive $M(f\otimes g)(j+1)$ with the $\infty$-period $\Omega_\infty(j+1)$.
   \begin{theorem}\label{thm:Beilinson-conj} Let  $0\le j\le \min\{k,k'\}$ and
   $L'(f,g,j+1)\in (L\otimes_\QQ\RR)^\times$ be the leading term of the
   $L$-function of $M(f\otimes g)$ at $j+1$. Then
   \[
   \frac{L'(f, g,j+1)}{\Omega_\infty(j+1)}=
   \frac{(-1)^{k'-j+1}4}{N_{\varepsilon_f}N_{\varepsilon_g}}\frac{(k-j)!(k'-j)!}{k!k'!}\in L^\times.
   \]
   \end{theorem}
   \begin{proof}
   This is just a reformulation of the computation in the
proof of Theorem \ref{thm:Beilinson}.
   \end{proof}


  \subsection{The Perrin-Riou conjecture (p-adic Beilinson conjecture)}

   We continue to assume that $f, g$ are new and $0\le j\le\min\{k,k'\}$, and we choose a prime $p$ not dividing the levels $N_f$, $N_g$, and a prime $\frP$ above $p$ of the coefficient field $L$, such that $f$ and $g$ are ordinary at $\frP$.

   To formulate the Perrin-Riou conjecture we
   first have to define the $p$-adic period (see \cite{PerrinRiou-fonctionsL}). One would like to
have a $p$-adic analogue of the complex period map
\[
\alpha_{M(f\otimes g)(j+1)}:M_B(f\otimes g)(j+1)^{+}_\RR\to
t(M(f\otimes g)(j+1))_\RR.
\]
The problem is that there is no good $p$-adic analogue of
the two dimensional $+$-part of Betti cohomology.
To remedy this defect, Perrin-Riou (as explained by Colmez (\cite{Colmez-fonctionsL})
proposes to choose elements $v_1,\ldots,v_4\in M_\dR(f\otimes g)_\Qp$
such that $\cN^{+} \coloneqq \langle v_1,v_2\rangle$ plays the role of $M_B(f\otimes g)^{+}$
and $\cN^{-} \coloneqq \langle v_3,v_4\rangle $ of $M_B(f\otimes g)^{-}$.
The natural projection to the tangent space induces a map
\[
\alpha_{M(f\otimes g)(j+1),\cN}:\cN^{(-1)^{j+1}}\to  t(M(f\otimes g)(j+1))_{\Qp}.
\]
For suitable choices of $\cN^{\pm}$ one gets a short exact sequence
\[
0\to \ker (\alpha_{M(f\otimes g)(j+1),\cN})\to \cN^{(-1)^{j+1}}\to t(M(f\otimes g)(j+1))_\Qp\to 0
\]
Using the pairing $M_\dR(f\otimes g)(j+1)\times M_\dR(f\otimes g)^{*}(-j)\to  L$ we can identify
$t(M(f\otimes g)^{*}(-j))\isom (\Fil^{j+1}M_\dR(f\otimes g))^{*}$
and get
\begin{equation}
\label{eq:N-pm-sequence}
0\to \Fil^{-j} M_\dR(f\otimes g)^{*}_\Qp\to (\cN^{(-1)^{j+1}})^{*}\to\ker (\alpha_{M(f\otimes g)(j+1),\cN})^{*} \to 0.
\end{equation}
If we compose the Abel--Jacobi map as in \S \ref{sect:ajdefs}
with the canonical projection we get
   \begin{equation} \label{eq:AJ-syn}H^3_{\syn}(Y_1(N)^2_{\Zp}, \TSym^{[k, k']}(\sH_\Qp)(2-j)) \rTo
   t(M(f\otimes g)^*(-j))_{\Qp}\rTo \ker (\alpha_{M(f\otimes g)(j+1),\cN})^{*}.\end{equation}
In our case we will choose $\cN^{+}=\cN^{-}$:

\begin{definition}
Assume that $\alpha_f,\alpha_g\in L$ and
let $\eta_f^{\alpha}$, $\eta_g'$ and $\omega_g'$ be the classes defined in \ref{def:etaf1}, \ref{prop:eta-prime} and \ref{lemma:gauss-sum} respectively, then we put
\begin{align*}
v_1=v_3=\frac{N_{\varepsilon_f}N_{\varepsilon_g}}{\cE(f)\cE^*(f)}\eta_f^\alpha \otimes\omega_g'&&\mbox{and}&&
v_2=v_4=\eta_f^{\alpha}\otimes
{\eta}_{g}'.
\end{align*}
so that $
\cN \coloneqq \cN^{+}=\cN^{-} = \eta_f^{\alpha}\Lp\otimes M_\dR(g)_\Qp\subset M_\dR(f\otimes g)_\Qp$.
\end{definition}

\begin{remark}
Perrin-Riou associates to each choice of $v_1, v_2, v_3, v_4$ a $p$-adic $L$-function.
In this paper we will restrict ourselves to a choice which gives the
$p$-adic $L$-function $L_p(f,g,s)$. The above formulae only define $v_2$ modulo $v_1$, but this is not important for the constructions below.
\end{remark}

The $p$-adic period is now defined as
follows:

\begin{definition}\label{def:p-period}
The \emph{$p$-adic period}
$\Omega_p(j+1,\underline{v})$ associated to $\underline{v} \coloneqq (v_1,v_2,v_3,v_4)$ is defined to be
\[
\Omega_p(j+1,\underline{v}) \coloneqq \left\langle \AJ_{\syn,f, g} \left(\Eis^{[k, k', j]}_{\syn,1, N}\right), v_1\right\rangle
\]
where one considers $\AJ_{\syn,f, g} \left(\Eis^{[k, k', j]}_{\syn, 1,N}\right)$ as an element in $\ker (\alpha_{M(f\otimes g)(j+1),\cN})^{*}$
using \ref{eq:AJ-syn}.
\end{definition}

\begin{remark}\mbox{~}
   \begin{enumerate}[(i)]
\item
Colmez chooses a splitting of the exact sequence \ref{eq:N-pm-sequence}
and shows that his definition does not depend on this choice. We have
taken $v_2=v_4=\eta_f^{\alpha}\otimes
\overline{\omega}_{g^{*}}$ to be the element mapping to a generator of
$t(M(f\otimes g)(j+1))_{\Qp}$.
\item From the definition and Theorem \ref{thm:syntomicreg} one sees
that $\Omega_p(j+1,\underline{v})\neq 0$ precisely if the $p$-adic $L$-function
$L_p(f,g,s)$ does not vanish at $s=j+1$. We point out again, that
for our choice of $j$, this is not a point of classical interpolation
so that we have no control on this vanishing in terms of the complex
$L$-function.
\end{enumerate}
\end{remark}
For $x\in M_\dR(f\otimes g)_\Qp$ write, as in Colmez \cite{Colmez-fonctionsL}, $t^{-n}x\in M_\dR(f\otimes g)(n)$ for the image of $x$ under the canonical isomorphism.
Note that this is compatible with the isomorphisms $M_\dR(f\otimes g)(n)_\Qp\isom D_\cris(M(f\otimes g)(n))=t^{-n}D_\cris(M(f\otimes g))\isom M_\dR(f\otimes  g)_\Qp$ and that $\varphi$ acts via $p^{-n}$ on $t^{-n}$.
Then an easy computation using $\varphi(\eta_f^{\alpha})=\alpha_f\eta_f^{\alpha}$ gives
   \begin{align*}
   \det\big(1-\varphi\mid t^{-j-1}\cN\big)&=
   \left(1-\frac{\alpha_f\alpha_g}{p^{j+1}}\right)
   \left(1-\frac{\alpha_f\beta_g}{p^{j+1}}\right)\\
   \det\big(1-p^{-1}\varphi^{-1}\mid t^{-j-1}\cN\big)&=
     \left(1-\frac{p^j}{\alpha_f\alpha_g}\right)
     \left(1-\frac{p^j}{\alpha_f\beta_g}\right)
   \end{align*}
   so that
   \begin{equation}\label{eq:euler-factor}
   \cE(f,g,j+1)=\frac{ \det\big(1-p^{-1}\varphi^{-1}\mid t^{-j-1}\cN\big)\det\big(1-\varphi\mid D_\cris(M(f\otimes g)(j+1))\big)}{\det\big(1-\varphi\mid t^{-j-1}\cN\big) }.
   \end{equation}
\begin{definition}
Denote by
\[
L'_{\{p\}}(f,g,j+1)
\]
the derivative of
$L(f,g,s)$ at $s=j+1$ without the Euler factor at $p$,
which is $\det\big(1-\varphi\mid D_\cris(M(f\otimes g)(j+1))\big)$.
\end{definition}

   The next Theorem is in the spirit of Perrin-Riou's conjecture
   \cite[4.2.2]{PerrinRiou-fonctionsL} (see also Colmez \cite[Conjecture 2.7]{Colmez-fonctionsL}) for Hida's $p$-adic Rankin--Selberg $L$-function at $s = 1 + j$. It shows that the complex Rankin-Selberg $L$-function, multiplied by some suitable periods defined via motivic cohomology, has a $p$-adic interpolation by Hida's $p$-adic Rankin-Selberg $L$-function in the range $0 \le j \le \min(k, k')$. Perrin-Riou expects such an interpolation in the cyclotomic variable $j$, whereas we get an interpolation in the direction of Hida families and only for certain cyclotomic twists. In this sense the next theorem is stronger but also weaker than Perrin-Riou's conjecture. 
   
   \begin{theorem}\label{thm:Perrin-Riou}
    Let $L_p(f,g,s)$ be Hida's $p$-adic Rankin--Selberg $L$-function and let $0 \le j \le \min(k, k')$. Suppose that $\Omega_p(j+1,\underline{v})\neq 0$ holds. Then
    \[
    L_p(f,g,j+1)=4^{-1}(-1)^{k'+1}\Gamma(j+1)\Gamma(j-k')^{*}G(\varepsilon_f^{-1}) G(\varepsilon_g^{-1})\frac{L'_{\{p\}}(f,g,j+1)}{\Omega_\infty(j+1)}
    \Omega_p(j+1,\frac{ 1-p^{-1}\varphi^{-1}}{1-\varphi}(-t)^{-j-1}\underline{v}),
    \]
where $\Gamma(j-k')^{*}=\frac{(-1)^{k'-j}}{(k'-j)!}$ is the residue of $\Gamma(s)$ at $s=j-k'$.
   \end{theorem}
\begin{proof}
By Theorem \ref{thm:syntomicreg} and the above calculation of the action of $\varphi$ one has
\begin{multline*}
 \Omega_p(j+1,\frac{ 1-p^{-1}\varphi^{-1}}{1-\varphi}(-t)^{-j-1}\underline{v})= (-1)^{j+1}\frac{ \det\big(1-p^{-1}\varphi^{-1}\mid t^{-j-1}\cN\big)}{\det\big(1-\varphi\mid t^{-j-1}\cN\big) } \left\langle \AJ_{\syn,f, g} \left(\Eis^{[k, k', j]}_{\syn, 1}\right), v_1\right\rangle\\
    =  (-1)^{k'} (k')! \binom{k}{j}\frac{G(\varepsilon_f^{-1}) G(\varepsilon_g^{-1})N_{\varepsilon_f}N_{\varepsilon_g}}{\mathcal{E}(f, g, j + 1)}\frac{ \det\big(1-p^{-1}\varphi^{-1}\mid t^{-j-1}\cN\big)}{\det\big(1-\varphi\mid t^{-j-1}\cN\big) }
    L_p(f, g, j+ 1),
\end{multline*}
which gives using \eqref{eq:euler-factor}
\[
\Omega_p(j+1,\frac{ 1-p^{-1}\varphi^{-1}}{1-\varphi}(-t)^{-j-1}\underline{v})=
(-1)^{k'} (k')! \binom{k}{j}\frac{G(\varepsilon_f^{-1}) G(\varepsilon_g^{-1})N_{\varepsilon_f}N_{\varepsilon_g}}{\det\big(1-\varphi\mid D_\cris(M(f\otimes g)(j+1))\big)}
    L_p(f, g, j+ 1).
\]
On the other hand, by Theorem \ref{thm:Beilinson-conj} we have
 \[
   4^{-1}\Gamma(j+1)\Gamma(j-k')^{*}\frac{L'_{\{p\}}(f, g,j+1)}{\Omega_\infty(j+1)}=
   \frac{-1}{N_{\varepsilon_f}N_{\varepsilon_g}k'!}\binom{k}{j}^{-1}\det\big(1-\varphi\mid D_\cris(M(f\otimes g)(j+1))\big),
   \]
which proves the theorem.
\end{proof}

\begin{remark}
 We have verified Perrin-Riou's conjecture for the $p$-adic $L$-values $L_p(f, g, s)$, for integer values of $s$ in a certain interval. One can also evaluate the $L$-function at more general characters $\sigma: \ZZ_p^\times \to \mathbf{C}_p^\times$, of the form $\sigma(z) = z^m \chi(z)$, for $m \in \ZZ$ and $\chi$ a Dirichlet character of $p$-power conductor. Perrin-Riou's conjecture also predicts the leading term of the $p$-adic $L$-function at these points.
 
 If $m = 1 + j$ for $0 \le j \le \min(k, k')$, and $\chi$ is non-trivial, then Perrin-Riou's conjecture for the value at $\sigma$ is not directly accessible via the above methods. Although one can certainly construct motivic cohomology classes for these character twists, one needs to work with modular curves with level divisible by $p$. These curves have bad reduction at $p$, so our proof of Theorem \ref{thm:syntomicreg} does not immediately generalise to this context. Syntomic regulators of Rankin--Eisenstein classes for certain weight 2 modular forms with level divisible by $p$ have been computed in \cite{BDR-BeilinsonFlach2}; but extending these computations to general weights would present considerable technical difficulties.
 
 However, as we shall show in the second paper in this series \cite{KLZ1b}, this can be circumvented via the use of Hida families. We can consider the triple $(f, g, \sigma)$ as a point in a 3-dimensional $p$-adic analytic space $\mathfrak{X}$, parametrising specialisations of two Hida families $\mathbf{f}$, $\mathbf{g}$ as well as a varying character of $\ZZ_p^\times$. Although Theorem \ref{thm:syntomicreg} does not apply directly at the point $(f, g, \sigma)$, it can be applied at a Zariski-dense set of other points in the $p$-adic parameter space $\mathfrak{X}$. Hence Theorem \ref{thm:syntomicreg} is in fact strong enough to completely determine all the values of the $p$-adic $L$-function, which is the key to the proof of the explicit reciprocity law which Theorem B of \cite{KLZ1b}. This can be used to prove Perrin-Riou's conjecture for $(f, g, \sigma)$ without the need for further geometric input; we hope to return to this issue in a future paper.
 
 (In contrast, for the leading term of the $L$-function at integer values $s \le 0$, both Perrin-Riou's conjecture and Beilinson's conjecture appear to be completely out of reach at present.)
\end{remark}

\providecommand{\bysame}{\leavevmode\hbox to3em{\hrulefill}\thinspace}

\renewcommand{\MR}[1]{%
 MR \href{http://www.ams.org/mathscinet-getitem?mr=#1}{#1}.
}
\newcommand{\articlehref}[2]{\href{#1}{#2}}

\end{document}